\documentclass[11pt]{amsart}
\usepackage{xcolor}
\usepackage{mathrsfs}
\usepackage{amsmath,amsthm,amssymb,amscd}
\usepackage{booktabs} 
\usepackage{lipsum}  
\usepackage{float}
\usepackage{tikz-cd}
\usepackage[hyperfootnotes=false, colorlinks, linkcolor={magenta}, citecolor={magenta}, filecolor={magenta}, urlcolor={magenta}, plainpages=false, pdfpagelabels]{hyperref}
\usepackage[overload]{textcase} 
\usepackage{mathtools}
\usepackage[numbers,sort&compress]{natbib} 
\usepackage{etoolbox} 
\apptocmd{\sloppy}{\hbadness 10000\relax}{}{} 
\usepackage[left=2.5cm, right=2.5cm, top=3cm]{geometry}
\allowdisplaybreaks
\usepackage{enumerate}
\usepackage[shortlabels]{enumitem}
\usepackage{bm}
\usepackage{url}
\usepackage[multiple]{footmisc}

\usepackage{colonequals}
\usepackage{longtable} 
\usepackage[ampersand]{easylist}
\usepackage{ltablex,booktabs}

\theoremstyle{plain}
\newtheorem{Theorem}{Theorem}[section]
\newtheorem{Corollary}[Theorem]{Corollary}
\newtheorem{Lemma}[Theorem]{Lemma}
\newtheorem{Proposition}[Theorem]{Proposition}

\theoremstyle{definition}
\newtheorem{Definition}[Theorem]{Definition}

\newtheorem{Remark}[Theorem]{Remark}
\newtheorem{Conjecture}[Theorem]{Conjecture}
\newtheorem{Question}[Theorem]{Question}

\newcommand{\thmref}[1]{Theorem~\ref{#1}}
\newcommand{\propref}[1]{Proposition~\ref{#1}}
\newcommand{\remref}[1]{Remark~\ref{#1}}
\newcommand{\lemref}[1]{Lemma~\ref{#1}}

\newcommand{\questref}[1]{Question~\ref{#1}}
\newcommand{\conjref}[1]{Conjecture~\ref{#1}}
\newcommand{\sectionref}[1]{Section~\ref{#1}}
\newcommand{\defref}[1]{Definition~\ref{#1}}
\newcommand{\corref}[1]{Corollary~\ref{#1}}

\newcommand{\hs}{\hspace{0.5cm}}
\newcommand{\neghs}{\mkern-12mu}

\DeclareMathOperator{\GL}{GL}
\DeclareMathOperator{\SL}{SL}
\DeclareMathOperator{\Gal}{Gal}

\DeclareMathOperator{\Frob}{Frob}

\DeclareMathOperator{\Aut}{Aut}

\DeclareMathOperator{\odd}{odd}

\newcommand{\N}{\mathbb{N}}
\newcommand{\Z}{\mathbb{Z}}
\newcommand{\Q}{\mathbb{Q}}
\newcommand{\R}{\mathbb{R}}

\newcommand{\F}{\mathbb{F}}

\newcommand{\fp}{\mathbf{p}}
\newcommand{\bigO}{\mathbf{O}}

\DeclareMathOperator{\lcm}{lcm}
\DeclareMathOperator{\rad}{rad}

\DeclareMathOperator{\Cyc}{cyc}

\DeclareMathOperator{\Prime}{prime}

\makeatletter
\renewcommand*{\p@equation}{}
\makeatother

\newcommand{\Eab}{\mathbb{E}^{a,b}}

\def\det{\operatorname{det}}
\def\tr{\operatorname{tr}}

\title[Opposing Average Congruence Class Biases in the Cyclicity and Koblitz Conjectures]{Opposing Average Congruence Class Biases in the Cyclicity and Koblitz Conjectures for Elliptic Curves} 
\author[Lee]{Sung Min Lee}
\address{Sung Min Lee, College of Arts and Sciences, Dakota State University, Madison, SD 57042}
\email{john.lee1@dsu.edu}
\author[Mayle]{Jacob Mayle}
\address{Jacob Mayle, Department of Mathematics, Wake Forest University, Winston-Salem, NC 27109}
\email{maylej@wfu.edu}
\author[Wang]{Tian Wang}
\address{Tian Wang, Max Planck Institute for Mathematics, Vivatsgasse 7, 53111 Bonn, Germany}
\email{twang@mpim-bonn.mpg.de}
\date{\today}

\begin{document}

\begin{abstract}

 The cyclicity and Koblitz conjectures ask about the distribution of primes of cyclic and prime-order reduction, respectively, for elliptic curves over $\mathbb{Q}$. In 1976, Serre gave a conditional proof of the cyclicity conjecture, but the Koblitz conjecture (refined by Zywina in 2011) remains open. The conjectures are now known unconditionally ``on average'' due to work of Banks--Shparlinski and Balog--Cojocaru--David. Recently, there has been a growing interest in the cyclicity conjecture for primes in arithmetic progressions (AP), with relevant work by Akbal--G\"ulo\u{g}lu and Wong. In this paper, we adapt Zywina's method to formulate the Koblitz conjecture for primes in AP and refine a theorem of Jones to establish results on the moments of the constants in both the cyclicity and Koblitz conjectures for AP. In doing so, we uncover a somewhat counterintuitive phenomenon: On average, these two constants are oppositely biased over congruence classes. Finally, in an accompanying repository, we give Magma code for computing the constants discussed in this paper.
\end{abstract}

\maketitle
\section{Introduction}
Let $E$ be an elliptic curve defined over the rationals. Let $N_E$ denote the conductor of $E$. For a prime $p$ not dividing $N_E$ (called a \emph{good prime} for $E$), we write $\widetilde{E}_p$ to denote the reduction of $E$ modulo $p$. The curve $\widetilde{E}_p$ is an elliptic curve over the finite field $\F_p$. Hence, the set of $\mathbb{F}_p$-points, denoted $\widetilde{E}_p(\F_p)$, forms a finite abelian group. It is known that
\[
\widetilde{E}_p(\F_p) \simeq \Z/d_p(E)\Z \oplus \Z/e_p(E) \Z
\quad \text{and} \quad
p + 1 - 2\sqrt{p} \leq |\widetilde{E}_p(\F_p)| \leq p + 1 + 2\sqrt{p}
\]
for some positive integers $d_p(E)$ and $e_p(E)$, which we take to satisfy $d_p(E) \mid e_p(E)$.

There has been considerable interest, dating back to the 1970s, in studying the distribution of primes $p$ for which  $\widetilde{E}_p(\F_p)$ has certain properties. In particular, one defines a good prime $p$ to be of \textit{cyclic reduction} for $E$ if $\widetilde{E}_p(\F_p)$ is a cyclic group and of \textit{Koblitz reduction} for $E$ if $|\widetilde{E}_p(\F_p)|$ is a prime. It is worth noting that every prime $p$ of Koblitz reduction is also of cyclic reduction since every group of prime order is cyclic. Let $\mathcal{X}$ be either ``$\Cyc$" or ``$\Prime$" and $\mathcal{X}_E(p)$ be either ``$p$ is of cyclic reduction'' or ``$p$ is of Koblitz reduction'' for $E$, respectively.
Define the counting function
\[
    \pi_E^{\mathcal{X}}(x) \coloneqq \#\{ p \leq x : p \nmid N_E \text{ and $\mathcal{X}_E(p)$ holds}\}.\\
\]
The problem of determining asymptotics for $\pi_E^{\mathcal{X}}(x)$ is called the \emph{cyclicity problem} or \emph{Koblitz's problem}, depending on the context. 

In recent years, there has been a growing interest in studying the two problems just described for primes lying in arithmetic progressions. To discuss this, fix integers $n,k$ with $n \geq 1$ and define
\[
\pi_E^{\mathcal{X}} (x;n,k) \coloneqq \#\{p \leq x : p\equiv k \neghs \pmod n, \, p \nmid N_E, \text{ and $\mathcal{X}_E(p)$ holds}\}. \]
Note that if $n$ and $k$ share a common factor, then $\pi_E^{\mathcal{X}} (x;n,k)$ is bounded as $x \to \infty$ for the trivial reason that there are only finitely many primes congruent to $k$ modulo $n$. As such, we will always take the integers $n$ and $k$ to be coprime. Broadly speaking, the goal of this paper is to examine the constants that appear in the main term of the asymptotics of $\pi_E^{\mathcal{X}}(x;n,k)$ and explore how they are influenced by the choice of $k$ modulo $n$. Before introducing our contributions, we outline aspects of the rich history of the cyclicity and Koblitz's problems relevant to our work.

We begin with the cyclicity problem, which has its origin in 1975 when I.\ Borosh, C.\ J.\ Moreno, and H.\ Porta \cite[pp.\ 962--963]{MR0404264} speculated that the density of primes of cyclic reduction exists and can be expressed as an Euler product.\footnote{In a related direction, S.\ Lang and H.\ Trotter \cite{MR427273} studied the density of primes $p$ for which the reduction of a given rational point $P$ on $E$ generates $\widetilde{E}_p(\mathbb{F}_p)$.} In 1976, J.-P.\ Serre \cite{MR3223094} observed that the cyclicity problem bears an alluring resemblance to Artin's primitive root conjecture, which was proven under the Generalized Riemann Hypothesis (GRH) by C.\ Hooley \cite{MR0207630} a decade prior. With this insight, Serre proposed the following conjecture, which he proved as a theorem under GRH.
\begin{Conjecture}[\protect{Cyclicity Conjecture \cite[pp.\ 465--468]{MR3223094}}]\label{Cyclicity} 
If $E/\Q$ is an elliptic curve, then 
\begin{equation}\label{cyclicasymptotic}
   \pi_E^{\Cyc}(x) \sim C^{\Cyc}_E \cdot \frac{x}{\log x}, 
\end{equation}
as $x \to \infty$, where $C^{\Cyc}_E \geq 0$ is the explicit constant defined in \eqref{CcycE}.
\end{Conjecture}
Serre noted that $C^{\Cyc}_E = 0$ if and only if $\Q(E[2]) = \Q$, in which case we interpret \eqref{cyclicasymptotic} as stating that $\pi_E^{\Cyc}(x)$ is bounded as $x \to \infty$. Numerous authors have made contributions to the cyclicity problem. See, for example, \cite{MR0698163, MR1055716, MR1932460, MR2570668, MR2534114, MR4177274, fredericks2021average, MR1975393, MR4623047} for some recent work on the problem.

In 2022, Y.\ Akbal and A.\ M.\ G\"ulo\u{g}lu \cite{MR4504664} studied the cyclicity problem for primes lying in an arithmetic progression. They proved that, under GRH,
\begin{equation}\label{cyclicasymptoticAP}
    \pi_E^{\Cyc}(x;n,k) \sim C^{\Cyc}_{E,n,k} \cdot \frac{x}{\log x},
\end{equation}
as $x \to \infty$, where $C^{\Cyc}_{E,n,k}$ is the explicit constant defined in \eqref{CcycEnk-def}. As before, if $C^{\Cyc}_{E,n,k} = 0$, then we interpret \eqref{cyclicasymptoticAP} as stating that $\pi_E^{\Cyc}(x;n,k)$ is bounded as $x \to \infty$. In 2015, J.\ Brau \cite{Brau} obtained a formula for the constant $C^{\Cyc}_{E,n,k}$ for all Serre curves outside of a small class (see \remref{Brauresults}). N.\ Jones and the first author determined all the possible scenarios in which the constant $C^{\Cyc}_{E,n,k}$ vanishes \cite{Jones-Lee}. Additionally, P.-J.\ Wong established \eqref{cyclicasymptoticAP} unconditionally for CM elliptic curves \cite{MR4765791}.

While Conjecture \ref{cyclicasymptotic} remains open without assuming GRH, researchers have found success in proving the conjecture is true ``on average'' in 
various senses. As observed in \cite[Remark 7(v)]{MR4041152}, there are two broad approaches regarding the average results. One approach is to compute the density of elliptic curves $E$ over $\F_p$ for which $E(\F_p)$ is cyclic, and average it over all primes $p$. Another approach is to count the number of primes for which an elliptic curve over $\Q$ has cyclic reduction and then average over the family of elliptic curves ordered by height. The former is called the ``local'' viewpoint while the latter is called the ``global'' viewpoint.

In 1999, 
S.\ G.\ Vl\u{a}du\c{t} \cite{MR1667099} obtained some statistics related to the cyclicity problem for elliptic curves over finite fields. In particular, he 
determined the ratio
\begin{equation}\label{localcyc}
    \frac{\#\{E \in \mathcal{F}_p : E(\F_p)\text{ is cyclic}\}}{\#\mathcal{F}_p},
\end{equation}
where $\mathcal{F}_p$ denotes the set of isomorphism classes of elliptic curves over $\F_p$. Later, E.-U.\ Gekeler \cite{MR2226271} built upon this result to obtain the local result for the average cyclicity problem; he computed that the average of \eqref{localcyc} over all primes is $C^{\Cyc}$, which is defined by \eqref{Ccyc}.

In 2009, building upon Vl\u{a}du\c{t}'s work, W.\ D.\ Banks and I.\ E.\ Shparlinski \cite{MR2570668} deduced a global result for the average cyclicity problem and demonstrated that it aligns with Gekeler's local result.  To set notation: For positive real numbers $A$ and $B$, let $\mathcal{F} \coloneqq \mathcal{F}(A,B)$ denote the family of elliptic curves $E/\Q$ defined by some short Weierstrass model 
\begin{equation}\label{basicmodel}
    E \colon Y^2 = X^3 + aX + b,
\end{equation}
with $a,b \in \mathbb{Z}$ satisfying $|a| \leq A$ and $|b|\leq B$. Banks and Shparlinski proved the following.
\begin{Theorem}[\protect{\cite[Theorem 18]{MR2570668}}]\label{BanksShparlinski}
    Let $x > 0, \epsilon >0$, and $K > 0$. Let $A \coloneqq A(x)$ and $B \coloneqq B(x)$ be parameters satisfying $x^{\epsilon} \leq A, B \leq x^{1-\epsilon},$ and $AB \geq x^{1+\epsilon}$. Then, we have
    $$\frac{1}{|\mathcal{F}|}\sum_{E\in \mathcal{F}} \pi_E^{\Cyc}(x) \sim C^{\Cyc} \cdot \frac{x}{\log x}, \hs \text{ as }  x\to \infty.$$
\end{Theorem}
Later, the inequality conditions on $A$ and $B$  were significantly relaxed by A.\ Akbary and A.\ T.\ Felix \cite[Corollary 1.5]{MR3337211}.

Building upon Banks and Shparlinski's methods, the first author refined the results to consider primes in arithmetic progressions  \cite[Theorem 1.3]{Lee}. To summarize his results, under the same assumptions of \thmref{BanksShparlinski}, for $n \leq \log x$ and $k$ coprime to $n$, there exists a positive constant $ C^{\Cyc}_{n,k}$ for which 
$$\frac{1}{|\mathcal{F}|} \sum_{E\in \mathcal{F}}\pi^{\Cyc}_E(x;n,k) \sim C^{\Cyc}_{n,k} \cdot \frac{x}{\log x}, \hs \text{ as } x\to \infty.$$
The average constants $C^{\Cyc}$ and $C^{\Cyc}_{n,k}$ are given explicitly in \eqref{Ccyc} and \eqref{Ccycnk}, respectively.

Related to the cyclicity problem is Koblitz's problem, which seeks to understand the asymptotics of $\pi_E^{\Prime}(x)$ and has significance for elliptic curve cryptography \cite{MR1846457,MR1945401,MR2459986}. In 1988, N.\ Koblitz \cite{MR0917870} made a conjecture analogous to Conjecture \ref{Cyclicity}. In particular, it follows from the conjecture that a non-CM elliptic curve $E/\Q$ has infinitely many primes of Koblitz reduction unless $E$ is rationally isogenous to an elliptic curve with nontrivial rational torsion. Koblitz's conjecture remained open for over 20 years until N.\ Jones gave a counterexample that appears in \cite[Section 1.1]{MR2805578}. The fundamental issue with the conjecture, which the counterexample exploits, is its failure to account for the possibility of entanglements of division fields. Properly accounting for this possibility, D.\ Zywina \cite{MR2805578} refined Koblitz's conjecture as follows.
\begin{Conjecture}[Refined Koblitz's Conjecture, \protect{\cite[Conjecture 1.2]{MR2805578}}]\label{Koblitz}
If $E/\Q$ is an elliptic curve, then
\begin{equation}\label{koblitzasymptotic}
   \pi_E^{\Prime}(x) \sim C_{E}^{\Prime} \cdot \frac{x}{(\log x)^2}, 
\end{equation}
as $x \to \infty$, where $C_{E}^{\Prime}\geq 0$ is the explicit constant defined in \eqref{Koblitzconstantoriginal}.
\end{Conjecture}
Similar to the cyclicity case, the constant $C^{\Prime}_E$ may vanish. In this case, we interpret \eqref{koblitzasymptotic} as indicating that $\pi_E^{\Prime}(x)$ is bounded as $x\to \infty$. Beyond the statement of the conjecture provided above, Zywina made the conjecture more generally for elliptic curves over number fields and allowed for a parameter $t$ to consider primes $p$ for which $|\widetilde{E}_p(\F_p)|/t$ is prime. 

Unlike the cyclicity problem, \conjref{Koblitz} remains open even under GRH. However, in 2011, A.\ Balog, A.\ C.\ Cojocaru, and C.\ David not only obtained a local result for the average version of Koblitz's problem but also applied it to deduce the following global result.
\begin{Theorem}[\protect{\cite[Theorem 1]{MR2843097}}]\label{BCD}
    Set $x > 0$ and $\epsilon > 0$. Let $A \coloneqq A(x)$ and $B \coloneqq B(x)$ be parameters satisfying $x^{\epsilon} < A,B$ and $AB > x\log^{10}x$. There exists a constant $C^{\Prime} > 0$ for which
    $$\frac{1}{|\mathcal{F}|} \sum_{E\in \mathcal{F}} \pi_E^{\Prime}(x) \sim C^{\Prime} \cdot \frac{x}{\log^2 x}, \hs \text{ as } x\to \infty.$$
\end{Theorem}
The average constant $C^{\Prime}$ is defined in \eqref{Cprime}. The inequality conditions on $A$ and $B$ can also be relaxed as in A.\ Akbary and A.\ T.\ Felix \cite[Equation (1.8)]{MR3337211}. 

A natural inquiry is whether each of these average results is consistent with the corresponding conjectured outcomes on average. This question was answered by N.\ Jones \cite{MR2534114} assuming an affirmative answer to Serre's uniformity question (\questref{SUQ}).
\begin{Theorem}[\protect{\cite[Theorem 6]{MR2534114}}]\label{Jones}
    Assume an affirmative answer to Serre's uniformity question. Let $\mathcal{X} \in \{\Cyc, \Prime\}$. There exists an exponent $\gamma > 0$ such that for any positive integer $t$, we have
    $$\frac{1}{|\mathcal{F}|} \sum_{E \in \mathcal{F}} \left|C_E^{\mathcal{X}}-C^{\mathcal{X}}\right|^t \ll_t \max \left\{\left(\frac{\log B \cdot \log^7 A}{B}\right)^{t/t+1}, \frac{\log^{\gamma}(\min\{A,B\})}{\sqrt{\min\{A,B\}}} \right\}, $$
    as $\min\{A,B\} \to \infty$.
\end{Theorem}
\thmref{Jones} has a corollary as follows. Suppose that $A \coloneqq A(x)$ and $B \coloneqq B(x)$ tend to infinity as $x \to \infty$. Assume an affirmative answer to Serre's uniformity question and that $(\log B \log^7 A)/B \to 0$ as $x \to \infty$. Then for $\mathcal{X} \in \{\Cyc, \Prime\}$, we have that
$$\frac{1}{|\mathcal{F}|}\sum_{E\in\mathcal{F}}C_E^{\mathcal{X}}  \to C^{\mathcal{X}}.$$

In this paper, we utilize Zywina's approach to propose Koblitz's constant $C^{\Prime}_{E,n,k}$ for primes in arithmetic progressions. Unlike the cyclicity problem, the average version of Koblitz's constant $C^{\Prime}_{E,n,k}$ has not yet been considered. We address this gap in the literature by providing a candidate for $C^{\Prime}_{n,k}$, the average version of $C^{\Prime}_{E,n,k}$, in Equation \eqref{Cprimenk}. We illustrate the suitability of these conjectural constants by proving an analogous version of \thmref{Jones} for them.

We start by formulating Koblitz's conjecture for primes in arithmetic progressions. 

\begin{Conjecture}\label{KoblitzConjAP}
    If $E/\Q$ is an elliptic curve, then there exists $C^{\Prime}_{E,n,k} \geq 0$ for which
    $$\pi_E^{\Prime}(x;n,k) \sim C^{\Prime}_{E,n,k} \cdot \frac{x}{\log^2 x},$$
    as $x \to \infty$, where $C^{\Prime}_{E,n,k}$ is the explicit constant defined in \eqref{KoblitzconstantAPoriginal}.
\end{Conjecture}
As before, if $C^{\Prime}_{E,n,k} = 0$, we interpret the above as saying that $\pi^{\Prime}_{E}(x;n,k)$ is bounded as $x \to \infty$.
As one piece of evidence to show $ C^{\Cyc}_{n,k}$ and $C^{\Prime}_{n,k}$ are the correct average constants,  we  compare  them  with the constants $ C^{\Cyc}_{E,n,k}$ and  $C^{\Prime}_{E, n,k}$, respectively, for Serre curves which, by Jones \cite{MR2563740}, make up 100\% of elliptic curves in the sense of density, when ordered by naive height.

To state our theorem, we first introduce some notation. Associated to $E$, we define the constant
\begin{equation} \label{E:Def-L}
L = \prod_{\ell \mid m_E} \ell^{\alpha_\ell}, \hs \text{ where } \hs \alpha_\ell = \begin{cases}
        v_\ell(n) & \text{ if } \ell \mid n, \\
        1 & \text{ otherwise},
    \end{cases}
\end{equation}
where $v_\ell(n)$ denotes the $\ell$-adic valuation of $n$ and $m_E$ denotes the adelic level of $E$ (defined in Sections \ref{GalRepAdLev} and \ref{GalRepCM}). The constants $m_E$ and $L$ play a crucial role in computing $C^{\Cyc}_{E,n,k}$ and $C^{\Prime}_{E,n,k}$. For a Serre curve $E$, Proposition \ref{mEofSerreCurves} gives a straightforward formula for $m_E$,
$$
    m_E = \begin{cases}
        2|\Delta'| & \text{ if } \Delta' \equiv 1 \pmod 4, \\
        4|\Delta'| & \text{ otherwise},
    \end{cases}
$$
where $\Delta'$ denotes the squarefree part of the discriminant $\Delta_E$ of any Weierstrass model of $E$. 

\begin{Theorem}\label{serre-constants}
    Let $E/\Q$ be a Serre curve and let $m_E$, $\Delta'$, and $L$ be as above. If $m_E \nmid L,$ then
    $$C^{\Cyc}_{E,n,k} = C^{\Cyc}_{n,k} \hs \text{ and } \hs C^{\Prime}_{E,n,k} = C^{\Prime}_{n,k}.$$
    Otherwise, if $m_E \mid L$, then
    \begin{align*}
        C^{\Cyc}_{E,n,k} &= C^{\Cyc}_{n,k} \left( 1 + \tau^{\Cyc} \frac{1}{5} \prod_{\substack{\ell \mid L \\ \ell \nmid 2n}} \frac{1}{\ell^4-\ell^3-\ell^2+\ell-1} \right), \\
        C^{\Prime}_{E,n,k} &= C^{\Prime}_{n,k} \left( 1 + \tau^{\Prime}\prod_{\substack{\ell \mid L \\\ell \nmid 2n}} \frac{1}{\ell^3-2 \ell^2-\ell+3}\right),
    \end{align*}
    where $\tau^{\Cyc},\tau^{\Prime} \in \{\pm 1\}$ are defined in \defref{Definitionoftau}.
\end{Theorem} 

\begin{Remark}\label{Brauresults}
    For a Serre curve $E$, the constant $C^{\Cyc}_{E,n,k}$ was previously obtained by J.\ Brau \cite[Proposition 2.5.8]{Brau} under the assumption that $\Delta' \not\in \{-2,-1,2\}$. Our formula for $C^{\Cyc}_{E,n,k}$ does not require this assumption and it aligns with Brau's. 
\end{Remark}

As another piece of evidence, we also consider the moments of the constants $C^{\Cyc}_{E, n, k}$ and $C^{\Prime}_{E, n, k}$ for $E\in  \mathcal{F}$. 
Building upon Jones's methods, we improve \thmref{Jones} unconditionally as follows.
\begin{Theorem}\label{Main}
    Let $n$ be a positive integer and $k$ be coprime to $n$. Then there exists an exponent $\gamma > 0$ such that for any positive integer $t$, we have
    \begin{align*}
        \frac{1}{|\mathcal{F}|} \sum_{E\in \mathcal{F}} \left|C^{\mathcal{\Cyc}}_{E,n,k} - C^{\Cyc}_{n,k} \right|^t &\ll_t \max \left\{ \left( \frac{n \log B \log^7 A}{B}\right)^{\frac{3t}{3t+1}}, \frac{\log^\gamma (\min\{A,B\})}{\sqrt{\min\{A,B\}}} \right\}, \\
        \frac{1}{|\mathcal{F}|} \sum_{E\in \mathcal{F}} \left|C^{\Prime}_{E,n,k} - C^{\Prime}_{n,k}\right|^t &\ll_{n,t} \max \left\{ \left( \frac{n \log B \log^7 A}{B}\right)^{\frac{2t}{2t+1}}, \left(\log \log (\max\{A^3,B^2\})\right)^t \frac{\log^\gamma (\min\{A,B\})}{\sqrt{\min\{A,B\}}} \right\},
    \end{align*}
    as $\min\{A,B\} \to \infty$.
\end{Theorem}
Observe that as $\min \{A,B\} \to \infty$, we have 
$$\frac{\log^\gamma (\min\{A,B\})}{\sqrt{\min\{A,B\}}} \to 0.$$
This gives us the following corollary.
\begin{Corollary}\label{cor-Main}
 Fix $n\in \N$. Let $A \coloneqq A(x)$ and $B \coloneqq B(x)$ both tend to infinity as $x \to \infty$. With the same notation as in \thmref{Main}, we have that
    \[
     \frac{1}{|\mathcal{F}|}\sum_{E \in \mathcal{F}} C^{\mathcal{\mathcal{X}}}_{E,n,k} \to C^{\mathcal{\mathcal{X}}}_{n,k}
    \]
     provided that as $x\to \infty$, 
     $$ \left( \frac{n \log B \log^7 A}{B}\right) \to 0$$ in the cyclicity case
    and 
    $$\left( \frac{n \log B \log^7 A}{B}\right) \to 0, \quad   \frac{\left(\log \log (\max\{A^3,B^2\})\right)^t \log^\gamma (\min\{A,B\})}{\sqrt{\min\{A,B\}}} \to 0$$
    in the Koblitz case.
\end{Corollary}

Based on the above considerations, the constant $C^{\Prime}_{n,k}$ that we propose in this paper 
appears to be a plausible candidate for the average counterpart of $C^{\Prime}_{E,n,k}$ in the sense of Balog, Cojocaru, and David \cite{MR2843097}.

As $C^{\Cyc}_{n,k}$ and $C^{\Prime}_{n,k}$ are given explicitly, we may compute their values (to any given precision) using the Magma \cite{magma} scripts available in this paper's GitHub repository \cite{LMW-GitHub}. Below are tables with the values of $C^{\mathcal{X}}_{n,k}$ for small moduli $n$.
\begin{table}[H]
\begin{tabular}{|l|l|l|l|l|l|}
\hline
$n$ \textbackslash{} $k$ & $1$       & $2$ & $3$ & $4$ & $5$ \\ \hline
$2$ & $0.813752$ & $-$ & $-$ & $-$ & $-$ \\ \hline
$3$ & $0.398219$ & $0.415533$ & $-$ & $-$ & $-$ \\ \hline
$4$ & $0.406876$ & $-$ & $0.406876$ & $-$ & $-$ \\ \hline
$5$ & $0.202164$ & $0.203863$ & $0.203863$ & $0.203863$ & $-$ \\ \hline
$6$ & $0.398219$ & $-$ &  $-$  & $-$ & $0.415533$   \\ \hline
\end{tabular}
\caption{The value of $C_{n,k}^{\Cyc}$ to six decimal places.}
\end{table}

\begin{table}[H]
\begin{tabular}{|l|l|l|l|l|l|}
\hline
$n$ \textbackslash{} $k$ & $1$       & $2$ & $3$ & $4$ & $5$ \\ \hline
$2$ & $0.505166$ & $-$ & $-$ & $-$ & $-$ \\ \hline
$3$ & $0.280648$ & $0.224518$ & $-$ & $-$ & $-$ \\ \hline
$4$ & $0.252583$ & $-$ & $0.252583$ & $-$ & $-$ \\ \hline
$5$ & $0.131482$ & $0.124562$ & $0.124562$ & $0.124562$ & $-$ \\ \hline
$6$ & $0.280648$ & $-$ &  $-$  & $-$ & $0.224518$   \\ \hline
\end{tabular}
\caption{The value of $C_{n,k}^{\Prime}$ to six decimal places.}
\end{table}

From the table, we observe that $C^{\mathcal{X}}_{2,1}=C^{\mathcal{X}}$. Moreover,  in each table, the sum of the values across any given row yields $C^{\mathcal{X}}$. 
We prove that, indeed, these sanity checks hold for general moduli in \propref{Avgsumcyc} and \propref{AvgSum}.

Let $p$ be a good prime for $E$. As noted previously,
$$|\widetilde{E}_p(\F_p)| \text{ is prime } \implies \widetilde{E}_p(\F_p) \text{ is cyclic}.$$
Hence, for arbitrary $E/\Q$, one might suspect that if primes in a certain congruence class are more likely to be primes of Koblitz reduction, then they are also more likely to be primes of cyclic reduction. However, the tables above suggest that the contrary holds on average. Indeed, it follows from the formulas \eqref{Ccycnk} and \eqref{Cprimenk} for $C^{\mathcal{X}}_{n,k}$ that these two average constants are oppositely biased for any given modulus $n$. More specifically, for any $k$ coprime to $n$, we have
\begin{equation*}
    C^{\Cyc}_{n,1} \leq C^{\Cyc}_{n,k} \leq C^{\Cyc}_{n,-1} \hs \text{ while } \hs C^{\Prime}_{n,1} \geq C^{\Prime}_{n,k} \geq C^{\Prime}_{n,-1}.
\end{equation*}
Furthermore, the inequalities are strict if and only if $n$ is not a power of $2$. The phenomenon of primes being statistically biased over congruence classes is referred to as the average congruence class bias and was first observed in the cyclicity problem by the first author in \cite{Lee}.

Lastly, it is notable that in both tables, 
$C^{\mathcal{X}}_{5,2} = C^{\mathcal{X}}_{5,3} = C^{\mathcal{X}}_{5,4}$. This is because, for a fixed $n$, the value of $C^{\mathcal{X}}_{n,k}$ depends solely on whether $k$ is congruent to $1$ or not modulo each prime factor of $n$. Therefore, for a fixed modulus $n$ that is supported by $s$ distinct odd primes, there are at most $2^s$ distinct values of $C^{\mathcal{X}}_{n,k}$. Whether there are exactly $2^s$ distinct values is a question proposed by the first author in \cite{Lee}.

\subsection{Outline of the paper} \label{Outline}
Sections \ref{Preliminaries} and \ref{CountMat} provide the essential groundwork for proving the main results. In Section \ref{Preliminaries}, we  introduce the properties of Galois representations of elliptic curves. In particular, we introduce the definition of the adelic level and characterize the Galois images of Serre curves and CM curves. In Section \ref{CountMat}, we count certain subsets of matrix groups that  will be used in calculating the Euler factors of product expansions of $C_{E, n, k}^{\Cyc}$ and $C_{E, n, k}^{\Prime}$.

Sections \ref{DefCon} and \ref{CycKobConSeCur} are dedicated to the computation of the constants $C_{E, n, k}^{\mathcal{X}}$. These computations extend Zywina's approach (a method that originates from Lang and Trotter's work \cite{MR0568299} on the Lang-Trotter conjecture) to obtain $C_{E}^{\Prime}$.    The general idea is to interpret the conditions for primes of Koblitz reduction for $E$ in terms of mod $m$ Galois representations, establish the heuristic constant at each level $m$, and then take the limit as $m\to \infty$. In Section \ref{DefCon}, we apply this idea to reformulate the constants $C_E^{\Cyc}$ and $C_{E, n, k}^{\Cyc}$ and express $C_{E, n, k}^{\Prime}$ in the form of an almost Euler product.  
We also propose the average constant $C_{n, k}^{\Prime}$ as a complete Euler product. In Section \ref{CycKobConSeCur}, we examine the special case where $E$ is a Serre curve, proving Theorem \ref{serre-constants} which gives explicit formulas for $C_{E, n, k}^{\Cyc}$ and $C_{E, n, k}^{\Prime}$ in this case. A critical aspect of these computations involves extracting as many Euler factors as possible from the limits (\ref{KoblitzconstantAPoriginal}) and (\ref{naiveheuden}), leading to the crucial definition of $L$ in (\ref{E:Def-L}).

Sections \ref{KobConNonSeCur} and \ref{Moments} address bounding the moments of $C_{E, n, k}^{\Cyc}$ and $C_{E, n, k}^{\Prime}$ for $E\in \mathcal{F}$. In Section \ref{KobConNonSeCur}, we build on the work carried out in Section \ref{CycKobConSeCur} to bound $C^{\Prime}_{E,n,k}$ for non-Serre, non-CM curves and CM curves. Using a result due to W. Masser and G. W\"ustholz \cite{MR1209248}, we bound $C^{\Prime}_{E,n,k}$ for non-Serre, non-CM curves in terms of the naive height of $E$. This approach allows us to avoid assuming an affirmative answer to Serre's uniformity question, in contrast to Jones. For CM elliptic curves, we first derive the conjectural constant $C_{E, n, k}^{\Prime}$ using a similar method to that of Section \ref{DefCon} and Section \ref{CycKobConSeCur} and bound it directly from its formula. In Section \ref{Moments}, we adapt the method of Jones \cite{MR2534114} to complete the moments computations and prove \thmref{Main}.

Finally, in Section \ref{NumEx}, we provide numerical examples that support our results. The numerical examples are computed using the Magma code available in this paper's GitHub repository \cite{LMW-GitHub}:

\centerline{\url{https://github.com/maylejacobj/CyclicityKoblitzAPs}}

\noindent We now summarize the main functions of the repository. The functions \texttt{AvgCyclicityAP} and \texttt{AvgKoblitzAP} allow one to compute  $C^{\Cyc}_{n,k}$ and $C^{\Prime}_{n,k}$ for given coprime integers $n$ and $k$, and were used to produce the tables above. Next, the functions \texttt{CyclicityAP} and \texttt{KoblitzAP} allow one to compute the constants $C_{E, n, k}^{\Cyc}$ and $C_{E, n, k}^{\Prime}$ for any given non-CM elliptic curve $E$. These functions are based on Proposition \ref{cycalmosteuler} and Proposition \ref{KoblitzAPeulerproduct} and rely crucially on Zywina’s \texttt{FindOpenImage} function \cite{GitHubZywina} to compute the adelic image of $E$. The functions \texttt{SerreCurveCyclicityAP} and \texttt{SerreCurveKoblitzAP} compute $C_{E, n, k}^{\Cyc}$ and $C_{E, n, k}^{\Prime}$ for a given Serre curve $E$ using Theorem \ref{serre-constants} and do not require Zywina's \texttt{FindOpenImage}. Lastly, the repository contains code for the examples in Section \ref{NumEx}.

\subsection{Notation and conventions} \label{Notation} We now give a brief overview of the notation used throughout the paper.  
\begin{itemize}
    \item For functions $f,g \colon \mathbb{R} \to \mathbb{R}$, we write $f \ll g$ or $f = \bigO(g)$ if there exists $C > 0$ and $x_0 \geq 0$ such that $|f(x)| \leq Cg(x)$ for all $x > x_0$. If $C$ depends on a parameter $m$, we write $f \ll_m g$ or $f = \bigO_m(g)$. 
    \item In the same setting as above, we write $f \sim g$ to denote that $\lim_{x\to \infty} f(x)/g(x) = 1$.
    \item Let $A$ and $B$ be positive real numbers. Let $\mathcal{F} \coloneqq \mathcal{F}(A,B)$ denote the family of models $Y^2 = X^3 + aX+b$ of elliptic curves for which $|a| \leq A$ and $|b| \leq B$.
    \item Given a subfamily $\mathcal{G} \subseteq \mathcal{F}$ of elliptic curves, let $f$ and $g$ be functions defined from $\mathcal{G}$ to $\R$. We write $f \ll g$ if there exists an absolute constant $M > 0$ for which $|f(E)| \leq Mg(E)$ for all $E \in \mathcal{G}$. When $M$ depends on a parameter $m$, we write $f \ll_m g$.
    \item $p$ and $\ell$ denote rational primes, $n$ a positive integer, and $k$ an integer coprime to $n$.
    \item We write $p^a \parallel n$ if $p^a \mid n$ and $p^{a+1} \nmid n$. In this case, $a$ is called the $p$-adic valuation of $n$, and is denoted by $v_p(n)$.
    \item Given a positive integer $n$, $n^{\odd}$ denotes the odd part of $n$, i.e., $n^{\odd} = n/2^{v_2(n)}$.
    \item We sometimes write $(m,n)$ as shorthand for $\gcd(m,n)$.
    \item $m^\infty$ denotes an arbitrarily large power of $m$. Thus, $\gcd(n,m^\infty)$ denotes $\prod_{p \mid (n,m)} p^{v_p(n)}$. If every prime factor of $n$ divides $m$, then we write $n \mid m^\infty$.
    \item $\left(\frac{\cdot}{d}\right)$ denotes the Jacobi symbol.
    \item $\phi$ denotes the Euler totient function.
    \item $\mu$ denotes the M\"obius function.
    \item $G(m)$ denotes the image of a subgroup $G$ of $\GL_2(\widehat{\Z})$ under the reduction modulo $m$ map.
    \item Given that $d \mid m$ and $M \in \GL_2(\Z/m\Z)$, $M_d$ denotes the reduction of $M$ modulo $d$.
    \item If $\mathcal{A}$ is the empty set, then we take $\prod_{a\in \mathcal{A}} a$ to be $1$.
    \end{itemize} 

\subsection{Acknowledgments} This paper emerged from some initial conversations at the 2023 LuCaNT (LMFDB, Computation, and Number Theory) conference held at ICERM (Institute for Computational and Experimental Research in Mathematics). We are grateful to the conference organizers and the organizations that provided funding.  An earlier version of this manuscript appears in the first author's doctoral thesis. We are thankful for the doctoral committee members for their helpful comments. The third author, who conducted most of the work at the Max Planck Institute for Mathematics, is grateful for its funding and stimulating atmosphere of research. 

\section{Preliminaries}\label{Preliminaries}
\subsection{Galois representations and the adelic level} \label{GalRepAdLev}

Let $E/\Q$ be an elliptic curve. Associated to $E$, we consider the \emph{adelic Tate module}, which is given by the inverse limit
\[ T(E) \coloneqq \varprojlim E[n] \]
where $E[n]$ denotes the $n$-torsion subgroup of $E(\overline{\Q})$. Let $\widehat{\Z}$ denote the ring of profinite integers. It is well-known that $T(E)$ is a free $\widehat{\Z}$-module of rank $2$. The absolute Galois group $\Gal(\overline{\Q}/\Q)$ acts naturally on $T(E)$, giving rise to the \emph{adelic Galois representation} of $E$,
\[ \rho_E \colon \Gal(\overline{\Q}/\Q) \longrightarrow \Aut(T(E)). \]
Upon fixing a $\widehat{\Z}$-basis for $T(E)$, we  consider $\rho_E$ as a map
\[ \rho_E \colon \Gal(\overline{\Q}/\Q) \longrightarrow \GL_2(\widehat{\Z}). \]
Let $G_E$ denote the image of $\rho_E$, which, because of the above choice of basis, is defined only up to conjugacy in $\GL_2(\widehat{\Z})$. With respect to the profinite topology on $\GL_2(\widehat{\Z})$, the subgroup $G_E$ is necessarily closed since $\rho_E$ is a continuous map.

We now state a foundational result of Serre, known as Serre's open image theorem.

\begin{Theorem}[Serre, \protect{\cite[Th\'eor\`eme 3]{MR0387283}}] \label{SOIT} If $E/\Q$ is without complex multiplication, then $G_E$ is an open subgroup of $\GL_2(\widehat{\Z})$. In particular, the index $[\GL_2(\widehat{\Z}) : G_E]$ is finite.
\end{Theorem}
Suppose that $E/\Q$ is non-CM. For each positive integer $m$, let $\pi_m$ be the natural reduction map
\[
\pi_m \colon \GL_2(\widehat{\Z}) \longrightarrow \GL_2(\Z/m\Z).
\]
Let $G_E(m)$ be the image of the mod $m$ Galois representation
\[ \rho_{E,m} \colon \Gal(\overline{\Q}/\Q) \to \GL_2(\Z/m\Z),\]
defined by the composition $\pi_m\circ \rho_E$. It follows from Theorem \ref{SOIT} that there exists a positive integer $m$ for which
\begin{equation}\label{preimagenonCM1}
         G_E = \pi_m^{-1}(G_E(m)). 
\end{equation}
One may observe that \eqref{preimagenonCM1} is equivalent to the statement that for every $n \in \N$,
\begin{equation}\label{preimagenonCM2}
      G_E(n) = \pi^{-1}(G_E(\gcd(n,m)))
\end{equation}
where $\pi\colon \GL_2(\Z/n\Z) \to \GL_2(\Z/\gcd(n,m)\Z)$ denotes the natural reduction map. The least positive integer $m$ with this property is called the \emph{adelic level} of $E$, and is denoted by $m_E$. The constant $m_E$ measures both the nonsurjectivity of the $\ell$-adic Galois representations of $E$ as well as the entanglements between their images.

We now give a fundamental property of $m_E$ that we will use several times.
\begin{Lemma}\label{propertyofmE}
    Let $E/\Q$ be a non-CM elliptic curve of adelic level $m_E$. For any $d_1,d_2 \in \N$ with $d_1 \mid m_E^\infty$ and $(d_2, m_E) = 1$, we have
    $$G_E(d_1d_2) \simeq G_E(d_1) \times \GL_2(\Z/d_2\Z)$$
    via the map $\GL_2(\Z/d_1d_2) \to \GL_2(\Z/d_1) \times \GL_2(\Z/d_2)$.
\end{Lemma}
\begin{proof}
    By the given condition, we have $(d_1,d_2) = 1$. Set $d' = \gcd(d_1,m_E)$. Let $\pi \colon \GL_2(\Z/d_1d_2\Z) \to \GL_2(\Z/d'\Z)$ and $\pi_1 \colon \GL_2(\Z/d_1\Z) \to \GL_2(\Z/d'\Z)$ be the natural reduction maps. By the Chinese remainder theorem, $\pi$ can be identified with
    $$\pi_1 \times \text{triv} \colon \GL_2(\Z/d_1\Z) \times \GL_2(\Z/d_2\Z) \to \GL_2(\Z/d'\Z) \times \{1\}.$$
    By \eqref{preimagenonCM2}, we have that
    \[G_E(d_1d_2) = \pi^{-1}(G_E(d')) \simeq (\pi_1 \times \text{triv})^{-1}(G_E(d')) = G_E(d_1) \times \GL_2(\Z/d_2\Z). \qedhere \]
\end{proof}

We conclude this subsection by recalling Serre's uniformity question.
\begin{Question}\label{SUQ}
    Does there exist an absolute constant $c$ such that for each elliptic curve $E/\Q$,
    $$G_E(\ell) = \GL_2(\Z/\ell\Z)$$
    holds for all rational primes $\ell > c$?
\end{Question}

While Question \ref{SUQ} remains open, it is widely conjectured to be true with $\ell = 37$ \cite{Su2016,Zy2015a} and considerable partial progress has been made toward its resolution \cite{MR3885140,MR0387283,MR0644559,Ma1978,MR3961086,furio2023serresuniformityquestionproper}.

\subsection{Serre curves}\label{SeCur}
In this subsection, we introduce the generic class of elliptic curves $E/\mathbb{Q}$ with maximal adelic Galois image $G_E$,  and provide an explicit  description of $G_E$ for curves in this class.

Serre noted \cite{MR0387283} that for an elliptic curve $E/\Q$, the adelic Galois representation $\rho_E$ cannot be surjective, that is, the adelic level $m_E$ is never $1$. We briefly give the argument here. If $E$ has complex multiplication, then $[\GL_2(\widehat{\Z}):G_E]$ is necessarily infinite \cite{MR0387283}, so we restrict our attention to the case that $E$ is non-CM. Assume that $E$ is defined by the factored Weierstrass equation
$$Y^2 = (X-e_1)(X-e_2)(X-e_3)$$
with $e_1,e_2,e_3 \in \overline{\Q}$. Then, the $2$-torsion of $E$ is given by
\[E[2] = \left\{\mathcal{O}, (e_1,0), (e_2,0), (e_3,0)\right\} \cong \Z/2\Z \oplus \Z/2\Z.\] Consequently, $\Aut(E[2])$ can be identified with $S_3$. The discriminant $\Delta_E$ of $E$ is given by
\begin{equation}\label{definitionofDeltaE}
    \Delta_E = \left[(e_1-e_2) (e_2-e_3)(e_3-e_1) \right]^2.
\end{equation}
 Let $\Delta'$ denote the squarefree part of $\Delta_E$, i.e., the unique squarefree integer such that $\Delta_E/\Delta' \in (\Q^\times)^2$. The discriminant $\Delta_E$ depends on the Weierstrass model of $E$, but $\Delta'$ does not.
 
Let us first assume that $\Delta_E \not \in (\Q^\times)^2$. Let $d_E$ be the conductor of $\Q(\sqrt{\Delta_E})$, that is, the smallest positive integer such that $\Q(\sqrt{\Delta_E}) \subseteq \Q(\zeta_{d_E})$. One may easily check that
$$ d_E = \begin{cases}
    |\Delta'| & \text{ if } \Delta' \equiv 1 \pmod 4,\\
    4|\Delta'| & \text{ otherwise}.
\end{cases}$$
Let us define the quadratic character associated to $\Q(\sqrt{\Delta_E})$ as follows,
$$\chi_{\Delta_E} \colon \Gal(\overline{\Q}/\Q) \xrightarrow{\text{rest}.} \Gal(\Q(\sqrt{\Delta_E})/\Q) \overset{\sim}{\longrightarrow} \{\pm 1\}.$$
Fix $\sigma \in \Gal(\overline{\Q}/\Q)$. Viewing $\rho_{E,2}(\sigma) \in G_E(2) \subseteq \Aut(E[2]) \simeq S_3$,  by \eqref{definitionofDeltaE}, we notice that

$$\chi_{\Delta_E}(\sigma) \left(\sqrt{\Delta_E}\right) = \epsilon(\rho_{E, 2}(\sigma))\left(\sqrt{\Delta_E}\right),$$
where $\epsilon \colon S_3 \to \{\pm 1\}$ denotes the signature map.\footnote{Note that the value of $\epsilon(\rho_{E, 2}(\sigma))$ is independent of the choice of isomorphism $\Aut(E[2]) \simeq S_3$.} Hence, $\chi_{\Delta_E}(\sigma) = \epsilon(\rho_{E,2}(\sigma))$.

On the other hand, we have that $\Q(\sqrt{\Delta_E}) \subseteq \Q(\zeta_{d_E})$. Since $\Gal(\Q(\zeta_{d_E})/\Q) \simeq (\Z/d_E\Z)^\times$, there exists a unique quadratic character $\alpha \colon \Gal(\Q(\zeta_{d_E})/\Q) \to \{\pm 1\}$ for which $\chi_{\Delta_E}(\sigma) = \alpha(\det \circ \rho_{E,{d_E}}(\sigma))$ for any $\sigma \in \Gal(\overline{\Q}/\Q)$. Therefore, we have
\begin{equation} \label{E:quadch}\epsilon(\rho_{E,2}(\sigma)) = \alpha(\det \circ \rho_{E,d_E}(\sigma)) \end{equation}
for any $\sigma \in \Gal(\overline{\Q}/\Q)$.

Let $M_E \coloneqq \lcm(2,d_E)$. Consider the subgroup
$$H_E(M_E) = \left\{ M \in \GL_2(\Z/M_E\Z) : \epsilon(M_2) = \alpha(\det M_{d_E})\right\},$$
where $M_2$ and $M_{d_E}$ denote the reductions of $M$ modulo $2$ and $d_E$, respectively. Note that the index of $H_E(M_E)$ in $\GL_2(\Z/M_E\Z)$ is $2$ and that $G_E(M_E) \subseteq H_E(M_E)$ by \eqref{E:quadch}. Let $H_E$ be the preimage of $H_E(M_E)$ under the reduction map $\GL_2(\widehat{\Z}) \to \GL_2(\Z/M_E\Z)$. Then $H_E$ is an index $2$ subgroup of $\GL_2(\widehat{\Z})$ that contains $G_E$. We say that $E$ is a \emph{Serre curve} if $H_E = G_E$, that is, $[\GL_2(\widehat{\Z}):G_E] = 2$.

In the above discussion, we supposed that $\Delta_E \not\in (\Q^\times)^2$. We now consider the opposite case that $\Delta_E \in (\Q^\times)^2$. Observe that $[\Q(E[2]):\Q]$ divides $3$, and hence $[\GL_2(\Z/2\Z) : G_E(2)]$ is divisible by $2$. Thus, by \cite[Proposition 2.14]{MayleRakvi},  $[\GL_2(\widehat{\Z}) : G_E] \geq 12$, which follows by considering the index of the commutator of $G_E$ in $\SL_2(\widehat{\Z})$. In particular, $E$ cannot be a Serre curve in this case.

Serre curves are useful for us  for two key reasons. First, as mentioned in the introduction, Jones \cite{MR2563740} showed that they are ``generic'' in the sense that the density of the subfamily of Serre curves among the family of all elliptic curves ordered by naive height is $1$. Second, the adelic image $G_E$ of a Serre curve $E$ can be explicitly described, as we will now discuss.

\begin{Proposition}\label{mEofSerreCurves}
    Let $E/\Q$ be a Serre curve and write $\Delta'$ to denote the squarefree part of the discriminant of $E$. Then 
    \begin{equation}\label{mEforSerreCurves}
        m_E = \begin{cases}
        2|\Delta'| & \text{ if } \Delta' \equiv 1 \pmod 4,\\
        4|\Delta'| & \text{ otherwise}.
    \end{cases}
    \end{equation}
    Furthermore, for any positive integer $m$,
       $$G_E(m) = \begin{cases}
        \GL_2(\Z/m\Z) & \text{ if }  m_E \nmid m, \\
        H_E(m) & \text{ if } m_E \mid m,
    \end{cases}$$
    where $H_E(m)$ denotes the image of $H_E$ under the reduction modulo $m$ map.
\end{Proposition}
\begin{proof}
The proof of \eqref{mEforSerreCurves} can be found in \cite[pp.\ 696-697]{MR2534114}. Observe that $m_E = M_E$ where $M_E$ is defined as above. Now, let $m$ be a positive integer. By \cite[Equation (13)]{MR2534114} and \eqref{preimagenonCM2}, one may deduce that $G_E(m) = \GL_2(\Z/m\Z)$ if $m_E \nmid m$. Suppose $m_E \mid m$. Then, $G_E(m) \subseteq H_E(m)$. The containment must be equal; otherwise, the index of $G_E$ in $\GL_2(\widehat{\Z})$ is greater than $[\GL_2(\Z/m\Z) : H_E(m)] = [\GL_2(\Z/m_E\Z) : H_E(m_E)] = 2$, contradicting the assumption that $E/\Q$ is a Serre curve.

\end{proof}

In order to compute $C^{\mathcal{X}}_{E,n,k}$, we need to know $G_E$ (meaning we must know the adelic level $m_E$ and the image of $G_E$ modulo $m_E$). For Serre curves, this is particularly tractable, and was exploited in the work of Jones \cite{MR2534114}. We now give the description of $G_E$ for Serre curves.

First, we define $\chi_4 \colon (\Z/4\Z)^\times \to \{\pm 1\}$ and $\chi_8 \colon (\Z/8\Z)^\times \to \{\pm 1\}$ as follows:
\[
\chi_4(k) =
\begin{cases}
1 & \text{ if }k\equiv  1\pmod 4\\
-1 & \text{ if }k\equiv  3\pmod 4
\end{cases}, 
\quad 
\chi_8(k)=\begin{cases}
1 & \text{ if }k\equiv  1, 7\pmod 8\\
-1 & \text{ if }k\equiv  3, 5\pmod 8
\end{cases}.
\]

We define the character $\psi_m \colon \GL_2(\Z/m\Z) \to \{\pm 1\}$ associated to $E$ by
\[
\psi_m=\prod_{\ell^{\alpha}\parallel m}\psi_{\ell^{\alpha}},
\]
where $\psi_{\ell^\alpha} \colon \GL_2(\Z/\ell^\alpha\Z) \to \{\pm 1\}$ is defined for $M \in \GL_2(\Z/\ell^\alpha\Z)$ by
\begin{equation*}
  \psi_{\ell^\alpha}(M) = 
\begin{cases}
 \left( \frac{\det M_\ell}{\ell}\right) & \text{ if } \ell \text{ is odd},\\
    \epsilon(M_2) & \text{ if } \ell= 2, \alpha \geq 1, \text{ and } \Delta' \equiv 1 \pmod 4, \\
    \chi_4(\det M_4) \epsilon(M_2) & \text{ if } \ell = 2 , \alpha \geq 2,\text{ and }  \Delta' \equiv 3 \pmod 4, \\
    \chi_8(\det M_8) \epsilon(M_2) & \text{ if } \ell =2 , \alpha \geq 3,\text{ and }  \Delta' \equiv 2 \pmod 8,\\
    \chi_8(\det M_8) \chi_4(\det M_4) \epsilon(M_2) & \text{ if } \ell = 2 , \alpha \geq 3, \text{ and } \Delta' \equiv 6 \pmod 8.
\end{cases}
\end{equation*}
As noted in \cite[p.\ 701]{MR2534114}, given $m_E\mid m$, one may see that for $M \in \GL_2(\Z/m\Z)$, we have
$$\epsilon(M_2) \left(\frac{\Delta'}{\det M_{m_E}}\right) = \psi_m(M).$$
In particular, we have $H_E(m) = \ker \psi_m$. Thus $G_E$ is preimage of $\ker \psi_m$ in $\GL_2(\widehat{\Z})$.

\subsection{Galois representations in the CM case} \label{GalRepCM} Having discussed Galois representations for non-CM elliptic curves, we now turn to the CM case. Suppose that $E$ has CM by an order $\mathcal{O}$ in an imaginary quadratic field $K$. In this case, the absolute Galois group $\Gal(\overline{K}/K)$ acts naturally on $T(E)$, which is a one-dimensional $\widehat{\mathcal{O}}$-module, where $\widehat{\mathcal{O}}$ denotes the profinite completion of $\mathcal{O}$. Hence, we can construct the adelic Galois representation associated to $E$,
$$\rho_E \colon\Gal(\overline{K}/K) \to \Aut(T(E)) \simeq \GL_1(\widehat{\mathcal{O}}) \simeq \widehat{\mathcal{O}}^\times.$$
Let $G_E$ denote the image of $\rho_E$.
We now state Serre's open image theorem for CM elliptic curves. 
\begin{Theorem}[Serre, \protect{\cite[p.\ 302, Corollaire]{MR0387283}}]\label{CMSerreopenimagetheorem}
    If $E/\Q$ has CM by $\mathcal{O}$, then $G_E$ is an open subgroup of $\widehat{\mathcal{O}}^\times$. In particular, the index $[\widehat{\mathcal{O}}^\times : G_E]$ is finite.
\end{Theorem}
For each positive integer $m$, consider the natural reduction map
$$\pi_m \colon \widehat{\mathcal{O}}^\times \to (\mathcal{O}/m\mathcal{O})^\times.$$
Let $G_E(m)$ denote the image of the modulo $m$ Galois representation 
\[
\rho_{E, m}: \Gal(\overline{K}/K) \to (\mathcal{O}/m\mathcal{O})^\times
\]
defined by the composition $\pi_m\circ \rho_E$. It follows from \thmref{CMSerreopenimagetheorem} that 
\begin{equation}\label{leveldefinition}
    G_E = \pi_m^{-1}(G_E(m))
\end{equation}
for some positive integer $m$. As in the non-CM case, \eqref{leveldefinition} is equivalent to the statement that for every $n \in \mathbb{N}$,
\begin{equation} \label{preimagedefinition2}
    G_E(n) = \pi^{-1}(G_E(\gcd(n,m))),
\end{equation}
where $\pi\colon  (\mathcal{O}/n\mathcal{O})^{\times}\to (\mathcal{O}/\gcd(n, m)\mathcal{O})^{\times}$ is the natural reduction map.

In the CM case, we follow \cite[p.\ 693]{MR2534114}  to define $m_E$ to be the smallest positive integer $m$ such that 
\eqref{preimagedefinition2} holds and for which
\begin{equation}\label{ramcondi}
    4 \left( \prod_{\ell \text{ ramifies in }K} \ell \right) \text{ divides } m.
\end{equation}

One can prove the following using the same argument sketched in the proof of Lemma \ref{propertyofmE}.

\begin{Lemma}\label{CMpropertyofmE}
    Let $E/\Q$ be a CM elliptic curve of level $m_E$. For any $d_1,d_2 \in \N$ with $d_1 \mid m_E^\infty$ and $(d_2, m_E) = 1$, we have
    $$G_E(d_1d_2) \simeq G_E(d_1) \times (\mathcal{O}/d_2\mathcal{O})^\times.$$
\end{Lemma}

Lemmas \ref{propertyofmE} and \ref{CMpropertyofmE} demonstrate one of the most important properties of $m_E$, which is used to express the constants $C^{\Cyc}_{E,n,k}$ and $C^{\Prime}_{E,n,k}$ as almost Euler products. It is worth noting that both lemmas hold even if $m_E$ is replaced by any positive multiple of it. Thus, the minimality condition in the definition of $m_E$ for both non-CM and CM curves is not required from a theoretical perspective for us. Nonetheless, the minimality of $m_E$ is useful for our computations as it allows us extract more Euler factors.

Let $K/\Q$ be an imaginary quadratic field. We denote its ring of integers by $\mathcal{O}_K$. Let $\mathcal{O}$ be an order of $K$. The index $f = [\mathcal{O}_K : \mathcal{O}]$ is necessarily finite and is called the \emph{conductor} of $\mathcal{O}$. Let $\chi_K$ be the Dirichlet character defined by 
\begin{equation}\label{Kroneckersymbol}
    \chi_K(\ell) =\begin{cases}
    0 & \text{ if } \ell \text{ ramifies in } K, \\
    1 & \text{ if } \ell \text{ splits in } K, \\
    -1 & \text{ if } \ell \text{ is inert in } K.
\end{cases}
\end{equation}
Let $d_K$ be the discriminant of $K$. One can check that
$$\chi_K(\ell) = \left(\frac{d_K}{\ell}\right)$$ for each odd prime $\ell$.
 By \cite[Theorem 9.13]{MR2378655}, we see that $\chi_K$ is a primitive quadratic character. 
 
 We now state a lemma on the size of the mod $\ell^a$ image of $E$. 
\begin{Lemma}\label{Kronecker}
    Let $E/\mathbb{Q}$ be a CM elliptic curve. For $\ell \nmid fm_E$, we have
    $$|G_E(\ell^\alpha)| = \ell^{2(\alpha-1)}(\ell-1)(\ell-\chi_K(\ell)).$$
\end{Lemma} 
\begin{proof}
Since $\mathcal{O}$ is an order of class number $1$, we have
$$\mathcal{O}/ (\ell \mathcal{O}_K \cap \mathcal{O}) = \mathcal{O}/\ell \mathcal{O} \simeq \mathcal{O}_K / \ell \mathcal{O}_K$$
for any $\ell \nmid f$. (See \cite[Proposition 7.20]{MR1028322}.) 
By \lemref{CMpropertyofmE}, we have $G_E(\ell^\alpha) \simeq \left(\mathcal{O}_K/\ell^\alpha \mathcal{O}_K\right)^\times$. Applying \cite[Equation (4)]{MR2869057}, we obtain the desired results.
\end{proof}

Moreover, we have the following uniformity result for CM elliptic curves over $\Q$.
\begin{Proposition}
\label{uniform-CM}
    There is an absolute  constant $C$ such that 
    \[
   fm_E \leq C
    \]
    holds for all CM elliptic curves $E / \Q$.
\end{Proposition}
\begin{proof}
    It suffices to show that the index $[\widehat{\mathcal{O}}^\times : G_E]$, the product of ramified primes in (\ref{ramcondi}), and the conductor of the CM-order $f = [\mathcal{O}_K : \mathcal{O}]$ are uniformly bounded for $E/\Q$. This follows from the fact that there are only finitely many endomorphism rings for CM elliptic curves over $\Q$ and \cite[Theorem 1.1]{MR4077686}. In fact, for CM elliptic curves $E/\Q$, it is known that the conductor of $\mathcal{O}$ is at most $3$. (See \cite[Appendix C, Example 11.3.2]{MR2514094}.)
\end{proof}

\section{Counting matrices}\label{CountMat}
 In this section, we will establish counting results that will play pivotal roles in determining the cyclicity and Koblitz constants for arithmetic progressions. We first outline the general strategy. 
 
 Let $\ell$ be a prime and $\mathcal{P}_\ell$ be a property that certain matrices in $\GL_2(\Z/\ell\Z)$ satisfy. Let $m$ and $n$ be positive integers and $k$ be coprime to $n$. Suppose that we are interested in counting the size of the set
$$X(m) \coloneqq \{ M \in \GL_2(\Z/m\Z) : M_\ell \text{ satisfies } \mathcal{P}_\ell \text{ for each } \ell \mid m, \, \det M \equiv k \neghs \pmod {\gcd(n,m)}\},$$
where $M_\ell$ denotes the reduction of $M$ modulo $\ell$. By the Chinese remainder theorem, it suffices to count the size of $X(\ell^a)$ for each $\ell^a \parallel m$. Furthermore, each element of $X(\ell^a)$ can be realized as a lifting of an element in $X(\ell)$ under the reduction map $\GL_2(\Z/\ell^a\Z) \to \GL_2(\Z/\ell\Z)$. Consequently, the problem of counting the size of $X(m)$ reduces to counting the size of $X(\ell)$ for each $\ell \mid m$. 

The condition that $\ell$ is a prime of cyclic or Koblitz reduction for $E$ can be interpreted as a condition on matrices modulo primes. Thus, with the above strategy in mind, we give a lemma and corollary that will be used  to compute the cyclicity constant $C^{\Cyc}_{E,n,k}$ for non-CM curves.

\begin{Lemma}\label{usefullemma}
Let $\ell$ be a prime, $a$ be a positive integer, and $k$ be an integer coprime to $\ell$. Fix $M \in \GL_2(\Z/\ell\Z)$ with $\det M \equiv k \pmod \ell$. For any integer $\widetilde{k}$ with $\widetilde{k} \equiv k \pmod \ell$, we have
    $$\#\left\{\widetilde{M} \in \GL_2(\Z/\ell^a\Z) : \widetilde{M} \equiv M \neghs \pmod \ell, \det \widetilde{M} \equiv \widetilde{k} \neghs \pmod {\ell^a} \right\} = \ell^{3(a-1)}.$$
\end{Lemma}
\begin{proof} Let $\pi \colon \GL_2(\Z/\ell^a\Z) \to \GL_2(\Z/\ell\Z)$ denote the reduction modulo $\ell$ map, which is a surjective group homomorphism. We have that
\[
\pi^{-1}(M) = \left\{\widetilde{M} \in \GL_2(\Z/\ell^a\Z) : \widetilde{M} \equiv M \neghs \pmod{\ell} \right\}.
\]
The image of $\pi^{-1}(M)$ under $\det \colon \GL_2(\Z/\ell^a\Z) \to (\Z/\ell^a\Z)^\times$ is
\[
\det(\pi^{-1}(M)) = \{k' \in (\Z/\ell^a \Z)^\times : k' \equiv k \neghs \pmod{\ell} \}.
\]
Hence, for any integer $\widetilde{k}$ with $\widetilde{k} \equiv k \pmod \ell$, we have
\begin{align*}
\#\left\{\widetilde{M} \in \GL_2(\Z/\ell^a\Z) : \widetilde{M} \equiv M \neghs \pmod \ell, \det \widetilde{M} \equiv \widetilde{k} \neghs \pmod {\ell^a} \right\} &= \frac{| \pi^{-1}(M) |}{| \det(\pi^{-1}(M)) |}.
\end{align*} 
Finally, we note that $| \pi^{-1}(M) | = | \ker(\pi) | = \ell^{4(a-1)}$ and $| \det(\pi^{-1}(M)) | = \ell^{a-1}$.
\end{proof}\begin{Corollary}\label{countingcyc}
    Fix a prime $\ell$ and positive integer $a$. Let $k$ be an integer coprime to $\ell$. Then 
    \begin{align*}
    \#\{M \in \GL_2(\Z/\ell^a\Z) : M &\not \equiv I \neghs \pmod \ell, \det M \equiv k \neghs \pmod {\ell^a} \} \\  &= \begin{cases}
        \ell^{3(a-1)} \cdot (\ell^3-\ell-1) & \text{ if } k \equiv 1 \neghs \pmod \ell, \\
        \ell^{3(a-1)} \cdot (\ell^3-\ell) & \text{ if } k \not \equiv 1 \neghs \pmod \ell.
    \end{cases}
    \end{align*}
\end{Corollary}
\begin{proof}
    Let $M \in \GL_2(\Z/\ell\Z)$. If $M \not\equiv I \pmod{\ell}$, then any  lifting $\widetilde{M}$ of $M$ in $\GL_2(\Z/\ell^a\Z)$ satisfies $\widetilde{M} \not \equiv I \pmod \ell$. Thus the condition $M \not \equiv I \pmod \ell$ is preserved under lifting.

    If $k \not \equiv 1 \pmod \ell$, then $\det M \equiv k \pmod \ell$ guarantees that $M \not\equiv I \pmod{\ell}$. Since the determinant map $\det \colon \GL_2(\Z/\ell\Z) \to (\Z/\ell\Z)^\times$ is a surjective group homomorphism, one can check that there are $\ell^3-\ell$ matrices $M$ in $\GL_2(\Z/\ell\Z)$ with $\det M \equiv k \pmod \ell$. On the other hand, if $k \equiv 1 \pmod \ell$, we have one less choice for $M$. Along with \lemref{usefullemma}, we obtain the desired results.
\end{proof}

The next lemma gives a corollary that will be useful when computing the Koblitz constant $C^{\Prime}_{E,n,k}$ for non-CM curves.

\begin{Lemma}\label{countinglemma}
Let $\ell$ be an odd prime, $d \in (\Z/\ell\Z)^\times$, and $t \in \Z/\ell\Z$. Then we have
    $$\#\left\{M \in \GL_2(\Z/\ell\Z) : \det M \equiv d \neghs \pmod \ell, \tr M \equiv t \neghs \pmod \ell \right\} = \ell^2 + \ell \cdot \left(\frac{t^2-4d}{\ell}\right),$$
    where $\left(\frac{\cdot}{\ell}\right)$ denotes the Legendre symbol. 
If $\ell = 2$, then we have
$$\#\left\{M \in \GL_2(\Z/2\Z) : \det M \equiv 1 \neghs \pmod 2, \tr M \equiv t \neghs \pmod 2 \right\} =\begin{cases}
 4  &   \text{ if }t \equiv 0 \pmod 2,\\
 2 &   \text{ if } t \equiv 1 \pmod 2.
 \end{cases}$$
\end{Lemma}
\begin{proof} The case when $\ell = 2$ is a direct calculation. See \cite[Lemma 2.7]{MR2178556} for the case when $\ell$ is odd.
\end{proof}

\begin{Corollary}\label{countingprime}
    Fix a prime $\ell$ and positive integer $a$. Let $k$ be an integer coprime to $\ell$. Then
    \begin{align*} \#\{M \in \GL_2(\Z/\ell^a\Z) : &\det (M-I) \not \equiv 0 \neghs \pmod \ell, \det M \equiv k \neghs \pmod {\ell^a} \} \\
    &= \begin{cases} \ell^{3(a-1)} \cdot (\ell^3-\ell^2-\ell) & \text{ if } k \equiv 1 \neghs \pmod \ell,\\
        \ell^{3(a-1)} \cdot (\ell^3-\ell^2-2\ell) & \text{ if } k \not \equiv 1 \neghs \pmod \ell. \end{cases}
    \end{align*}
\end{Corollary}
\begin{proof} The condition that $\det(M-I) \equiv 0 \pmod \ell$ is preserved under lifting, since if $\widetilde{N} \in \GL_2(\Z/\ell^a\Z)$ is any lift of a matrix $N \in \GL_2(\Z/\ell\Z)$, then $\widetilde{N} - I \equiv N - I \pmod{\ell}$, so $\det(\widetilde{N} - I) \equiv \det(N - I) \pmod{\ell}$.

    Now, let $M \in \GL_2(\Z/\ell^a\Z)$ be such that $\det M \equiv k \pmod {\ell^a}$ and note that 
    $$\det (M-I) \equiv 0 \pmod{\ell} \iff \tr M \equiv k +1 \pmod \ell.$$
    Thus, if $\ell\neq 2$, we have that
    $$\left(\frac{(k+1)^2-4k}{\ell}\right) = \left(\frac{(k-1)^2}{\ell}\right) = \begin{cases}
    0 & \text{ if } k \equiv 1 \neghs \pmod \ell, \\
    1 & \text{ if } k \not \equiv 1 \neghs \pmod \ell.
\end{cases}$$
    By \lemref{countinglemma}, this completes the proof when $\ell \neq 2$. When $\ell=2$ and $a = 1$, it is straightforward to check that the lemma holds.
\end{proof}
Now, we turn our attention to the CM case. Let $K$ be an imaginary quadratic field and write $\mathcal{O}_K$ to denote the ring of integers of $K$. Then $\mathcal{O}_K$ is a free $\Z$-module of rank $2$. Fixing a $\Z$-basis,  we can identify $\GL_1(\mathcal{O}_K)=\mathcal{O}_K^{\times}$ as a subgroup of $\GL_2(\Z)$. In the following discussion (and henceforth) the determinant of $g$ for $g \in \mathcal{O}_K^\times$ means the determinant of $g$ considered as a matrix in $\GL_2(\Z)$.
 Moreover, we note that for any odd rational  prime $\ell$ and any integer $a\geq 1$, the determinant of any element in $\ell^a \mathcal{O}_K$ lies in $\ell^a\Z$, so we obtain the induced determinant map  $\det\colon  (\mathcal{O}_K/\ell^a \mathcal{O}_K)^{\times} \to (\Z/\ell^a\Z)^{\times}$, which does not depend on the choice of the basis. 
 \begin{Lemma}\label{usefullemma2}
Let $K$ be an imaginary quadratic field and  $\mathcal{O}_K$ be the ring of integers of $K$. Let $\ell$ be an odd rational prime unramified in $K$ and $a$ be a positive integer. Let $k$ be an integer that is coprime to $\ell$ and fix $g \in (\mathcal{O}_K/\ell\mathcal{O}_K)^\times$ with $\det g \equiv k \pmod \ell$. Then
$$\#\left\{\widetilde{g} \in (\mathcal{O}_K/\ell^a\mathcal{O}_K)^\times : \widetilde{g} \equiv g \neghs \pmod {\ell\mathcal{O}_K}, \det\widetilde{g}\equiv k \neghs \pmod{\ell^a} \right\} = \ell^{a-1}.$$
\end{Lemma}
\begin{proof}
    The reduction map $\pi \colon  (\mathcal{O}_K/\ell^a \mathcal{O}_K)^\times \to (\mathcal{O}_K/\ell\mathcal{O}_K)^\times$ is a surjective group homomorphism. Regardless of whether $\ell$ splits or is inert in $K$, we have $|\ker \pi| = \ell^{2(a-1)}$ by \lemref{Kronecker}. Therefore, 
    $$|\left\{\widetilde{g} \in (\mathcal{O}_K/\ell^a\mathcal{O}_K)^\times : \widetilde{g} \equiv g \neghs \pmod {\ell\mathcal{O}_K} \right\}| = |\pi^{-1}(g)| = |\ker \pi| = \ell^{2(a-1)}.$$
    The image of $\pi^{-1}(g)$ under $\det \colon  (\mathcal{O}/\ell^a\mathcal{O})^\times \to (\Z/\ell^a\Z)^\times$ is
    $$\det(\pi^{-1}(g)) = \left\{ k' \in (\Z/\ell^a\Z)^\times : k' \equiv k \neghs \pmod \ell \right\}.$$
    Thus, we have  and $\left|\det (\pi^{-1}(g))\right| = \ell^{a-1}$. Finally, note that
    $$\#\left\{\widetilde{g} \in (\mathcal{O}/\ell^a\mathcal{O})^\times : \widetilde{g} \equiv g \neghs \pmod {\ell\mathcal{O}_K}, \det g \equiv k \neghs \pmod {\ell^a} \right\} = \frac{|\pi^{-1}(g)|}{|\det(\pi^{-1}(g))|} = \ell^{a-1}.$$ \qedhere
\end{proof}
We now prove a  corollary that will be used to identify the Euler factors of the Koblitz constant $C^{\Prime}_{E,n,k}$ for CM curves.
\begin{Corollary}\label{countingprimeCM}
    Let $K$ be an imaginary quadratic field. Fix an odd rational prime $\ell$ that is unramified in $K$. Let $k$ be an integer that is coprime to $\ell$. If $\ell$ splits in $K$, then
    \begin{align*}\#\{ g \in (\mathcal{O}_K/\ell^a\mathcal{O}_K)^\times : \det(g-1) &\not\equiv 0 \neghs \pmod \ell, \det g \equiv k \neghs \pmod {\ell^a}  \} \\
    &= \begin{cases}
        \ell^{a-1}(\ell-2) & \text{ if } k \equiv 1 \neghs \pmod \ell, \\
        \ell^{a-1}(\ell-3) & \text{ if } k \not \equiv 1 \neghs \pmod \ell.
    \end{cases}
    \end{align*}
    If $\ell$ is inert in $K$, then
    \begin{align*}
    \#\{ g \in (\mathcal{O}_K/\ell^a\mathcal{O}_K)^\times : \det(g-1) &\not\equiv 0 \neghs \pmod {\ell}, \det g \equiv k \neghs \pmod {\ell^a}  \} \\
    &= \begin{cases}
        \ell^a & \text{ if } k \equiv 1 \neghs \pmod \ell, \\
        \ell^{a-1}(\ell+1) & \text{ if } k \not \equiv 1 \neghs \pmod \ell.
    \end{cases}
    \end{align*}
\end{Corollary}
\begin{proof}
   By \lemref{usefullemma2}, it suffices to compute the case where $a = 1$.
    
    Suppose $\ell$ splits in $K$. Then we have that  $\mathcal{O}_K/\ell \mathcal{O}_K \simeq \F_\ell \times \F_\ell$ and the determinant map $\det \colon  \F_\ell^{\times} \times \F_\ell^{\times} \to \F_\ell^{\times}$ is given by  $(a,b) \mapsto ab$. Thus, the set in question can be expressed as
    $$\left\{(g_1,g_2) \in \F_\ell^\times \times \F_\ell^\times : g_1-1, g_2-1 \in \F_\ell^\times, g_1g_2 \equiv k \neghs \pmod {\ell}\right\}.$$
    Hence, any element in the set is of the form $(g,kg^{-1})$ where both $g$ and $kg^{-1}$ are not congruent to $1$ modulo $\ell$. Thus, the size of the set is $\ell-2$ if $k \equiv 1 \pmod \ell$ and $\ell-3$ otherwise.

    Now, suppose $\ell$ is inert in $K$. Then we have $\mathcal{O}_K/\ell\mathcal{O}_K \simeq \F_{\ell^2}$ and the determinant map $\det \colon  \F_{\ell^2} \to \F_\ell$ is identified with the norm map $N_{\F_{\ell^2}/\F_\ell}; x \mapsto x^{\ell+1}$. Thus, the set in consideration can be understood as
    $$\left\{g \in \F_{\ell^2}^\times : (g-1)^{\ell+1} \in \F_\ell^\times, g^{\ell+1} \equiv k \neghs \pmod \ell\right\}.$$
    For each $k$ coprime to $\ell$, there are exactly $\ell+1$ choices of $g \in \F_{\ell^2}^\times$ with $g^{\ell+1} \equiv k \pmod \ell$. In case $k \equiv 1 \pmod \ell$, we have one less choice due to the constraint $(g-1)^{\ell+1} \in \F_\ell^\times$.
\end{proof}

\section{Definitions of the constants}\label{DefCon}
\subsection{On the cyclicity constant}\label{OnCycCon}
In this subsection, we introduce the definition of the cyclicity constant $C_E^{\Cyc}$, given by Serre, and its average counterpart $C^{\Cyc}$. For coprime integers $n$ and $k$, we introduce  the cyclicity constant for primes in arithmetic progression $C_{E, n, k}^{\Cyc}$, given by Akbal and G\"ulo\u{g}lu, and its average counterpart $C^{\Cyc}_{n, k}$.

First of all, Serre \cite[pp.\ 465-468]{MR3223094} defined the cyclicity constant $C^{\Cyc}_E$ to be
\begin{equation}\label{CcycE}
    C_E^{\Cyc} \coloneqq \sum_{n \geq 1} \frac{\mu(n)}{[\Q(E[n]):\Q]},
\end{equation}
where $\mu$ denotes the M\"obius function and $\Q(E[n])$ is the $n$-th division field of $E$. He proved that, under GRH, $C^{\Cyc}_E$ is the density of primes of cyclic reduction for $E$; see Conjecture \ref{Cyclicity}.

For a non-CM elliptic curve $E/\Q$, Jones \cite[p.\ 692]{MR2534114} observed that \eqref{CcycE} can be expressed as an almost Euler product using the adelic level of $E$. Specifically, he showed that
\begin{equation}\label{Ccycalmosteuler}
    C_E^{\Cyc} = \left(\sum_{d\mid m_E} \frac{\mu(d)}{[\Q(E[d]):\Q]}\right) \prod_{\ell \nmid m_E} \left( 1 - \frac{1}{|\GL_2(\Z/\ell\Z)|}\right).
\end{equation}
The average counterpart of $C^{\Cyc}_E$ is 
\begin{equation}\label{Ccyc} 
    C^{\Cyc} \coloneqq \prod_{\ell}\left( 1 - \frac{1}{|\GL_2(\Z/\ell\Z)|}\right) \approx 0.813752.
\end{equation}
As mentioned in the introduction, Gekeler \cite{MR2226271} demonstrated that $C^{\Cyc}$ represents the average cyclicity constant in the local viewpoint. Later, Banks and Shparlinski \cite{MR2570668} verified that the constant also describes the density of primes of cyclic reduction on average in the global sense. Furthermore, Jones \cite{MR2534114} verified that the average of $C^{\Cyc}_E$ coincides with $C^{\Cyc}$.

Let $\zeta_n$ denote a primitive $n$-th root of unity, and let $\sigma_k \in \Gal(\Q(\zeta_n)/\Q)$ map $\zeta_n \mapsto \zeta_n^k$. Define
$$\gamma_{n,k}(\Q(E[d])) \coloneqq \begin{cases}
    1 & \text{ if } \sigma_k \text{ fixes } \Q(E[d]) \cap \Q(\zeta_n) \text{ pointwise},\\
    0 & \text{ otherwise}.
\end{cases}$$
Akbal and G\"ulo\u{g}lu \cite{MR4504664} defined the constant $C^{\Cyc}_{E,n,k}$ as follows,
\begin{equation} \label{CcycEnk-def}
    C^{\Cyc}_{E,n,k} \coloneqq \sum_{d \geq 1} \frac{\mu(d) \gamma_{n,k}(\Q(E[d]))}{[\Q(E[d])\Q(\zeta_n):\Q]}.
\end{equation}
They proved that this constant represents the density of primes $p \equiv k \pmod n$ of cyclic reduction for $E$, under GRH. Recently, Jones and the first author \cite{Jones-Lee} demonstrated that for a non-CM elliptic curve $E/\Q$, this density can be expressed as an almost Euler product as follows,
\begin{equation}\label{CCycEnk}
C^{\Cyc}_{E,n,k} = \left( \sum_{d\mid m_E} \frac{\mu(d)\gamma_{n,k}(\Q(E[d]))}{[\Q(E[d])\Q(\zeta_n) : \Q]} \right) \prod_{\substack{\ell \nmid m_E \\ \ell \mid (n,k-1)}} \left( 1 - \frac{\phi(\ell)}{|\GL_2(\Z/\ell\Z)|}\right) \prod_{\substack{\ell \nmid nm_E}}\left( 1 - \frac{1}{|\GL_2(\Z/\ell\Z)|}\right).
\end{equation}
Finally, the average counterpart of $C^{\Cyc}_{E,n,k}$ is given by 
\begin{equation}\label{Ccycnk}
    C^{\Cyc}_{n,k} \coloneqq \frac{1}{\phi(n)}\prod_{\substack{\ell \mid (n,k-1)}} \left( 1 - \frac{\phi(\ell)}{|\GL_2(\Z/\ell\Z)|}\right) \prod_{\substack{\ell \nmid n}}\left( 1 - \frac{1}{|\GL_2(\Z/\ell\Z)|}\right).
\end{equation}
Observe that \eqref{Ccycnk} coincides with $\eqref{CCycEnk}$ if $m_E$ is taken to be $1$. While $m_E = 1$ is impossible for any given elliptic curve over $\Q$, it is plausible to think that the role of $m_E$ is invisible when considered over the family of elliptic curves ordered by height. Indeed, as mentioned in the introduction, the first author \cite{Lee} demonstrated that $C^{\Cyc}_{n,k}$ represents the average density of primes $p \equiv k \pmod n$ of cyclic reduction for the  family of elliptic curves ordered by height. 

We now prove a proposition that serves as a sanity check for $C^{\Cyc}_{n,k}$. While it can be derived from the main theorem of \cite{Lee}, we opt to include a self-contained proof to draw a parallel with the upcoming Proposition \ref{AvgSum}.
\begin{Proposition}\label{Avgsumcyc}
    For any positive integer $n$, we have
    $$\sum_{\substack{1 \leq k \leq n \\ (n,k) = 1}}C^{\Cyc}_{n,k} = C^{\Cyc},$$
    where $C^{\Cyc}$ and $C^{\Cyc}_{n,k}$ are defined in \eqref{Ccyc} and \eqref{Ccycnk}.
\end{Proposition}
\begin{proof}
    For notational convenience, we define
    $$f(\ell) = 1 - \frac{\phi(\ell)}{|\GL_2(\Z/\ell\Z)|}.$$
    It suffices to verify that
    \begin{equation}\label{obviousequation}
        F(n) \coloneqq \frac{1}{\phi(n)}\sum_{\substack{1 \leq k \leq n \\ (k,n)=1}} \prod_{\substack{\ell \mid n \\ k \equiv 1 (\ell)}}f(\ell) = \prod_{\ell \mid n } \left(1-\frac{1}{|\GL_2(\Z/\ell\Z)|}\right).
    \end{equation}
    First, we prove that \eqref{obviousequation} holds for $n = p^a$, a prime power. Observe that
\begin{align*}
    F(p^a) &= \frac{1}{\phi(p^a)} \sum_{\substack{1\leq k\leq p^a\\ (k,p^a) = 1}} \prod_{\substack{\ell \mid p^a \\ k \equiv 1(\ell)}} f(\ell)\\
    &= \frac{1}{\phi(p^a)} \left( p^{a-1} f(p) + p^{a-1}(p-2) \right)  = 1 - \frac{1}{|\GL_2(\Z/p\Z)|}.
\end{align*}
    Now, we prove that $F$ is multiplicative. Let $p^a$ be a prime power and $n$ be a positive integer coprime to $p$. Then 
    \begin{align*}
        F(p^an) &= \frac{1}{\phi(p^an)}\sum_{\substack{1 \leq k \leq p^an \\ (k,p^an) = 1}} \prod_{\substack{\ell \mid p^an \\ k \equiv 1 (\ell)}} f(\ell) \\
        &= \frac{1}{\phi(p^a)} \cdot \frac{1}{\phi(n)} \left[ \sum_{\substack{1 \leq k \leq p^a n \\ (k,pn) = 1 \\ k \equiv 1 (p)}} f(p) \prod_{\substack{\ell \mid n \\ k \equiv 1 (\ell)}} f(\ell) + \sum_{\substack{1 \leq k \leq p^a n \\ (k,pn) = 1 \\ k \not \equiv 1 (p)}} \prod_{\substack{\ell \mid n \\ k \equiv 1 (\ell)}}f(\ell) \right]\\
        &= \frac{\left( p^{a-1} + p^{a-1}(p-2)\right)}{\phi(p^a)} \cdot \frac{1}{\phi(n)} \left[ \sum_{\substack{1 \leq k \leq n \\ (k,n) = 1 }}\prod_{\substack{\ell\mid n \\ k \equiv 1 (\ell)}} f(\ell)\right] = F(p^a) \cdot F(n).
    \end{align*}
    This completes the proof.
\end{proof}

\subsection{On the Koblitz constant} \label{OnKobCon}
Now we give the definition of Koblitz's constant $C_E^{\Prime}$ defined by Zywina and its average counterpart $C^{\Prime}$ given by Balog, Cojocaru, and David. Based on Zywina's method, for coprime integers $n, k$, we propose Koblitz's constants $C_{E, n, k}^{\Prime}$ for primes in arithmetic progression and its average counterpart $C_{n, k}^{\Prime}$.  

Let $E/\Q$ be a non-CM elliptic curve of conductor $N_E$ and $m$ be a positive integer. For $p \nmid mN_E$, let $\Frob_p$ be a Frobenius element at $p$ in $\Gal(\overline{\Q}/\Q)$ (see \cite[2.1, I-6]{MR1484415} for the definition of $\Frob_p$). 
We have that
\begin{align}\label{Koblitzcondition}
    |\widetilde{E}_p(\F_p)| \equiv \det(I - \rho_{E,m}(\Frob_p))  \pmod m.
\end{align}
by \cite[Chapter V.\ Theorem 2.3.1]{MR2514094}. Thus, we see that an odd prime $p$ is of Koblitz reduction if and only if the right-hand side of \eqref{Koblitzcondition} is invertible modulo $m$, for every $m < |\widetilde{E}_p(\F_p)|$ such that $\gcd(p,m) = 1$.\footnote{This biconditional statement fails if $|\widetilde{E}_p(\F_p)| = p^r$ for some integer $r \geq 2$. However, this can only happen if $p = 2$ due to the Hasse bound.} For such an integer $m$, we set
\begin{equation} \label{PsiPrime}
\Psi^{\Prime}(m) \coloneqq \left\{ M \in \GL_2(\Z/m\Z) : \det(I-M) \in (\Z/m\Z)^\times\right\}.
\end{equation}
Define the ratio
$$\delta^{\Prime}_E(m) \coloneqq \frac{|G_E(m) \cap \Psi^{\Prime}(m)|}{|G_E(m)|}.$$
The Koblitz constant, proposed by Zywina \cite{MR2805578}, is defined by
\begin{equation}\label{Koblitzconstantoriginal}
    C_E^{\Prime} \coloneqq \lim_{m\to \infty} \frac{\delta_E^{\Prime}(m)}{\prod_{\ell \mid m} \left(1-1/\ell\right)},
\end{equation}
where the limit is taken over all positive integers ordered by divisibility.

We start by proving some properties of $\delta_E^{\Prime}(\cdot)$, which were originally remarked in \cite{MR2805578}.

\begin{Proposition}\label{Koblitzeulerproduct}
    Let $E/\Q$ be a non-CM elliptic curve of adelic level $m_E$. Then $\delta_E^{\Prime}(\cdot)$, as an arithmetic function, satisfies the following properties:
    \begin{enumerate}
        \item for any positive integer $m$, $\delta_E^{\Prime}(m) = \delta_E^{\Prime}(\rad(m))$;
        \item for any prime $\ell \nmid m_E$ and integer $d$ coprime to $\ell$, $\delta_E^{\Prime}(d \ell) = \delta_E^{\Prime}(d) \cdot \delta_E^{\Prime}(\ell)$.
    \end{enumerate}
    Therefore, \eqref{Koblitzconstantoriginal} can be expressed as follows,
    \begin{equation}\label{Kobproduct}
        C^{\Prime}_E = \frac{\delta_E^{\Prime}(\rad (m_E))}{\prod_{\ell \mid m_E} (1-1/\ell)} \cdot \prod_{\ell \nmid m_E} \frac{\delta_{E}^{\Prime}(\ell)}{1-1/\ell}.
    \end{equation}
\end{Proposition}
\begin{proof}
    Let us prove item (1). Let $r = \rad(m)$ and $\varpi \colon G_E(m) \to G_E(r)$ be the usual reduction map. In particular, $\varpi$ is a surjective group homomorphism. We will show that 
    \begin{equation}\label{easycheck0}
        \varpi^{-1}\left(G_E(r) \cap \Psi^{\Prime}(r)\right) = G_E(m) \cap \Psi^{\Prime}(m).
    \end{equation}
    Let $M \in G_E(r) \cap \Psi^{\Prime}(r)$ and $\widetilde{M} \in \varpi^{-1}(M)$. Recall that $\det(M-I)$ is invertible modulo $r$ and that $m$ is only supported by the prime factors of $r$. Thus, $\det(\widetilde{M}-I)$ is invertible modulo $m$ and $\widetilde{M} \in G_E(m) \cap \Psi^{\Prime}(m)$. The other inclusion is obvious, and hence \eqref{easycheck0} is obtained. Therefore,
    
    $$\delta_E^{\Prime}(m) = \frac{\left| G_E(m) \cap \Psi^{\Prime}(m)\right|}{\left|G_E(m)\right|} = \frac{\left|\varpi^{-1}\left(G_E(r) \cap \Psi^{\Prime}(r)\right)\right|}{\left|\varpi^{-1}\left(G_E(r)\right)\right|} = \delta_E^{\Prime}(r).$$

    Let us prove item (2). By \lemref{propertyofmE}, we have an isomorphism,
    \begin{equation}\label{easyiso}
        G_E(d\ell) \simeq G_E(d) \times \GL_2(\Z/\ell\Z).
    \end{equation}
   It suffices to show that the isomorphism induces a bijection between the two sets
    \begin{equation}\label{easybij0}
        G_E(d\ell) \cap\Psi^{\Prime}(d\ell) \, \text{ and } \,
        \left(G_E(d) \cap \Psi^{\Prime}(d)\right) \times \left(\GL_2(\Z/\ell\Z) \cap \Psi^{\Prime}(\ell)\right).
    \end{equation}
    Take $M \in G_E(d\ell) \cap \Psi^{\Prime}(d\ell)$. By a similar argument to the proof of (1), we have that $M_d \in G_E(d) \cap \Psi^{\Prime}(d)$ and $M_\ell \in \GL_2(\Z/\ell\Z) \cap \Psi^{\Prime}(\ell)$.    
Now, let $M' \in G_E(d) \cap \Psi^{\Prime}(d)$ and $M'' \in \GL_2(\Z/\ell\Z) \cap \Psi^{\Prime}(\ell)$. Viewing $(M',M'') \in G_E(d) \times \GL_2(\Z/\ell\Z)$, there exists a unique element $M \in G_E(d\ell)$ with $M_d = M'$ and $M_\ell = M''$ by \eqref{easyiso}. Since $\det(M'-I)\in (\Z/d\Z)^{\times}$ and $\det(M''-I)\in (\Z/\ell\Z)^{\times}$,  we have  $\det(M-I)\in (\Z/d\ell\Z)^{\times}$; in particular, $M \in \Psi^{\Prime}(d\ell)$. Therefore, \eqref{easybij0} is established.

Along with \eqref{easyiso}, we obtain
\begin{align*}
        \delta_{E}^{\Prime}(d\ell) &= \frac{\left|G_E(d\ell)\cap \Psi^{\Prime}(d\ell)\right|}{\left|G_E(d\ell)\right|} \\
        &= \frac{|G_E(d) \cap \Psi^{\Prime}(d)|}{|G_E(d)|} \cdot \frac{|\GL_2(\Z/\ell\Z) \cap \Psi^{\Prime}(\ell)|}{|\GL_2(\Z/\ell\Z)|} \\
        &= \delta^{\Prime}_E(d) \cdot \delta^{\Prime}_E(\ell).
    \end{align*}
    This completes the proof.
\end{proof}
\begin{Remark}\label{absconv}
    Suppose that $\ell \nmid m_E$ and $M \in \GL_2(\Z/\ell\Z)$. Note that $\det(M-I) \in (\Z/\ell\Z)^\times$ if and only if $1$ is not an eigenvalue of $M$. One can check from Table 12.4 in \cite[XVIII]{MR1878556} that
    \begin{equation*}
        \#\left\{ M \in \GL_2(\Z/\ell \Z) : M \text{ has eigenvalues } 1 \text{ and } k \right\} = \begin{cases}
        \ell^2 + \ell & \text{ if } k \not \equiv 1 \pmod \ell,\\
        \ell^2 & \text{ if } k \equiv 1 \pmod \ell.
    \end{cases}
    \end{equation*}
    Thus, we see that
    \begin{equation}\label{Koblitzpart}
        \frac{\delta_E^{\Prime}(\ell)}{1-1/\ell} = \frac{\ell}{\ell-1} \cdot \left(1-\frac{(\ell-2)(\ell^2+\ell) + \ell^2}{(\ell^2-1)(\ell^2-\ell)}\right) = 1 - \frac{\ell^2-\ell-1}{(\ell-1)^3(\ell+1)} \sim 1 - \frac{1}{\ell^2} \hs \text{ as } \ell \to \infty,
    \end{equation}
    and hence the infinite product in \eqref{Kobproduct} converges absolutely.
\end{Remark}

The average counterpart of $C^{\Prime}_E$ is given by
\begin{equation}\label{Cprime}
    C^{\Prime} \coloneqq  \prod_{\ell} \left( 1 - \frac{\ell^2-\ell-1}{(\ell-1)^3(\ell+1)}\right) \approx 0.505166.
\end{equation}
As mentioned earlier, Balog, Cojocaru, and David \cite{MR2843097} demonstrated that $C^{\Prime}$ represents the average Koblitz constant, while Jones \cite{MR2534114} verified that the average of $C^{\Prime}_E$ coincides with $C^{\Prime}$. Unlike for the cyclicity problem, Koblitz's problem in arithmetic progressions has not yet been studied in the literature. We construct $C^{\Prime}_{E,n,k}$ in a parallel way to Zywina's method and propose a candidate for the average constant $C^{\Prime}_{n,k}$.

Let $E/\Q$ be an elliptic curve of conductor $N_E$. For a prime $p \nmid nN_E$, let $\Frob_p$ be a Frobenius element lying above $p$ in $\Gal(\overline{\Q}/\Q)$. We have that
$$\det(\rho_{E,n}(\Frob_p)) \equiv p \pmod n.$$
Along with \eqref{Koblitzcondition}, let us consider the set
\begin{equation}\label{Psiprimenk}
    \Psi^{\Prime}_{n,k}(m) \coloneqq \left\{ M \in \GL_2(\Z/m\Z) : \det(I-M) \in (\Z/m\Z)^\times, \det M \equiv k \neghs \pmod {\gcd(m,n)}\right\}.
\end{equation}
One may note that $\rho_{E,m}(\Frob_p) \in G_E(m) \cap \Psi^{\Prime}_{n,k}(m)$ if and only if $p \equiv k \pmod {\gcd(n, m)}$ and $|\widetilde{E}_p(\F_p)|$ is invertible $\Z/m\Z$. For this reason, we define the ratio
\begin{equation*}
    \delta^{\Prime}_{E,n,k}(m) \coloneqq \frac{\left|G_E(m) \cap \Psi^{\Prime}_{n,k}(m)\right|}{\left|G_E(m)\right|}.
\end{equation*}
Building upon Zywina's approach, we are led to define 
\begin{equation}\label{KoblitzconstantAPoriginal}
    C_{E,n,k}^{\Prime} \coloneqq \lim_{m\to \infty} \frac{\delta_{E,n,k}^{\Prime}(m)}{\prod_{\ell \mid m} \left(1-1/\ell\right)},
\end{equation}
where the limit is taken over all positive integers, ordered by divisibility.

\begin{Proposition}\label{KoblitzAPeulerproduct}
    Let $E/\Q$ be a non-CM elliptic curve of adelic level $m_E$ and $n$ be a positive integer. Let $L$ be defined as in \eqref{E:Def-L}.
    Then, $\delta_{E,n,k}^{\Prime}(\cdot)$, as an arithmetic function, satisfies the following properties:
    \begin{enumerate}
        \item Let $L \mid L' \mid L^\infty$. Then, $\delta_{E,n,k}^{\Prime}(L) = \delta_{E,n,k}^{\Prime}(L')$;
        \item Let $\ell^\alpha$ be a prime power and $d$ be a positive integer with $(\ell, Ld) = 1$. Then, $\delta_{E,n,k}^{\Prime}(d\ell^a) = \delta_{E,n,k}^{\Prime}(d) \cdot \delta_{E,n,k}^{\Prime}(\ell^\alpha)$.
        \item Let $\ell^\alpha \parallel n$ and $(\ell, L) = 1$. Then, for any $\beta \geq \alpha$, $\delta^{\Prime}_{E,n,k}(\ell^\beta) = \delta^{\Prime}_{E,n,k}(\ell^\alpha)$. Further, if $\ell \nmid nL$, we have $\delta_{E,n,k}^{\Prime}(\ell^\beta) = \delta^{\Prime}_E(\ell)$.
    \end{enumerate}
    Therefore, \eqref{KoblitzconstantAPoriginal} can be expressed as follows,
    \begin{equation}\label{KobAPproduct}
        C^{\Prime}_{E,n,k} = \frac{\delta_{E,n,k}^{\Prime}(L)}{\prod_{\ell \mid L}(1-1/\ell)} \cdot \prod_{\substack{\ell \nmid m_E \\ \ell^{\alpha} \parallel n}} \frac{\delta_{E,n,k}^{\Prime}(\ell^{\alpha})}{1-1/\ell} \cdot \prod_{\ell \nmid nm_E} \frac{\delta_E^{\Prime}(\ell)}{1-1/\ell}.
    \end{equation}
    and the infinite product converges absolutely.
\end{Proposition}
\begin{proof}

    Let us prove item (1). Consider the natural reduction map $\varpi \colon G_E(L') \to G_E(L)$, which is a surjective group homomorphism. We will show that
    \begin{equation}\label{easycheck}
        \varpi^{-1}\left(G_E(L) \cap \Psi^{\Prime}_{n,k}(L)\right) = G_E(L') \cap \Psi^{\Prime}_{n,k}(L')
    \end{equation}

    Let $M \in G_E(L) \cap \Psi^{\Prime}_{n,k}(L)$ and $\widetilde{M} \in \varpi^{-1}(M)$. Recall that $\det(M-I)$ is invertible modulo $L$ and that $L'$ is only supported by the prime factors of $L$. Thus, $\det(\widetilde{M}-I)$ is invertible modulo $L'$. Since $\gcd(n,L) = \gcd(n,L')$, we also have $\det \widetilde{M} \equiv k \pmod {\gcd(n,L')}$. Thus, $\widetilde{M} \in G_E(L') \cap \Psi^{\Prime}_{n,k}(L')$. The other inclusion is obvious, and hence \eqref{easycheck} is obtained. Therefore, we have

    $$\delta_{E,n,k}^{\Prime}(L') = \frac{\left|G_E(L') \cap \Psi^{\Prime}_{n,k}(L')\right|}{|G_E(L')|} = \frac{\left|\varpi^{-1}\left(G_E(L) \cap \Psi^{\Prime}_{n,k}(L)\right)\right|}{|\varpi^{-1}(G_E(L))|} = \delta^{\Prime}_{E,n,k}(L).$$

 Let us prove item (2). By \lemref{propertyofmE}, we have an isomorphism,
    \begin{equation}\label{easyiso1}
        G_E(d\ell^\alpha) \simeq G_E(d) \times \GL_2(\Z/\ell^\alpha\Z).
    \end{equation}
It suffices to show that the isomorphism induces a map between the sets 
    \begin{equation}\label{easybij1}
        G_E(d\ell^\alpha) \cap\Psi_{n,k}^{\Prime}(d\ell^\alpha) \text{ and } \left(G_E(d) \cap \Psi^{\Prime}_{n,k}(d)\right) \times \left(\GL_2(\Z/\ell^\alpha\Z) \cap \Psi^{\Prime}_{n,k}(\ell^\alpha)\right).
    \end{equation}
    Say $M \in G_E(d\ell^\alpha) \cap \Psi^{\Prime}_{n,k}(d\ell^\alpha)$. By a similar argument to the proof of (1), one may see that $M_d \in G_E(d) \cap \Psi_{n,k}^{\Prime}(d)$ and $M_{\ell^\alpha} \in \GL_2(\Z/\ell^\alpha\Z) \cap \Psi_{n,k}^{\Prime}(\ell^\alpha)$.    
Now, let $M' \in G_E(d) \cap \Psi^{\Prime}(d)$ and $M'' \in \GL_2(\Z/\ell^\alpha\Z) \cap \Psi^{\Prime}(\ell^\alpha)$. Viewing $(M',M'') \in G_E(d) \times \GL_2(\Z/\ell^\alpha\Z)$, there exists a unique element $M \in G_E(d\ell^\alpha)$ with $M_d = M'$ and $M_{\ell^\alpha} = M''$ by \eqref{easyiso1}. Note that since $\det(M'-I)\in (\Z/d\Z)^{\times}$ and $\det(M''-I)\in (\Z/\ell^{\alpha}\Z)^{\times}$, we have $\det(M-I)\in (\Z/d\ell^{\alpha}\Z)^{\times}$; in particular, $M \in \Psi^{\Prime}_{n,k}(d\ell^\alpha)$. Therefore, \eqref{easybij1} is established.

Along with \eqref{easyiso1}, we obtain
\begin{align*}
        \delta_{E,n,k}^{\Prime}(d\ell^\alpha) &= \frac{\left|G_E(d\ell^\alpha)\cap \Psi^{\Prime}_{n,k}(d\ell^\alpha)\right|}{\left|G_E(d\ell^\alpha)\right|} \\
        &= \frac{|G_E(d) \cap \Psi^{\Prime}(d)|}{|G_E(d)|} \cdot \frac{|\GL_2(\Z/\ell^\alpha\Z) \cap \Psi^{\Prime}(\ell^\alpha)|}{|\GL_2(\Z/\ell^\alpha\Z)|} = \delta^{\Prime}_{E,n,k}(d) \cdot \delta^{\Prime}_{E,n,k}(\ell^\alpha).
    \end{align*}

    Finally, let us prove item (3).  Since $\ell \nmid m_E$, by \lemref{propertyofmE}, $G_E(\ell^\alpha)$ and $G_E(\ell^\beta)$ are the full groups, $\GL_2(\Z/\ell^\alpha\Z)$ and $\GL_2(\Z/\ell^\beta\Z)$. Let $\varpi \colon \GL_2(\Z/\ell^\beta\Z) \to \GL_2(\Z/\ell^\alpha\Z)$ be the natural reduction map which is a surjective group homomorphism. By a similar argument as in the proof of item (1), it suffices to check that 
    \begin{equation}\label{equivalence3}
        \varpi^{-1} \left( \Psi^{\Prime}_{n,k}(\ell^{\alpha})\right) = \Psi^{\Prime}_{n,k}(\ell^\beta).
    \end{equation}

    Take $M \in \Psi^{\Prime}_{n,k}(\ell^\alpha)$ and let $\widetilde{M} \in \varpi^{-1}(M)$. By the same reasoning in the proof of item (1), $\det(\widetilde{M}-I)$ is invertible modulo $\ell^\beta$. Since $\gcd(n,\ell^\alpha) = \gcd(n,\ell^\beta) = \ell^\alpha$, we also have $\det \widetilde{M} \equiv k \pmod {\ell^\alpha}$. The other inclusion is obvious, and hence \eqref{equivalence3} is obtained. Thus, we have 
    $$\delta_{E,n,k}^{\Prime}(\ell^\beta) = \frac{\left|\Psi^{\Prime}_{n,k}(\ell^\beta)\right|}{\left|\GL_2(\Z/\ell^\beta\Z)\right|} = \frac{\left|\varpi^{-1}\left(\Psi^{\Prime}_{n,k}(\ell^\alpha)\right) \right|}{\left| \varpi^{-1}\left(\GL_2(\Z/\ell^\alpha\Z)\right)\right|} = \delta_{E,n,k}^{\Prime}(\ell^\alpha).$$
    In case $\ell \nmid nL$, let $\varpi \colon \GL_2(\Z/\ell^\beta\Z) \to \GL_2(\Z/\ell\Z)$. It suffices to check
    \begin{equation}\label{lastcheck}
        \varpi^{-1}\left(\Psi^{\Prime}(\ell)\right) = \Psi^{\Prime}_{n,k}(\ell^\beta). 
    \end{equation}
    Note that the condition $\det M \equiv k \pmod {\gcd(n,\ell)}$ is trivial, and hence $\Psi^{\Prime}_{E,n,k}(\ell) = \Psi^{\Prime}_E(\ell)$. Let $M \in \Psi^{\Prime}_E(\ell)$. Note that every lifting $\widetilde{M} \in \varpi^{-1}(M)$ belongs to $\Psi^{\Prime}_{E,n,k}(\ell^\beta)$. The other inclusion is obvious, and hence \eqref{lastcheck} is obtained. Thus, we have
    $$\delta_{E,n,k}^{\Prime}(\ell^\beta) = \frac{\left|\Psi^{\Prime}_{n,k}(\ell^\beta)\right|}{|\GL_2(\Z/\ell^\beta\Z)|} = \frac{\left|\varpi^{-1}\left(\Psi^{\Prime}(\ell)\right)\right|}{\left|\varpi^{-1}(\GL_2(\Z/\ell\Z))\right|} = \delta_E^{\Prime}(\ell).$$
    By grouping the prime factors of $M$ in \eqref{KoblitzconstantAPoriginal} according to whether they divide $L$ or not, we obtain \eqref{KobAPproduct}. The absolute convergence of \eqref{KobAPproduct} follows from \remref{absconv}.
\end{proof}

\begin{Lemma}\label{deltaprimeellprime}
    Suppose $\ell^{\alpha} \parallel n$ and $\ell \nmid m_E$. Then
    \begin{equation*}
    \frac{\delta_{E,n,k}^{\Prime}(\ell^{\alpha})}{1-1/\ell} = \begin{cases}
        \frac{1}{\phi(\ell^{\alpha})} \left( 1 - \frac{\ell}{|\GL_2(\Z/\ell\Z)|}\right) & \text{ if } k \equiv 1 \pmod \ell, \\
        \frac{1}{\phi(\ell^{\alpha})}\left( 1 - \frac{\ell^2+\ell}{|\GL_2(\Z/\ell\Z)|}\right) & \text{ if } k \not \equiv 1 \pmod \ell.
    \end{cases}
\end{equation*}
\end{Lemma}
\begin{proof}
    By the assumption, we have $G_E(\ell^{\alpha}) \simeq \GL_2(\Z/\ell^{\alpha}\Z)$. Recall that
\begin{align*}
    \Psi^{\Prime}_{n,k}(\ell^{\alpha}) = \left\{M \in \GL_2(\Z/\ell^{\alpha}\Z) : \det(M-I) \in (\Z/\ell^{\alpha}\Z)^\times, \det M \equiv k \neghs \pmod {\gcd(n,\ell^{\alpha})} \right\}.
\end{align*}

whose cardinality was determined in \corref{countingprime}. A brief calculation reveals the desired result.
\end{proof}
Let $E/\Q$ be a non-CM elliptic curve of adelic level $m_E$. Let $n = n_1n_2$ where $n_1 = \gcd(n,m_E^\infty)$ and $(n_2,m_E) = 1$. By \eqref{Koblitzpart}, \eqref{KobAPproduct}, and Lemma \ref{deltaprimeellprime}, we have
\begin{equation*}
\begin{split}
        C^{\Prime}_{E,n,k} = \frac{\delta^{\Prime}_{E,n,k}(L)}{\prod_{\ell \mid L}(1-1/\ell)} &  \cdot \frac{1}{\phi(n_2)} \prod_{\substack{\ell \nmid m_E \\ \ell \mid n \\ \ell \nmid k-1}} \left( 1 - \frac{\ell^2+\ell}{|\GL_2(\Z/\ell\Z)|}\right) \prod_{\substack{\ell \nmid m_E \\ \ell \mid (n,k-1)}} \left( 1 - \frac{\ell}{|\GL_2(\Z/\ell\Z)|}\right) \\
    &\prod_{\ell \nmid nm_E}\left( 1 - \frac{\ell^2-\ell-1}{(\ell-1)^3(\ell+1)}\right).
\end{split}
\end{equation*}
We now propose the average counterpart of $C^{\Prime}_{E,n,k}$,
\begin{equation}\label{Cprimenk}
  C_{n,k}^{\Prime} \coloneqq \frac{1}{\phi(n)} \prod_{\substack{\ell \mid n \\ \ell \nmid k-1}} \left( 1 - \frac{\ell^2+\ell}{|\GL_2(\Z/\ell\Z)|}\right) \prod_{\ell \mid (n,k-1)} \left( 1 - \frac{\ell}{|\GL_2(\Z/\ell\Z)|}\right) \prod_{\ell \nmid n}\left( 1 - \frac{\ell^2-\ell-1}{(\ell-1)^3(\ell+1)}\right).
\end{equation}
The formula for $C_{n,k}^{\Prime}$ coincides with $C^{\Prime}_{E,n,k}$ if one takes $m_E = 1$, similar to the case for $C_{n,k}^{\Cyc}$ in \eqref{Ccycnk}. Parallel to Proposition \ref{Avgsumcyc},   we  show that $C_{n,k}^{\Prime}$ behaves as expected when we sum over $k$. 

\begin{Proposition} \label{AvgSum}
    For any positive integer $n$, we have
    $$\sum_{\substack{1 \leq k \leq n \\ (n,k) = 1}}C^{\Prime}_{n,k} = C^{\Prime},$$
    where $C^{\Prime}$ and $C^{\Prime}_{n,k}$ are defined in \eqref{Cprime} and \eqref{Cprimenk}.
\end{Proposition}
\begin{proof}
For notational convenience, we define
    $$f_1(\ell) = 1 - \frac{\ell^2+\ell}{|\GL_2(\Z/\ell\Z)|}, \hs f_2(\ell) = 1 - \frac{\ell}{|\GL_2(\Z/\ell\Z)|}$$
To show the desired equation, we need to verify that
\begin{equation}\label{verification}
    F(n) \coloneqq \frac{1}{\phi(n)} \sum_{\substack{1 \leq k \leq n \\ (k,n) = 1}} \prod_{\substack{\ell \mid n \\ k \not \equiv 1 (\ell)}} f_1(\ell) \prod_{\substack{\ell \mid n \\ k \equiv 1(\ell)}}f_2(\ell) = \prod_{\ell \mid n} \left( 1 - \frac{\ell^2-\ell-1}{(\ell-1)^3(\ell+1)}\right).
\end{equation}
First, we prove that \eqref{verification} is true for $n = p^a$, a prime power. Observe that
\begin{align*}
    F(p^a) &= \frac{1}{\phi(p^a)}\left(p^{a-1}(f_1(p)(p-2) + f_2(p))\right)\\
    &= \frac{p^{a-1}}{\phi(p^a)} \left[ (p-2)\left( 1 - \frac{p^2+p}{|\GL_2(\Z/p\Z)|}\right) + \left( 1 - \frac{p}{|\GL_2(\Z/p\Z)|}\right) \right] \\
    &= 1 - \frac{p^2-p-1}{(p-1)^3(p+1)}.
\end{align*}
Let us prove that $F$ is multiplicative. Let $n$ be coprime to $p^a$, a prime power. We see that
\begin{align*}
    F(p^an) &= \frac{1}{\phi(p^an)}\sum_{\substack{1 \leq k \leq p^a n \\ (k,pn) = 1}} \prod_{\substack{\ell \mid pn \\ k \not \equiv 1 (\ell)}}f_1(\ell) \prod_{\substack{\ell \mid pn \\ k \equiv 1 (\ell)}}f_2(\ell) \\
    &= \frac{1}{\phi(p^a)} \frac{1}{\phi(n)}\left[\sum_{\substack{1 \leq k \leq p^an \\ (k,pn) = 1 \\ k \not \equiv 1 (p)}}f_1(p) \prod_{\substack{\ell \mid n \\ k \not \equiv 1(\ell)}}f_1(\ell) \prod_{\substack{\ell \mid n \\ k \equiv 1 (\ell)}} f_2(\ell) + \sum_{\substack{1 \leq k \leq p^an \\ (k,pn) = 1 \\ k \equiv 1 (p)}} f_2(p) \prod_{\substack{\ell \mid n \\ k \not \equiv 1 (\ell)}} f_1(\ell) \prod_{\substack{\ell \mid n \\ k \equiv 1(\ell)}}f_2(\ell)\right] \\
    &= \frac{f_1(p) (p-2)p^{a-1}}{\phi(p^a)} \frac{1}{\phi(n)}\sum_{\substack{1 \leq k \leq n \\ (k,n) = 1}} \prod_{\substack{\ell \mid n \\ k \not \equiv 1 (\ell)}} f_1(\ell) \prod_{\substack{\ell \mid n \\ k \equiv 1 (\ell)}}f_2(\ell) \\
    &\hspace{1in} + \frac{f_2(p) p^{a-1}}{\phi(p^a)} \frac{1}{\phi(n)} \sum_{\substack{1 \leq k \leq n \\ (k,n) = 1}} \prod_{\substack{\ell \mid n \\ k \not \equiv 1 (\ell)}} f_1(\ell) \prod_{\substack{\ell \mid n \\ k \equiv 1 (\ell)}}f_2(\ell) \\
    &=\frac{1}{\phi(p^a)}\left(p^{a-1}(f_1(p)(p-2) + f_2(p))\right) \cdot  F(n) = F(p^a)F(n).
\end{align*}

This completes the proof.
\end{proof}

\subsection{Applying Zywina's approach for the cyclicity problem}\label{ZywArgCyc}
Zywina \cite{MR2805578} refined Koblitz's conjecture by developing a heuristic explanation for the constant $C^{\Prime}_E$. In essence, he interprets the desired property of a prime in terms of Galois representations, examines the ratio of elements meeting the property in each finite level $G_E(m)$, and considers the limit as $m$ approaches infinity. In this subsection, we employ Zywina's approach to determine the heuristic densities of primes of cyclic reduction for $E$ and verify their concurrence with the densities proposed by Serre and Akbal-G\"ulo\u{g}lu. In this subsection, we specifically concentrate on non-CM elliptic curves.

Let $E/\Q$ be an elliptic curve and fix a good prime $p \neq 2$. We now give a criterion for $p$ to be a prime of cyclic reduction for $E$.\footnote{Note that if $p = 2$ is a prime of good reduction for $E$, then $\widetilde{E}_2(\F_2)$ is necessarily cyclic.} Let $\Frob_p$ denote a Frobenius element in $\Gal(\overline{\Q}/\Q)$ at $p$. By \cite[Lemma 2.1]{MR2099195}, we have that
\begin{align*}
    \widetilde{E}_p(\F_p) \text{ is cyclic }
    & \iff \forall \text{ primes } \ell \neq p, \, \widetilde{E}_p(\F_p) \text{ does not contain a subgroup isomorphic to } \Z/\ell \Z \oplus \Z/\ell \Z \\
    & \iff \forall \text{ primes } \ell \neq p, \, \rho_{E,\ell}(\Frob_p) \not\equiv I \pmod{\ell} \\
    &\iff \forall m \in \N \text{ with } p \nmid m \text{ and } \forall \text{ prime } \ell \mid m, \, \rho_{E,\ell}(\Frob_p) \not\equiv I \pmod{\ell}.
\end{align*}
Drawing a parallel to \eqref{PsiPrime}, we consider the set
$$\Psi^{\Cyc}(m) \coloneqq \left\{M \in \GL_2(\Z/m\Z): M \not \equiv I \neghs \pmod \ell \text{ for all } \ell \mid m \right\},$$
and the ratio
$$\delta^{\Cyc}_E(m) \coloneqq \frac{|G_E(m) \cap \Psi^{\Cyc}(m)|}{|G_E(m)|}.$$ 
Taking the limit of $\delta^{\Cyc}_E(m)$ over all positive integers, ordered by divisibility, we expect to obtain the heuristic density of primes of cyclic reduction. 

\begin{Proposition}\label{cycheudeneulerproduct}
    Let $E/\Q$ be a non-CM elliptic curve of adelic level $m_E$. Then $\delta^{\Cyc}_{E}(\cdot)$, as an arithmetic function, satisfies the following properties:
    \begin{enumerate}
        \item for any positive integer $m$, $\delta^{\Cyc}_E(m) = \delta^{\Cyc}_E(\rad(m))$;
        \item for any prime $\ell \nmid m_E$ and integer $d$ coprime to $\ell$, $\delta^{\Cyc}_{E}(d\ell) = \delta^{\Cyc}_E(d) \cdot \delta^{\Cyc}_E(\ell)$.
    \end{enumerate}
    Therefore, the heuristic density of primes of cyclic reduction can be expressed as follows,
    \begin{equation*}
        \lim_{m \to \infty} \delta_E^{\Cyc}(m) = \delta_{E}^{\Cyc}(\rad(m_E)) \cdot \prod_{\ell \nmid m_E} \delta_E^{\Cyc}(\ell).
    \end{equation*}
\end{Proposition}
\begin{proof}
    Follows similarly to the proof of \propref{Koblitzeulerproduct}.
\end{proof}
\begin{Remark}\label{absconscyc}
    One can easily check that for $\ell \nmid m_E$,
    \begin{equation}\label{cycnoAP}
        \delta_E^{\Cyc}(\ell) = 1 - \frac{1}{|\GL_2(\Z/\ell\Z)|} \sim 1 - \frac{1}{\ell^4}, \hs \text{ as } \ell \to \infty
    \end{equation}
    and hence the infinite product converges absolutely. 
\end{Remark}
    
We now verify that the limit $\lim_{m \to \infty} \delta_E^{\Cyc}(m)$ appearing in Proposition \ref{cycheudeneulerproduct} coincides with the cyclicity constant $C^{\Cyc}_E$ originally defined by Serre \cite[pp.\ 465-468]{MR3223094}.

\begin{Proposition}\label{breakingdelta}
    Let $E/\Q$ be a non-CM elliptic curve. Then we have
    $$C^{\Cyc}_E = \delta_E^{\Cyc}(\rad(m_E)) \cdot \prod_{\ell \nmid m_E} \left( 1 - \frac{1}{|\GL_2(\Z/\ell\Z)|}\right).$$
\end{Proposition}
\begin{proof}
    Let $R = \rad(m_E)$. By \eqref{Ccycalmosteuler} and \eqref{cycnoAP}, it suffices to check
    $$ \sum_{d\mid m_E} \frac{\mu(d)}{[\Q(E[d]) : \Q]} = \delta^{\Cyc}_E(R).$$
    Let $m$ be a positive integer and $d \mid m$. We define
    \begin{align*}
        S'_E(m) &\coloneqq \{M \in G_E(m) : M \not \equiv I \neghs \pmod \ell \text{ for all } \ell \mid m \} \\
        S^{(d)}_E(m) & \coloneqq \{M \in G_E(m) : M \equiv I \neghs \pmod {d} \}.
    \end{align*}
    From the definition, one may observe that $G_E(R) \cap \Psi^{\Cyc}(R) = S'_E(R)$. Thus, we have
    $$\delta_E^{\Cyc}(R) = \frac{|S'_E(R)|}{|G_E(R)|}.$$
    Also, note that $S_E^{(d)}(d) = \{I\}$. Let $\varpi \colon  G_E(m) \to G_E(d)$ be the natural reduction map. Then, 
    $$\frac{|S^{(d)}_E(m)|}{|G_E(m)|} = \frac{\left|\varpi^{-1}(S^{(d)}_E(d))\right|}{\left|\varpi^{-1}(G_E(d))\right|} = \frac{\left|S_E^{(d)}(d)\right|}{\left|G_E(d)\right|} =  \frac{1}{|G_E(d)|}.$$
    Observe that $S'_E(R) = G_E(R) - \bigcup_{\ell \mid R} S^{(\ell)}_E(R)$. By the principle of inclusion-exclusion, we obtain
    $$\delta_E^{\Cyc}(R) = \frac{|S'_E(R)|}{|G_E(R)|} = \sum_{d \mid R} \frac{\mu(d)}{|G_E(d)|} = \sum_{d\mid m_E} \frac{\mu(d)}{[\Q(E[d]):\Q]}.$$
    This completes the proof.
\end{proof}
Now, let us construct a heuristic density of primes of cyclic reduction that lie in an arithmetic progression. Consider
$$\Psi^{\Cyc}_{n,k}(m) \coloneqq \left\{M \in \GL_2(\Z/m\Z) : M \not \equiv I \neghs \pmod \ell \text{ for all } \ell \mid m, \det M \equiv k \neghs \pmod {\gcd(m,n)} \right\}.$$
We define 
$$\delta_{E,n,k}^{\Cyc}(m) \coloneqq \frac{\left|G_E(m) \cap \Psi^{\Cyc}_{n,k}(m)\right|}{|G_E(m)|}.$$
Drawing parallels from Zywina's approach, we consider the limit
\begin{equation}\label{naiveheuden}
    \lim_{m \to \infty} \delta_{E,n,k}^{\Cyc}(m),
\end{equation}
where the limit is taken over all positive integers, ordered by divisibility. We'll prove in Proposition \ref{densitycheck} that \eqref{naiveheuden} coincides with $C_{E,n,k}^{\Cyc}$ as defined in \cite{MR4504664}. To do so, we'll first give some properties of $\delta_{E,n,k}^{\Cyc}(\cdot)$.

\begin{Proposition}\label{cycalmosteuler}
Let $E/\Q$ be a non-CM elliptic curve of adelic level $m_E$. Fix a positive integer $n$. Set $L$ as in \eqref{E:Def-L}. Then, $\delta_{E,n,k}^{\Cyc}(\cdot)$, as an arithmetic function, satisfies the following properties:
    \begin{enumerate}
        \item Let $L \mid L' \mid L^\infty$. Then, $\delta_{E,n,k}^{\Cyc}(L) = \delta_{E,n,k}^{\Cyc}(L')$;
        \item Let $\ell^\alpha$ be a prime power and $d$ be a positive integer with $(\ell, Ld) = 1$. Then, $\delta_{E,n,k}^{\Cyc}(d\ell^a) = \delta_{E,n,k}^{\Cyc}(d) \cdot \delta_{E,n,k}^{\Cyc}(\ell^\alpha)$.
        \item Let $\ell^\alpha \parallel n$ and $(\ell, L) = 1$. Then, for any $\beta \geq \alpha$, $\delta^{\Cyc}_{E,n,k}(\ell^\beta) = \delta^{\Cyc}_{E,n,k}(\ell^\alpha)$. Further, if $\ell \nmid nL$, we have $\delta_{E,n,k}^{\Cyc}(\ell^\beta) = \delta^{\Cyc}_E(\ell)$.
    \end{enumerate}
    Therefore, \eqref{naiveheuden} can be expressed as follows,
    \begin{equation}\label{cyclicheuden}
        \lim_{m \to \infty} \delta_{E,n,k}^{\Cyc}(m) = \delta^{\Cyc}_{E,n,k}(L) \cdot \prod_{\substack{\ell \nmid m_E \\ \ell^{\alpha} \parallel n}} \delta_{E,n,k}^{\Cyc}(\ell^\alpha) \cdot \prod_{\ell \nmid nm_E} \delta^{\Cyc}_{E}(\ell).
    \end{equation}
    and the product converges absolutely.
\end{Proposition}
\begin{proof}
    One can argue similarly to the proof of \propref{KoblitzAPeulerproduct} to obtain the desired results. The absolute convergence of \eqref{cyclicheuden} follows from \remref{absconscyc}.
\end{proof}
\begin{Lemma}\label{Sizeofdeltacyc}
    Let $E/\Q$ be a non-CM elliptic curve of adelic level $m_E$. Suppose $\ell^a \parallel n$ and $\ell \nmid m_E$. For any $k$ coprime to $n$, we have
    $$\delta^{\Cyc}_{E,n,k}(\ell^a) = \begin{cases}
        \frac{1}{\phi(\ell^a)} & \text{ if } \ell \mid n \text{ and } \ell \nmid (k-1) \\
        \frac{1}{\phi(\ell^a)}\left ( 1 - \frac{\phi(\ell)}{|\GL_2(\Z/\ell\Z)|}\right) & \text{ if } \ell \mid (n,k-1).
    \end{cases}$$
\end{Lemma}
\begin{proof}
    Since $\ell \nmid m_E$, we have $G_E(\ell^a) \simeq \GL_2(\Z/\ell^a\Z)$, and hence $|G_E(\ell^a)| = (\ell^2-1)(\ell^2-\ell)\ell^{4(a-1)}$. Applying \corref{countingcyc}, we obtain the desired results.
\end{proof}
Let $E/\Q$ be a non-CM elliptic curve of adelic level $m_E$. Let $n = n_1n_2$ where $n_1 = \gcd(n,m_E^\infty)$ and $(n_2,m_E) = 1$. By \eqref{cycnoAP}, \eqref{cyclicheuden}, and Lemma \ref{Sizeofdeltacyc}, we obtain
\begin{equation}\label{refinedcycheuden}
    \lim_{m\to \infty} \delta^{\Cyc}_{E,n,k}(m) = \frac{\delta_{E,n,k}^{\Cyc}(L)}{\phi(n_2)} \prod_{\substack{\ell \nmid m_E \\ \ell \mid (n,k-1)}}\left( 1- \frac{\phi(\ell)}{|\GL_2(\Z/\ell\Z)|}\right) \prod_{\ell \nmid nm_E} \left( 1 - \frac{1}{|\GL_2(\Z/\ell\Z)|}\right).
\end{equation}
We now prove that \eqref{refinedcycheuden} equals the cyclicity constant proposed by Akbal and G\"ul\u{g}lu.

\begin{Proposition} \label{densitycheck}
    Let $E/\Q$ be a non-CM elliptic curve of adelic level $m_E$ and $n$ be a positive integer. Let $n = n_1n_2$ where $n_1 = \gcd(n,m_E^\infty)$ and $(n_2,m_E) = 1$. Then we have 
    \begin{equation*}
        C^{\Cyc}_{E,n,k} = \frac{\delta_{E,n,k}^{\Cyc}(L)}{\phi(n_2)} \prod_{\substack{\ell \nmid m_E \\ \ell \mid (n,k-1)}}\left( 1- \frac{\phi(\ell)}{|\GL_2(\Z/\ell\Z)|}\right) \prod_{\ell \nmid nm_E} \left( 1 - \frac{1}{|\GL_2(\Z/\ell\Z)|}\right).
    \end{equation*}
\end{Proposition}
\begin{proof}
    Define
    $$S'_{E,n,k}(m) \coloneqq \{ \sigma \in \Gal(\Q(E[m])\Q(\zeta_n)/\Q) : \sigma|_{\Q(\zeta_n)} = \sigma_k, \sigma|_{\Q(E[\ell])} \not \equiv 1 \text{ for all } \ell \mid m\}.$$
    Let $R = \rad(m_E)$. By \cite[p.\ 13]{Jones-Lee}, \eqref{CCycEnk} can be expressed as follows,
    \begin{equation}\label{goodidea}
        \frac{|S'_{E,n,k}(R)|}{|\Gal(\Q(E[R]) \Q(\zeta_n)/\Q)|} \prod_{\substack{\ell \nmid m_E \\ \ell (n,k-1)}} \left(1-\frac{\phi(\ell)}{|\GL_2(\Z/\ell\Z)|}\right) \prod_{\ell \nmid nm_E} \left(1-\frac{1}{|\GL_2(\Z/\ell\Z)|}\right).
    \end{equation}
    As usual, let $L$ be as in \eqref{E:Def-L}; in particular, $n_1$ and $R$ divide $L$. It suffices to verify that
    $$\frac{|S'_{E,n,k}(R)|}{|\Gal(\Q(E[R])\Q(\zeta_n))/\Q|} = \frac{\delta_{E,n,k}^{\Cyc}(L)}{\phi(n_2)}.$$
    By the Weil pairing, we have $\Q(\zeta_{n_2}) \subseteq \Q(E[n_2])$. Thus, we see that $\Q(E[R])\Q(\zeta_{n_1})$ and $\Q(\zeta_{n_2})$ must be linearly disjoint by \lemref{propertyofmE}, and hence
    $$\Gal(\Q(E[R])\Q(\zeta_n)/\Q) \simeq \Gal(\Q(E[R]\Q(\zeta_{n_1}))/\Q) \times \Gal(\Q(\zeta_{n_2})/\Q).$$
    Under the isomorphism, the set $S'_{E,n,k}(R)$ can be identified as $S'_{E,n_1,k}(R) \times \left\{\sigma_k\right\}$, and hence $|S'_{E,n,k}(R)| = |S'_{E,n_1,k}(R)|$. Thus, we have
    $$\frac{|S'_{E,n,k}(R)|}{|\Gal(\Q(E[R])\Q(\zeta_n)/\Q)|} = \frac{1}{\phi(n_2)} \cdot \frac{|S'_{E,n_1,k}(R)|}{|\Gal(\Q(E[R])\Q(\zeta_{n_1})/\Q)|}.$$
    Remark that $\Q(E[R]) \subseteq \Q(E[L])$ and $\Q(\zeta_{n_1}) \subseteq \Q(E[L])$ by the definition of $L$. Thus, the usual restriction $\varpi \colon G_E(L) \to \Gal(\Q(E[R])\Q(\zeta_{n_1})/\Q)$ gives a surjective group homomorphism.
    
    Viewing $G_E(L)$ as a subgroup of $\GL_2(\Z/L\Z)$, we may observe that
    \begin{align*}
        \varpi^{-1}(S'_{E,n_1,k}(R)) &= \left\{\widetilde{\sigma} \in G_E(L) : \widetilde{\sigma}|_{\Q(\zeta_{n_1})} = \sigma_{k}, \widetilde{\sigma}|_{\Q(E[\ell])} \not \equiv 1 \neghs \pmod {\ell} \text{ for all } \ell \mid L \right\} \\
        &= G_E(L) \cap \left\{M \in \GL_2(\Z/L\Z) : \det M \equiv k \neghs \pmod {n_1}, M \not \equiv I \neghs \pmod {\ell} \text{ for all } \ell \mid L \right\} \\
        &= G_E(L) \cap \Psi^{\Cyc}_{n,k}(L).
    \end{align*}
    Therefore,
    $$\frac{|S'_{E,n_1,k}(R)|}{|\Gal(\Q(E[R])\Q(\zeta_{n_1})/\Q)|} = \frac{\left|\varpi^{-1}\left(S'_{E,n_1,k}(R)\right)\right|}{\left|\varpi^{-1}(\Gal(\Q(E[R])\Q(\zeta_{n_1})/\Q))\right|} = \frac{\left| G_E(L)\cap \Psi^{\Cyc}_{n,k}(L)\right|}{\left|G_E(L)\right|} = \delta^{\Cyc}_{E,n,k}(L).$$
    This completes the proof.
\end{proof}

\begin{Remark}
    As one may have observed from Conjecture \ref{Cyclicity} and Conjecture \ref{Koblitz}, the conjectural growth rates of $\pi_E^{\Cyc}(x)$ and $\pi_E^{\Prime}(x)$ are different. Thus, there is an intrinsic difference between $C^{\Cyc}_E$ and $C^{\Prime}_E$. In particular, $C^{\Cyc}_E$ can be interpreted as the (conjectural) density of primes of cyclic reduction for $E$ whereas $C^{\Prime}_E$ should not be interpreted analogously. A similar remark holds for $\pi_E^{\Cyc}(x;n,k)$ and $\pi_E^{\Prime}(x;n,k)$ and their respective constants.
\end{Remark}

\section{On the cyclicity and Koblitz constants for Serre curves} \label{CycKobConSeCur}
We begin by fixing some notation that will hold throughout the section. Let $E/\Q$ be a Serre curve of discriminant $\Delta_E$, $n$ be a positive integer, and $k$ be an integer coprime to $n$. Let $\Delta'$ be the squarefree part of $\Delta_E$. By \propref{mEofSerreCurves}, we have 
\begin{equation*}
    m_E = \begin{cases}
    2|\Delta'| & \text{ if } \Delta' \equiv 1 \pmod 4, \\
    4|\Delta'| & \text{ otherwise}.
\end{cases}
\end{equation*}
Let $L$ be defined as in \eqref{E:Def-L}. The goal of this section is to develop formulas for $C_{E,n,k}^{\Cyc}$ and $C_{E,n,k}^{\Prime}$ with our assumption that $E$ is a Serre curve. By Proposition \ref{KoblitzAPeulerproduct} and Proposition \ref{cycalmosteuler}, it suffices to compute $\delta_{E,n,k}^{\Cyc}(L)$ and $\delta_{E,n,k}^{\Prime}(L)$.

For an integer $n$, we set $n = n_1n_2$ where $n_1 = (n,m_E^\infty)$ and $(n_2,m_E) = 1$. There are two cases to consider: $m_E \nmid L$ and $m_E \mid L$.  The former happens if and only if one of the following holds:
\begin{itemize}
    \item $\Delta' \equiv 3 \pmod 4$ and $2 \nmid n$;
    \item $\Delta' \equiv 2 \pmod 4$ and $4 \nmid n$.
\end{itemize}
We write $L = 2^{\alpha} \cdot L^{\odd}$ where $L^{\odd}$ is an odd integer; observe that $|\Delta'|$ divides $L^{\odd}$. We now define two sign functions that depend on $\Delta', k$ and appear in Theorem \ref{serre-constants}.
\begin{Definition}\label{Definitionoftau}
   Assume $m_E\mid L$. We define $\tau = \tau(\Delta',k)$ as follows.
\begin{itemize}
    \item If $\Delta' \equiv 1 \pmod 4$, we define $\tau = -1$, regardless of the choice of $k$.
    \item If $\Delta' \equiv 3 \pmod 4$, then $4 \mid n$. We define
    $$\tau = \begin{cases}
        -1 & \text{ if } k \equiv 1 \pmod 4,\\
        1 & \text{ if } k \equiv 3 \pmod 4.
    \end{cases}$$
    \item If $\Delta' \equiv 2 \pmod 8$, then $8 \mid n$. We define
    $$\tau = \begin{cases}
        -1 & \text{ if } k \equiv 1,7 \pmod 8, \\
        1 & \text{ if } k \equiv 3 ,5 \pmod 8.
    \end{cases}$$
    \item If $\Delta' \equiv 6 \pmod 8$, then $8 \mid n$. We define
    $$\tau = \begin{cases}
        -1 & \text{ if } k \equiv 1, 3\pmod 8, \\
        1 & \text{ if } k \equiv 5,7 \pmod 8.
    \end{cases}$$
\end{itemize}
 Finally, we define $\tau^{\mathcal{X}} \coloneqq \tau^{\mathcal{X}}(\Delta',n,k) \in \{\pm 1\}$ as follows,
\begin{align*}
    \tau^{\Cyc} &\coloneqq \tau \prod_{\substack{\ell \mid L^{\odd} \\ \ell \nmid n}} (-1) \prod_{\substack{\ell \mid (n,L^{\odd}) \\ \ell \nmid k-1}}\left(\frac{k}{\ell}\right), \\ 
    \tau^{\Prime} &\coloneqq -\tau \prod_{\substack{\ell \mid (n,L^{\odd}) \\ \ell \nmid k-1}}\left(\frac{k}{\ell}\right).
\end{align*}
\end{Definition}
Having defined $\tau^{\Cyc}$ and $\tau^{\Prime}$, the rest of the section is devoted to proving \thmref{serre-constants}. First, suppose $m_E \nmid L$. Then, by \propref{mEofSerreCurves}, we have $G_E(L) \simeq \GL_2(\Z/L\Z) \simeq \prod_{\ell^{\alpha} \parallel L} \GL_2(\Z/\ell^{\alpha}\Z)$. One can check that the isomorphism induces bijections between the sets,
\begin{align}
\Psi^{\mathcal{X}}_{n,k}(L) & \; \text{ and } \prod_{\ell^{\alpha}\parallel L} \Psi^{\mathcal{X}}_{n,k}(\ell^{\alpha}), \label{bijection-1}
\end{align}
for $\mathcal{X} \in \{\Cyc,\Prime\}$. For $\ell \mid L$, either we have $\ell \nmid n$ or $\ell \mid n$.
Recall from \eqref{E:Def-L} that $\ell \nmid n \implies \alpha = 1$. The condition $\det M \equiv k \pmod {\gcd(n,\ell)}$ becomes trivial, and hence we have
$$ \Psi^{\mathcal{X}}_{n,k}(\ell^\alpha) = \Psi^{\mathcal{X}}(\ell)$$
for $\mathcal{X} \in \{\Cyc,\Prime\}$. Suppose $\ell^\alpha \parallel n$. We have already determined the size of $\Psi^{\mathcal{X}}_{n,k}(\ell^\alpha)$ in \corref{countingcyc} and \corref{countingprime}. Based on those counts, we obtain the following.
\begin{Lemma}\label{psi-L}
We have 
\begin{enumerate}
\item \label{psi-L-cyc} $\displaystyle \Psi_{n, k}^{\Cyc}(L)=\prod_{\substack{\ell\mid L\\\ell \nmid n}}\left((\ell^2-1)(\ell^2-\ell)-1 \right) \prod_{\substack{\ell^{\alpha}\parallel (L, n)\\ \ell\mid k-1}}\left(\ell^{3(\alpha-1)}(\ell^3-\ell-1) \right) \prod_{\substack{\ell^{\alpha}\parallel (L, n)\\ \ell\nmid k-1}}\left(\ell^{3(\alpha-1)}(\ell^3-\ell) \right)$.
\item \label{psi-L-prime} $\displaystyle \Psi_{n, k}^{\Prime}(L)=\prod_{\substack{\ell\mid L\\\ell \nmid n}}\left(\ell (\ell^3-2\ell^2-\ell+3)\right) \prod_{\substack{\ell^{\alpha}\parallel (L, n)\\ \ell\mid k-1}}\left(\ell^{3(\alpha-1)}(\ell^3-\ell^2-\ell) \right)\prod_{\substack{\ell^{\alpha}\parallel (L, n)\\ \ell\nmid k-1}}\left(\ell^{3(\alpha-1)}(\ell^3-\ell^2-2\ell) \right)$.
\end{enumerate}
\end{Lemma}
Applying \lemref{psi-L}.\eqref{psi-L-cyc}, we obtain
\begin{equation}\label{compu1}
\begin{split}
    \delta_{E,n,k}^{\Cyc}(L) &= \prod_{\substack{\ell^{\alpha} \parallel (n,L) \\ \ell\mid k-1 }} \frac{\ell^{3(\alpha-1)} (\ell^3-\ell-1)}{|\GL_2(\Z/\ell^{\alpha}\Z)|} \cdot \prod_{\substack{\ell^{\alpha} \parallel (n,L) \\ \ell\nmid k-1 }} \frac{\ell^{3(\alpha-1)}(\ell^3-\ell)}{|\GL_2(\Z/\ell^{\alpha}\Z)|} \cdot \prod_{\substack{\ell \mid L \\ \ell \nmid n}} \left( 1 - \frac{1}{|\GL_2(\Z/\ell\Z)|}\right) \\
    &= \prod_{\ell^{\alpha}\parallel (n,L)} \frac{1}{\ell^{\alpha-1}} \cdot \prod_{\substack{\ell \mid (n,L) \\ \ell\mid k-1 }}\left( \frac{1}{\ell-1} - \frac{1}{|\GL_2(\Z/\ell\Z)|}\right) \prod_{\substack{\ell \mid (n,L) \\ \ell\nmid k-1}} \left(\frac{1}{\ell-1}\right) \cdot \prod_{\substack{\ell \mid L \\ \ell \nmid n}} \left( 1 - \frac{1}{|\GL_2(\Z/\ell\Z)|}\right)  \\
    &= \frac{1}{\phi(n_1)} \prod_{\substack{\ell \mid (n,L) \\ \ell \mid k-1}}\left(1-\frac{\phi(\ell)}{|\GL_2(\Z/\ell\Z)|}\right) \cdot \prod_{\substack{\ell \mid L \\ \ell \nmid n}} \left( 1 - \frac{1}{|\GL_2(\Z/\ell\Z)|}\right) .
    \end{split}
\end{equation}

Thus, \eqref{refinedcycheuden} and \eqref{compu1} give
\begin{align*}
    C^{\Cyc}_{E,n,k} = & \ \frac{1}{\phi(n_1)} \prod_{\substack{\ell \mid L \\ k \equiv 1 (\ell)}}\left(1-\frac{\phi(\ell)}{|\GL_2(\Z/\ell\Z)|}\right) \cdot \prod_{\substack{\ell \mid L \\ \ell \nmid n}} \left( 1 - \frac{1}{|\GL_2(\Z/\ell\Z)|}\right) \\
    \cdot & \ \frac{1}{\phi(n_2)} \prod_{\substack{\ell \nmid m_E \\ \ell \mid (n,k-1) }}\left(1- \frac{\phi(\ell)}{|\GL_2(\Z/\ell\Z)|}\right) \prod_{\ell \nmid nm_E} \left( 1 - \frac{1}{|\GL_2(\Z/\ell\Z)|}\right) \\
    =  & \ \frac{1}{\phi(n)} \prod_{\substack{\ell \mid (n,k-1)}} \left( 1 - \frac{\phi(\ell)}{|\GL_2(\Z/\ell\Z)|}\right) \prod_{\ell \nmid n} \left( 1 - \frac{1}{|\GL_2(\Z/\ell\Z)|}\right) = C^{\Cyc}_{n,k}.
\end{align*}
So we obtain $C^{\Cyc}_{E,n,k} = C^{\Cyc}_{n,k}$ if $m_E\nmid L$.

Similarly, for the Koblitz case, applying \lemref{psi-L}.\eqref{psi-L-prime}, we see 
\begin{equation}\label{compu2}
\begin{split}
    \frac{\delta_{E,n,k}^{\Prime}(L)}{\prod_{\ell \mid L}(1-1/\ell)} &= \prod_{\substack{\ell^{\alpha} \parallel (n,L) \\ \ell\nmid k-1}} \frac{\ell \cdot \ell^{3(\alpha-1)} \cdot (\ell^3-\ell^2-2\ell)}{(\ell-1) |\GL_2(\Z/\ell^{\alpha}\Z)|}  \prod_{\substack{\ell^{\alpha} \parallel (n,L) \\ \ell\mid k-1}} \frac{\ell \cdot \ell^{3(\alpha-1)}\cdot(\ell^3-\ell^2-\ell)}{(\ell-1)|\GL_2(\Z/\ell^{\alpha}\Z)|} \prod_{\substack{\ell \mid L \\ \ell \nmid n}} \left( 1 - \frac{\ell^2-\ell-1}{(\ell-1)^3(\ell+1)} \right) \\
    &= \frac{1}{\phi(n_1)} \prod_{\substack{\ell \mid (n,L) \\\ell\nmid k-1}} \left(1 - \frac{\ell^2+\ell}{|\GL_2(\Z/\ell\Z)|}\right)\prod_{\substack{\ell \mid (n,L) \\\ell\mid k-1}}\left( 1 - \frac{\ell}{|\GL_2(\Z/\ell\Z)|}\right) \prod_{\substack{\ell \mid L \\ \ell \nmid n}} \left( 1 - \frac{\ell^2-\ell-1}{(\ell-1)^3(\ell+1)} \right).
    \end{split}
\end{equation}
Now by \propref{KoblitzAPeulerproduct}, \lemref{deltaprimeellprime}, and \eqref{compu2}, we obtain
\begin{align*}
    C^{\Prime}_{E,n,k} = &\frac{1}{\phi(n_1)} \prod_{\substack{\ell \mid (n,L) \\ \ell\mid k-1}}\left( 1 - \frac{\ell}{|\GL_2(\Z/\ell\Z)|}\right) \prod_{\substack{\ell \mid (n,L) \\ \ell\nmid k-1}} \left(1 - \frac{\ell^2+\ell}{|\GL_2(\Z/\ell\Z)|}\right) \prod_{\substack{\ell \mid L \\ \ell \nmid n}} \left( 1 - \frac{\ell^2-\ell-1}{(\ell-1)^3(\ell+1)} \right) \\
    \cdot &\frac{1}{\phi(n_2)} \prod_{\substack{\ell \nmid m_E \\ \ell \mid n \\ \ell\mid k-1}} \left(1 - \frac{\ell}{|\GL_2(\Z/\ell\Z)|}\right) \prod_{\substack{\ell \nmid m_E \\ \ell \mid n \\\ell\nmid k-1}}\left(1-\frac{\ell^2+\ell}{|\GL_2(\Z/\ell\Z)|}\right) \prod_{\ell \nmid nm_E}\left(1-\frac{\ell^2-\ell-1}{(\ell-1)^3(\ell+1)}\right) \\
    =& C^{\Prime}_{n,k}.
\end{align*}
This completes the proof of the theorem for the case where $m_E \nmid L$.

Now, suppose that $m_E\mid L$. This case is a bit more involved. Recall from \sectionref{SeCur} the definition of $\psi_{\ell^{\alpha}}$ and the fact that $G_E(L) = \ker \psi_L$. By \cite[Lemma 16]{MR2534114} and  (\ref{bijection-1}) we have
\begin{equation}\label{estimation-num}
\left|G_E(L)\cap \Psi^{\mathcal{X}}_{n,k}(L)\right|=
\frac{1}{2}\left(\left|\Psi^{\mathcal{X}}_{n,k}(L)\right| +\prod_{\ell^{\alpha}\parallel L}\left(\left|Y^{\mathcal{X}}_{\ell^{\alpha}, +}\right|-\left|Y^{\mathcal{X}}_{\ell^{\alpha}, -}\right|\right)\right),
\end{equation}
for $\mathcal{X}\in \{\Cyc, \Prime\}$, where
\begin{align*}
Y^{\Cyc}_{\ell^{\alpha}, \pm} &\coloneqq
   \left\{M \in \GL_2(\Z/\ell^{\alpha} \Z) : \psi_{\ell^{\alpha}}(M)= \pm 1, M \not \equiv I \neghs \pmod \ell, \det M \equiv k \neghs \pmod {\gcd(\ell^{\alpha}, n)} \right\}, \\
Y^{\Prime}_{\ell^{\alpha}, \pm} &\coloneqq
   \left\{M \in \GL_2(\Z/\ell^{\alpha} \Z) : \psi_{\ell^{\alpha}}(M)= \pm 1, \det(M-I) \not \equiv 0 \neghs \pmod \ell, \det M \equiv k \neghs \pmod {\gcd(\ell^{\alpha},n)} \right\}.
\end{align*}
The sets $Y^{\Cyc}_{\ell^{\alpha}, +}$, $Y^{\Cyc}_{\ell^{\alpha}, -}$, $Y^{\Prime}_{\ell^{\alpha}, +}$, and $Y^{\Prime}_{\ell^{\alpha}, -}$ all depend on $n$ and $k$, though we do not include this dependence in the notation for brevity. We first focus on the size of $|Y^{\mathcal{X}}_{\ell^\alpha,+}| - |Y^{\mathcal{X}}_{\ell^\alpha,-}|$ for primes $\ell$ dividing $L^{\odd}$.

\begin{Lemma}\label{Y-ell}
We have 
\begin{enumerate}
    \item \label{Y-ell-cyc}
    \begin{align*}
    \displaystyle\prod_{\substack{\ell^{\alpha}\parallel L^{\odd}}}\left(|Y^{\Cyc}_{\ell^{\alpha}, +}|-|Y^{\Cyc}_{\ell^{\alpha}, -}| \right) & = \\
    & \hspace{-.5cm} \prod_{\substack{\ell \mid L^{\odd}\\ \ell \nmid n}} (-1)\prod_{\substack{\ell^{\alpha}\parallel (n, L^{\odd})\\ \ell\mid k-1}} \ell^{3(\alpha-1)}(\ell^3-\ell-1) \prod_{\substack{\ell^{\alpha}\parallel (n, L^{\odd})\\ \ell\nmid k-1}}\left(\frac{k}{\ell}\right)\ell^{3(\alpha-1)}(\ell^3-\ell).
      \end{align*}
    \item \label{Y-ell-prime}
    \begin{align*}
    \displaystyle\prod_{\substack{\ell^{\alpha}\parallel L^{\odd}}}\left(|Y^{\Prime}_{\ell^{\alpha}, +}|-|Y^{\Prime}_{\ell^{\alpha}, -}| \right) & = \\
    &  \hspace{-1cm}\prod_{\substack{\ell \mid L^{\odd}\\ \ell \nmid n}} \ell \prod_{\substack{\ell^{\alpha}\parallel (n, L^{\odd})\\ \ell\mid k-1}} \ell^{3(\alpha-1)}(\ell^3-\ell^2-\ell) \prod_{\substack{\ell^{\alpha}\parallel (n, L^{\odd})\\ \ell\nmid k-1}}\left(\frac{k}{\ell}\right)\ell^{3(\alpha-1)}(\ell^3-\ell^2-2\ell).
    \end{align*}
\end{enumerate}
\end{Lemma}
\begin{proof}
From the definition of $\psi_{\ell^{\alpha}}$ for an odd prime $\ell\mid L$, we have  
\begin{align*}
    \left|Y^{\Cyc}_{\ell^{\alpha}, \pm}\right| &=\#\left\{M \in \GL_2(\Z/\ell^{\alpha} \Z) : \left(\frac{\det M}{\ell}\right)  = \pm 1, M \not \equiv I \neghs \pmod \ell, \det M \equiv k \neghs \pmod {\ell^{\alpha}} \right\}, \\
    \left|Y^{\Prime}_{{\ell^\alpha},\pm}\right| &= \#\left\{M \in \GL_2(\Z/\ell^\alpha\Z) : \left( \frac{\det M}{\ell}\right) = \pm 1, \det(M-I) \not \equiv 0 \pmod \ell, \det M \equiv k \neghs \pmod {\ell^\alpha} \right\}.
\end{align*}

By Corollaries \ref{countingcyc} and \ref{countingprime}, it is an easy exercise to check that
\begin{align*}
    \left|Y^{\Cyc}_{\ell^{\alpha},+}\right| &= \begin{cases}
        \ell^{3(\alpha-1)} \left(\ell^3-\ell-1 \right) & \text{ if } \ell\mid n \text{ and } k \equiv 1 \pmod {\ell}, \\
        \ell^{3(\alpha-1)}(\ell^3-\ell) & \text{ if } \ell\mid n, k \not \equiv 1 \pmod {\ell}, \text{ and } \left(\frac{k}{\ell}\right) = 1, \\
        0 & \text{ if }  \ell\mid n, \left(\frac{k}{\ell}\right) = -1, \\
        \frac{(\ell^2-\ell)(\ell^2-1)}{2}-1 & \text{ if } \ell \nmid n,
    \end{cases} \\
    \left| Y^{\Cyc}_{\ell^{\alpha},-}\right| &= \begin{cases}
        \ell^{3(\alpha-1)}(\ell^3-\ell) & \text{ if } \ell\mid n \text{ and } \left(\frac{k}{\ell}\right) = -1, \\
        0 & \text{ if } \ell\mid n \text{ and } \left(\frac{k}{\ell}\right) = 1,\\
        \frac{(\ell^2-\ell)(\ell^2-1)}{2} & \text{ if }  \ell \nmid n,
    \end{cases} \\
    \left|Y^{\Prime}_{\ell^{\alpha},+}\right| & = \begin{cases}
        \ell^{3(\alpha-1)} (\ell^3-\ell^2-\ell) & \text{ if } \ell\mid n\text{ and } k \equiv 1 \pmod \ell \\
        \ell^{3(\alpha-1)} (\ell^3-\ell^2-2\ell) & \text{ if } \ell\mid n \text{ and } k \not \equiv 1 \pmod \ell, \text{ and } \left(\frac{k}{\ell}\right) = 1, \\
        0 & \text{ if } \ell\mid n \text{ and } \left(\frac{k}{\ell}\right) = - 1, \\
        \frac{(\ell-1)(\ell^3-\ell^2-2\ell)}{2} + \ell & \text{ if } \ell \nmid n,
    \end{cases} \\
    \left|Y^{\Prime}_{\ell^{\alpha}, -}\right| &= \begin{cases}
        \ell^{3(\alpha-1)} (\ell^3-\ell^2-2\ell) & \text{ if } \ell\mid n \text{ and } \left(\frac{k}{\ell}\right) = - 1, \\
        0 & \text{ if }\ell^\mid n \text{ and } \left(\frac{k}{\ell}\right) = 1, \\
        \frac{(\ell-1)(\ell^3-\ell^2-2\ell)}{2} & \text{ if }  \ell \nmid n.
        \end{cases}
\end{align*}
The result now follows from some simple computations.
\end{proof}

Finally, we evaluate  $|Y^{\mathcal{X}}_{\ell^{\alpha}, +}|-|Y^{\mathcal{X}}_{\ell^{\alpha}, -}|$ when $\ell=2$. 
Suppose $\Delta' \not \equiv 1 \pmod 4$. Then we have $4 \mid m_E$. Since we are in the case that $m_E \mid L$, we must have $4 \mid n$. In particular, $n$ is even, and hence $k$ must be odd. On the other hand, if $\Delta' \equiv 1 \pmod 4$, then $n$ may not be even.

\begin{Lemma}\label{Y-L-2}
For fixed $\Delta'$ and $k$, let $\tau$ be defined  as in Definition \ref{Definitionoftau}.
Then 
\begin{enumerate}
\item $|Y^{\Cyc}_{2^\alpha,+}| - |Y^{\Cyc}_{2^{\alpha},-}| = \tau \cdot 2^{3(\alpha-1)}$
\item $|Y^{\Prime}_{2^{\alpha},+}| - |Y^{\Prime}_{2^{\alpha},-}| = -(2\tau)  \cdot  2^{3(\alpha-1)}.$
\end{enumerate}
\end{Lemma}
\begin{proof}
If $\Delta'\equiv 1\pmod 4$, then by the definition of $\psi_{2^{\alpha}}(\cdot)$, 
\begin{align*}
    Y^{\Cyc}_{2^{\alpha},\pm} &= \left\{M \in \GL_2(\Z/2^\alpha\Z) : \epsilon(M_2) = \pm 1, M \not \equiv I \neghs \pmod 2, \det M \equiv k \neghs \pmod {2^{\alpha}}\right\}, \\
     Y^{\Prime}_{2^{\alpha},\pm} &= \left\{M \in \GL_2(\Z/2^\alpha\Z) : \epsilon(M_2) = \pm 1, \det(M-I) \not \equiv 0 \neghs \pmod 2, \det M \equiv k \neghs \pmod {2^{\alpha}}\right\}.
\end{align*}

 Let $h^{\Cyc}_{\pm} = |Y^{\Cyc}_{2,\pm}|$.  In the case where $\alpha = 1$, it is clear that $h^{\Cyc}_+=2$ and  $h^{\Cyc}_- =3$.
For $\alpha \geq 2$, by \lemref{usefullemma}, we obtain
\begin{equation*}
    \left|Y^{\Cyc}_{2^{\alpha},\pm}\right| = h^{\Cyc}_\pm \cdot 2^{3(\alpha-1)}
\end{equation*}
and hence $|Y^{\Cyc}_{2^\alpha,+}| - |Y^{\Cyc}_{2^\alpha,-}| = -2^{3(\alpha-1)}$.

Let us check the size of $Y^{\Prime}_{2^{\alpha},\pm}$.  
In the case where $\alpha = 1$, we have that
\[
Y^{\Prime}_{2,+} = \left\{ \begin{pmatrix}
    1 & 1 \\ 1 & 0
\end{pmatrix},
\begin{pmatrix}
    0 & 1 \\ 1 & 1
\end{pmatrix}
\right\}
\quad
\text{and}
\quad
Y^{\Prime}_{2,-} = \emptyset.
\]
Setting 
$h^{\Prime}_{\pm} \coloneqq |Y^{\Prime}_{2,\pm}|$, we see that $h^{\Prime}_+ = 2$ and $h^{\Prime}_- =0$.  By \lemref{usefullemma}, we obtain
\[
  \left|Y^{\Prime}_{2^{\alpha},\pm}\right| = h^{\Prime}_\pm \cdot 2^{3(\alpha-1)}
\]
and hence  
$|Y^{\Prime}_{2^{\alpha},+}| -|Y^{\Prime}_{2^{\alpha},-}|  = 2 \cdot 2^{3(\alpha-1)}.$

Next, we assume $\Delta'\equiv 3\pmod 4$. Then, by the definition of $\psi_{2^{\alpha}}(\cdot)$, we have 
\begin{align*}
 Y^{\Cyc}_{2^{\alpha},\pm} &= \left\{M \in \GL_2(\Z/2^\alpha\Z) : \epsilon(M_2)\chi_4(k) = \pm 1, M \not \equiv I \neghs \pmod 2, \det M \equiv k \neghs \pmod {2^{\alpha}}\right\}, \\
 Y^{\Prime}_{2^{\alpha},\pm} &= \left\{M \in \GL_2(\Z/2^\alpha\Z) : \epsilon(M_2)\chi_4(k) = \pm 1, \det(M-I) \not \equiv 0 \neghs \pmod 2, \det M \equiv k \neghs \pmod {2^{\alpha}}\right\}.
\end{align*}
Then 
\begin{align*}
  |Y^{\Cyc}_{2^{\alpha},+}|-|Y^{\Cyc}_{2^{\alpha},-}|&=
 \begin{cases}
    -2^{3(\alpha-1)} & \text{ if } k  \equiv 1 \pmod 4, \\
    2^{3(\alpha-1)} &  \text{ if } k  \equiv 3 \pmod 4.
        \end{cases} \\
 |Y^{\Prime}_{2^{\alpha},+}|- |Y^{\Prime}_{2^{\alpha},-}|&=
 \begin{cases}
    2^{3\alpha-2} & \text{ if } k  \equiv 1 \pmod 4, \\
     -2^{3\alpha-2} &  \text{ if } k  \equiv 3 \pmod 4.
        \end{cases}
\end{align*}
Similar arguments can be applied to deduce the results for $\Delta'\equiv 2\pmod 8$ and  $\Delta'\equiv 6 \pmod 8$.
\end{proof}

With the results of the above lemmas in hand, we now determine $|G_E(L) \cap \Psi^{\mathcal{X}}_{n,k}|$. Let us treat the cyclicity case first. By Lemma \ref{Y-L-2} and \eqref{estimation-num}, we find that $|G_E(L) \cap \Psi^{\Cyc}_{n,k}|$ is equal to
\begin{align*}
    &\frac{1}{2}\left(|\Psi^{\Cyc}_{n,k}(L)| + \prod_{\ell^\alpha \parallel L}(|Y^{\Cyc}_{\ell^\alpha,+}| - |Y^{\Cyc}_{\ell^\alpha,-}| ) \right) \\
    = &\frac{1}{2} \left( \prod_{\substack{\ell^\alpha \parallel (L,n) \\\ell\mid k-1}} \ell^{3(\alpha-1)}(\ell^3-\ell-1) \prod_{\substack{\ell^\alpha \parallel (L,n) \\ \ell\nmid k-1  }}\ell^{3(\alpha-1)} (\ell^3-\ell) \prod_{\substack{\ell \mid L \\ \ell \nmid n}}(|\GL_2(\Z/\ell\Z)|-1) \right. \\
& + 2^{3(\alpha-1)} \tau    \prod_{\substack{\ell^{\alpha}\parallel (n, L^{\odd})\\ \ell\mid k-1}} \ell^{3(\alpha-1)}(\ell^3-\ell-1) \prod_{\substack{\ell^{\alpha}\parallel (n, L^{\odd}) \\ \ell\nmid k-1}}\left(\frac{k}{\ell}\right)\ell^{3(\alpha-1)}(\ell^3-\ell)\left. \prod_{\substack{\ell \mid L^{\odd}\\ \ell \nmid n}} (-1) \right) \\
   = &\frac{1}{2}\prod_{\substack{\ell^\alpha \parallel (L,n)}} \ell^{3(\alpha-1)} \prod_{\substack{\ell \mid (L^{\odd},n) \\ \ell \mid k-1}}(\ell^3-\ell-1) \prod_{\substack{\ell \mid (L^{\odd},n) \\ \ell \nmid k-1}}(\ell^3-\ell) \left( 5\prod_{\substack{\ell \mid L^{\odd} \\ \ell \nmid n}} (|\GL_2(\Z/\ell\Z)|-1) + \tau^{\Cyc} \right).
\end{align*}

Since we are assuming that $m_E \mid L$, $G_E(L)$ must be an index $2$ subgroup of $\GL_2(\Z/L\Z)$. Thus, we have
\begin{equation}\label{sizeofGEL}
    |G_E(L)| = \frac{1}{2} \cdot \prod_{\ell^{\alpha}\parallel L}|\GL_2(\Z/\ell^{\alpha}\Z)|=\frac{1}{2} \cdot \prod_{\ell^\alpha \parallel L}\ell^{4(\alpha-1)}|\GL_2(\Z/\ell\Z)|.
\end{equation}

Along with \lemref{Sizeofdeltacyc} and \propref{cycalmosteuler}, a short computation reveals that
\begin{align*}
   C_{E, n, k}^{\Cyc} & = \frac{|G_E(L) \cap \Psi^{\Cyc}_{n,k}(L)|}{|G_E(L)|} \prod_{\substack{\ell \nmid m_E \\ \ell^\alpha \parallel n}} \delta^{\Cyc}_{E,n,k}(\ell^\alpha) \prod_{\ell \nmid nm_E} \delta^{\Cyc}_{E}(\ell) \\ 
  & =C_{n, k}^{\Cyc}\left(1+ \tau^{\Cyc}\frac{1}{\displaystyle 5\prod_{\substack{\ell \mid L^{\odd} \\ \ell \nmid n}} (|\GL_2(\Z/\ell\Z)|-1)}\right).
\end{align*}

Now we move on to the Koblitz case. By Lemma \ref{Y-L-2}, we  have $|Y^{\Prime}_{2^{\alpha},+}| - |Y^{\Prime}_{2^{\alpha},-}|=-\tau 2^{3\alpha-2}$. 
Hence, by (\ref{estimation-num}), a simple calculation reveals that $|G_E(L) \cap \Psi^{\Prime}_{n,k}(L)|$ equals
\begin{align*}
    &\frac{1}{2}\left(|\Psi^{\Prime}_{n,k}(L)| + \prod_{\ell^\alpha \parallel L}(|Y^{\Prime}_{\ell^\alpha,+}| - |Y^{\Prime}_{\ell^\alpha,-}|) \right) \\
    &= \frac{1}{2}\left(  \prod_{\substack{\ell^{\alpha}\parallel (L, n)\\ \ell\mid k-1}}\ell^{3(\alpha-1)}(\ell^3-\ell^2-\ell)\prod_{\substack{\ell^{\alpha}\parallel (L, n)\\ \ell\nmid k-1}}\ell^{3(\alpha-1)}(\ell^3-\ell^2-2\ell)\prod_{\substack{\ell\mid L\\\ell \nmid n}}\ell(\ell^3-2\ell^2-\ell+3)  \right. \\
    &\left. - 2^{3\alpha-2} \tau 
\prod_{\substack{\ell^{\alpha}\mid (n, L^{\odd})\\ \ell\mid k-1}} \ell^{3(\alpha-1)}(\ell^3-\ell^2-\ell) \prod_{\substack{\ell^{\alpha}\mid (n, L^{\odd})\\ \ell\nmid k-1}}\left(\frac{k}{\ell}\right)\ell^{3(\alpha-1)}(\ell^3-\ell^2-2\ell)  \prod_{\substack{\ell \mid L^{\odd}\\ \ell \nmid n}} \ell 
 \right) \\
    &= \frac{1}{2} \prod_{\ell^\alpha \parallel L}\ell^{3(\alpha-1)} \prod_{\substack{\ell \mid (L,n) \\ \ell \mid k-1}}(\ell^3-\ell^2-\ell) \prod_{\substack{\ell \mid (L,n) \\ \ell \nmid k-1}}(\ell^3-\ell^2-2\ell) \left( \prod_{\substack{\ell \mid L \\ \ell \nmid n}}\ell(\ell^3-2\ell^2-\ell+3) +\tau^{\Prime} \prod_{\substack{\ell \mid L \\ \ell \nmid n}}\ell\right).
\end{align*}

Finally, by \propref{KoblitzAPeulerproduct}, \lemref{deltaprimeellprime}, (\ref{compu2}), and \eqref{sizeofGEL}, we get 
\begin{align*}
   C_{E, n, k}^{\Prime} & =\frac{|G_E(L) \cap \Psi^{\Prime}_{n,k}(L)|}{|G_E(L)|\cdot  \prod_{\ell \mid L}(1-1/\ell)} \prod_{\substack{\ell \nmid m_E \\ \ell^\alpha \parallel n}} \frac{\delta^{\Prime}_{E,n,k}(\ell^\alpha)}{1-1/\ell} \prod_{\ell \nmid nm_E} \frac{\delta^{\Prime}_{E}(\ell)}{1-1/\ell} \\
   & = C_{n, k}^{\Prime}\left(1+ \frac{\tau^{\Prime} \cdot \displaystyle \prod_{\substack{\ell\mid L \\\ell\nmid n}}\ell}{\displaystyle 
 \prod_{\substack{\ell\mid L \\\ell\nmid n}}\ell(\ell^3-2\ell^2-\ell+3)}\right) =  C_{n, k}^{\Prime}\left(1+ \frac{\tau^{\Prime}}{\displaystyle 
 \prod_{\substack{\ell\mid L \\\ell\nmid n}}(\ell^3-2\ell^2-\ell+3)}\right).
\end{align*}
This completes the proof of Theorem \ref{serre-constants}.

\section{On the Koblitz constant for non-Serre curves} \label{KobConNonSeCur}
\subsection{Bounding the Koblitz constant for non-CM, non-Serre curves} \label{BoundKobConNonCM}
In this subsection, we will determine an upper bound for $C^{\Prime}_{E,n,k}$ in the case of non-CM, non-Serre curves. 

Let $E/\Q$ be a non-CM, non-Serre curve, defined by the model \eqref{basicmodel}, of adelic level $m_E$. Let $L$ be as defined in \eqref{E:Def-L}. Then we write  $L = L_1L_2$ such that $L_2$ is the product of prime powers $\ell^{\alpha} \parallel L$ with $\ell \not \in \{2,3,5\}$ and $G_E(\ell) \simeq \GL_2(\Z/\ell\Z)$. By \cite[Appendix, Theorem 1]{MR2118760}, $G_E(L_2) \simeq \GL_2(\Z/L_2\Z)$. Let $\varpi \colon \GL_2(\Z/L\Z) \to \GL_2(\Z/L_2\Z)$ be the natural reduction map. Note that
$$\varpi\left(G_E(L) \cap \Psi^{\Prime}_{n,k}(L)\right) \subseteq G_E(L_2) \cap \Psi^{\Prime}_{n,k}(L_2).$$
Since $\varpi$ is a surjective group homomorphism, we have 
\begin{equation}\label{aux2}
    \begin{split}
        \delta_{E,n,k}^{\Prime}(L) &= \frac{\left|G_E(L) \cap \Psi^{\Prime}_{n,k}(L)\right|}{\left|G_E(L)\right|} \\
&\leq \frac{\left| \varpi^{-1}\left(G_E(L_2) \cap \Psi^{\Prime}_{n,k}(L_2)\right)\right|}{|\varpi^{-1}(G_E(L_2))|}
= \frac{\left| G_E(L_2) \cap \Psi^{\Prime}_{n,k}(L_2)\right|}{|G_E(L_2)|} = \delta_{E,n,k}^{\Prime}(L_2).
    \end{split}
\end{equation}
Since $\rho_{E,L_2}$ is surjective, we  apply the same argument as in the proof of \lemref{deltaprimeellprime} and obtain
$$\frac{\delta_{E,n,k}^{\Prime}(L_2)}{\prod_{\ell \mid L_2}(1-1/\ell)}\leq 1.$$
Before proceeding to bound the constant $C_{E, n, k}^{\Prime}$, we first state a standard analytic result. 
\begin{Lemma}\label{easylemma}
    For any positive integer $M$, we have
    $$\prod_{\ell \mid M} \left(1-\frac{1}{\ell}\right)^{-1} \ll \max\{1, \log \log M\}.$$
\end{Lemma}
\begin{proof}
    Follows from Mertens' theorem \cite[p.\ 53, (15)]{MR1579612}. See \cite[p.\ 767]{MR2805578} for the argument.
\end{proof}

From \lemref{easylemma}, \eqref{Koblitzpart}, \eqref{KobAPproduct},  \lemref{deltaprimeellprime}, and \eqref{aux2}, we obtain
\begin{align}\label{upperboundofCprimeEnk}
\begin{split}
    C^{\Prime}_{E,n,k} &= \frac{\delta_{E,n,k}^{\Prime}(L)}{\prod_{\ell \mid L}(1-1/\ell)} \cdot \prod_{\substack{\ell^\alpha \parallel n \\ \ell \nmid m_E}}\frac{\delta_{E,n,k}^{\Prime}(\ell^\alpha)}{1-1/\ell} \prod_{\ell \nmid nm_E}\frac{\delta_{E,n,k}^{\Prime}(\ell)}{1-1/\ell} \\
    &\leq \frac{1}{\prod_{\ell \mid L_1}(1-1/\ell)} \cdot \frac{\delta_{E,n,k}^{\Prime}(L_2)}{\prod_{\ell \mid L_2}(1-1/\ell)} \leq \prod_{\ell \mid L_1} \frac{1}{1-1/\ell} \ll \max\{1, \log \log \rad(L_1)\}.
\end{split}
\end{align}
Our next task is to bound $\rad (L_1)$ in terms of $a$ and $b$ as in \eqref{basicmodel}. Write $j_E$ to denote the $j$-invariant of $E$ and $h \coloneqq h(j_E)$ for the Weil height of $E$. If $p \mid L_1$, then either $p \leq 5$ or $p$ is an exceptional prime (meaning $G_E(p) \neq \GL_2(\F_p)$). By the main theorem of \cite{MR1209248}, there exist absolute constant $\kappa$ and $\lambda$ for which $\rho_{E,\ell}$ is not exceptional for all $\ell > \kappa(\max{1,h})^\lambda$. Since $\rad(L_1)$ is squarefree, we have
\begin{align}\label{hEab2}
   &  \rad(L_1) \leq 30\prod_{\ell \leq \kappa (\max\{1,h\})^\lambda} \ell \nonumber \\
  &   \implies \log \rad (L_1) \ll  \sum_{\ell \leq \kappa(\max\{1,h\})^\lambda} \log \ell \ll (\max\{1,h\})^\lambda \log \max\{1, h\}.
\end{align}
Since $E$ is given by the model \eqref{basicmodel}, we have that 
\begin{equation}\label{hEab1}
    h=h(j_{E}) \ll \log \max \{|a|^3,|b|^2\}.
\end{equation}
Combining  \eqref{upperboundofCprimeEnk}, \eqref{hEab2}, and \eqref{hEab1}, we obtain the following result.
\begin{Proposition}\label{nonSerre}
    Let $E/\Q$ be a non-CM, non-Serre curve given by \eqref{basicmodel}. Then we have
    $$C^{\Prime}_{E,n,k} \ll \log \log \max\{|a|^3,|b|^2\}.$$
\end{Proposition}

\subsection{Bounding the Koblitz constant for CM curves}\label{BoundKobConCM}
In this subsection,  we focus on CM elliptic curves $E/\Q$. The goal is to show that the constant $C^{\Prime}_{E, n,  k}$ is  bounded independent of the choice of the CM curve (Proposition \ref{CMupperbound}). We keep the notation from \sectionref{GalRepCM}.

Let $E/\Q$ be an elliptic curve with CM by an order $\mathcal{O}$ in an imaginary quadratic field $K$. Let $p$ be a prime of Koblitz reduction for $E/\Q$. Since $[K:\Q] = 2$, the prime $p$ either splits completely, stays inert, or ramifies over  $K/\Q$.

If $p$ does not split over $K/\Q$, then by Deuring's criterion \cite{MR0005125},  $p$ is a supersingular prime for $E$ and we have $a_p(E) = 0$.
Therefore, 
$$|\widetilde{E}_p(\F_p)| =p + 1,$$
which is an even number if $p > 2$. Thus, an odd supersingular prime cannot be a prime of Koblitz reduction for $E$.

On the other hand, if  $p$ splits completely over $K/\Q$. Let $\fp$ be one of the primes in $K$ lying above $p$. Then $\F_\fp\simeq \F_p$, and hence
 the $\widetilde{E}_p$ is isomorphic to $\widetilde{E}_\fp$ as an elliptic curve over the base field. In particular, 
 \[|\widetilde{E}_p(\F_p)| = |\widetilde{E}_\fp(\F_\fp)|.\]
 Thus, we obtain
\begin{equation*}
\begin{split}
      \pi_E^{\Prime}(x;n,k)  &\coloneqq  \#\{p \leq x: p \nmid N_E, |\widetilde{E}_p(\F_p)| \text{ is prime}, p \equiv k \neghs \pmod n\} \\
    &=  \#\{p \leq x : p \nmid N_E, |\widetilde{E}_p(\F_p)| \text{ is prime}, p \text{ splits over }K/\Q, p \equiv k \neghs \pmod n \} + \bigO(1) \\
   &= \frac{1}{2} \#\{\fp:  N_{K/\Q}(\fp) \leq x, N_{K/\Q}(\fp) \nmid N_E, |\widetilde{E}_\fp(\F_\fp)| \text{ is prime}, \\
   & \hspace{1.2cm}
   N_{K/\Q}(\fp) \text{ is a rational prime}, N_{K/\Q}(\fp) \equiv k \neghs \pmod n \} + \bigO(1).
   \end{split}
\end{equation*}
Koblitz's conjecture in arithmetic progressions  for 
 CM elliptic curves can be formulated as follows:
 \begin{Conjecture}\label{CMcaseconjecture}
    Let $E/\Q$ be an elliptic curve with  CM by $\mathcal{O}$ in an imaginary quadratic field $K$. Let $m_E$ be defined as in \lemref{CMpropertyofmE},  $n$ be a positive integer, and $k$ be an integer coprime to $n$.  Then there exists a constant $C^{\Prime}_{E/K,n,k}$  defined in (\ref{KoblitzAPeulerCM}) such that  
    \begin{equation}\label{expectation}
\pi^{\Prime}_{E}(x;n,k) \sim \frac{C^{\Prime}_{E/K,n,k}}{2} \cdot \frac{x}{\log^2 x} \hs \text{ as } x\to \infty.
\end{equation}
If the constant vanishes, we interpret \eqref{expectation} as stating that there are only finitely many primes $p \equiv k \pmod n$ of Koblitz reduction for $E$.
\end{Conjecture}
Comparing with Conjecture  \ref{KoblitzConjAP}, we have  
\begin{equation}\label{relation-CM-constants}
C^{\Prime}_{E,n,k}=\frac{C^{\Prime}_{E/K,n,k}}{2},
\end{equation}
where $C^{\Prime}_{E/K,n,k}$ is defined in (\ref{KoblitzAPoriginalCM}).

We now introduce some notation used to determine the constant $C^{\Prime}_{E/K,n,k}$. For a positive integer $m$, let us fix a $\Z/m\Z$-basis of $\mathcal{O}/m\mathcal{O}$. This allows us to view $\GL_1(\mathcal{O}/m\mathcal{O})=(\mathcal{O}/m\mathcal{O})^{\times}$  a subgroup of $\GL_2(\Z/m\Z)$. Let $\det \colon (\mathcal{O}/m\mathcal{O})^\times \to (\Z/m\Z)^\times$ be the determinant map, defined in the natural way. Fixing a standard orthogonal basis of $\mathcal{O}/m\mathcal{O}$, $N$ is identified with the determinant map. 
Thus, drawing a parallel from \eqref{Psiprimenk}, we are led to define
$$\Psi^{\Prime}_{K,n,k}(m) = \left\{g \in (\mathcal{O}/m\mathcal{O})^\times : \det(g-1) \in (\Z/m \Z)^{\times}, \det g \equiv k \neghs \pmod {\gcd(m,n)}\right\}.$$
Observe that $\rho_{E,m}(\Frob_\fp) \in G_E(m) \cap \Psi^{\Prime}_{K,n,k}(m)$ if and only if $|\widetilde{E}_\fp(\F_\fp)|$ is invertible in $\Z/m\Z$ and $\det(\rho_{E,m}(\Frob_\fp)) \equiv k \pmod {\gcd(m,n)}$. 
Hence, we are led to define
$$\delta_{E/K,n,k}^{\Prime}(m) \coloneqq \frac{\left|G_E(m) \cap \Psi^{\Prime}_{K,n,k}(m)\right|}{\left|G_E(m)\right|}.$$
Drawing a parallel from \eqref{KoblitzconstantAPoriginal}, we set  
\begin{equation}\label{KoblitzAPoriginalCM}
    C_{E/K,n,k}^{\Prime} \coloneqq \lim_{m\to \infty} \frac{\delta_{E/K,n,k}^{\Prime}(m)}{\prod_{\ell \mid m} (1-1/\ell)},
\end{equation}
where the limit is taken over all positive integers ordered by divisibility.

\begin{Lemma}\label{CMdeltaprime}
    Let $E/\Q$ be an elliptic  curve with CM by $\mathcal{O}$ of conductor $f$ in an imaginary quadratic field $K$. We denote the adelic level of $E$ by $m_E$. Let $\chi\coloneqq \chi_K$ be as given in \eqref{Kroneckersymbol}. For each rational prime $\ell \nmid fm_E$ and $\ell\nmid n$, we have
    \[
     \frac{\delta_{E/K,n,k}^{\Prime}(\ell)}{1-1/\ell}= \displaystyle 1 - \chi(\ell) \frac{\ell^2-\ell-1}{(\ell-\chi(\ell))(\ell-1)^2}.
    \]
    For each prime $\ell\nmid f m_E$ and $\ell^{\alpha} \parallel n$, we have 
    \begin{align*}
        \frac{\delta_{E/K,n,k}^{\Prime}(\ell^{\alpha})}{1-1/\ell} = \begin{cases}
            \displaystyle \frac{1}{\phi(\ell^{\alpha})}\left( 1 - \chi(\ell) \frac{1}{(\ell-\chi(\ell))(\ell-1)} \right) & \text{ if } \ell^{\alpha} \parallel n \text{ and } k \equiv 1 \neghs \pmod \ell, \\
            \displaystyle \frac{1}{\phi(\ell^{\alpha})} \left(1 - \chi(\ell) \frac{\ell+1}{(\ell-\chi(\ell))(\ell-1)}\right) & \text{ if } \ell^{\alpha} \parallel n \text{ and } k \not \equiv 1 \neghs \pmod \ell.
        \end{cases}
    \end{align*}
    \end{Lemma}
    \begin{proof}
        First, we consider the case where  $\ell \nmid nfm_E$. By \lemref{CMpropertyofmE}, we have $G_E(\ell)\simeq (\mathcal{O}_K/\ell\mathcal{O}_K)^{\times}$ and the  condition $\det g\equiv k \pmod {\gcd(\ell, n)}$ trivially holds. 
        Hence 
        \[
        G_E(\ell)\cap \Psi_{K,n,k}^{\Prime}(\ell)=\{g\in (\mathcal{O}_K/\ell\mathcal{O}_K)^\times: \det(g-1)\not\equiv 0 \neghs \pmod \ell \}.
        \]
        Therefore, by Corollary \ref{countingprimeCM}, we get 
        \[
         |\Psi_{K,n,k}^{\Prime}(\ell)|=(\ell-2)^2 \; \text{ or } \;
        |\Psi_{K,n,k}^{\Prime}(\ell)|=\ell^2-2 
        \]
       depending on whether $\ell$ splits or is inert in $K$.

        Now we assume $\ell^{\alpha}\parallel n$. Similarly, we have  $G_E(\ell^{\alpha})\simeq (\mathcal{O}_K/\ell^{\alpha}\mathcal{O}_K)^{\times}$ and hence 
        \[
        G_E(\ell^{\alpha})\cap \Psi_{K,n,k}^{\Prime}(\ell^{\alpha})=\Psi_{K,n,k}^{\Prime}(\ell^{\alpha}).
        \]
        Then the  condition $\det g\equiv k \pmod {\gcd(\ell^{\alpha}, n)}$ becomes $\det g\equiv k \pmod {\ell^{\alpha}}$. So we get 
          \[
        \Psi_{K,n,k}^{\Prime}(\ell^{\alpha})=\left\{ g \in (\mathcal{O}_K/\ell^a\mathcal{O}_K)^\times : \det(g-1) \not \equiv 0 \neghs \pmod {\ell}, \det g \equiv k \neghs \pmod {\ell^a}  \right\}.
        \]
       If $k\equiv 1\pmod \ell$, then by \corref{countingprimeCM}, 
        \[
        |\Psi_{K,n,k}^{\Prime}(\ell^{\alpha})|= \ell^{a-1}(\ell-2) \; \text{ or }  \; |\Psi_{K,n,k}^{\Prime}(\ell^{\alpha})|= \ell^{a}
        \]
        depending on whether $\ell$ splits or is inert in $K$. If  $k\not \equiv 1\pmod \ell$, then 
        \[
        |\Psi_{K,n,k}^{\Prime}(\ell^{\alpha})|= \ell^{a-1}(\ell-3) \; \text{ or }  \; |\Psi_{K,n,k}^{\Prime}(\ell^{\alpha})|= \ell^{a-1}(\ell+1),
        \]
        depending on whether $\ell$ splits or is inert in $K$.
    \end{proof}

For a CM elliptic curve $E/\Q$ with CM  by $\mathcal{O}$ of conductor $f$, we set
    \begin{equation}\label{CM-L}
        L \coloneqq \prod_{\ell \mid fm_E} \ell^{\alpha_\ell}, \hs \text{ where } \alpha_\ell = \begin{cases}
        v_\ell(n) & \text{ if } \ell \mid n, \\
        1 & \text{ otherwise}.
    \end{cases}
    \end{equation}
To save notation, we will write $\ell^{\alpha}$ instead of $\ell^{\alpha_\ell}$.
    
\begin{Proposition}\label{KoblitzAPCM}
    Let $E/\Q$ have a CM by $\mathcal{O}$ of conductor $f$ in an imaginary quadratic field $K$.  Let $\chi\coloneqq \chi_K$ be as given in \eqref{Kroneckersymbol}.  Let $m_E$ denote the adelic level of $E$. Let $L$ be defined as in (\ref{CM-L}). Fix a positive integer $n$. 
Then, $\delta^{\Prime}_{E/K,n,k}(\cdot)$, as an arithmetic function, satisfies the following properties:
\begin{enumerate}
    \item Let $L \mid L' \mid L^\infty$. Then, $\delta^{\Prime}_{E/K,n,k}(L) = \delta_{E/K,n,k}^{\Prime}(L')$;
    \item Let $\ell^\alpha$ be a prime power and $d$ be a positive integer with $(\ell, Ld) = 1$. Then, $\delta^{\Prime}_{E/K,n,k}(d\ell^\alpha) = \delta^{\Prime}_{E/K,n,k}(d) \cdot \delta^{\Prime}_{E/K,n,k}(\ell^\alpha)$.
    \item Let $\ell^\alpha \parallel n$ and $(\ell, L) = 1$. Then, for any $\beta \geq \alpha$, $\delta_{E/K,n,k}^{\Prime}(\ell^\beta) = \delta_{E/K,n,k}^{\Prime}(\ell^\alpha)$. Further, if $\ell \nmid nL$, we have $\delta_{E/K,n,k}^{\Prime}(\ell^\beta) = \delta_{E/K,n,k}^{\Prime}(\ell)$.
\end{enumerate}
Therefore, \eqref{KoblitzAPoriginalCM} can be expressed as

\begin{equation}\label{KoblitzAPeulerCM}
    C^{\Prime}_{E/K,n,k} = \frac{\delta_{E/K,n,k}^{\Prime}(L)}{\prod_{\ell \mid L} (1-1/\ell)} \cdot \prod_{\substack{\ell \nmid fm_E \\ \ell^{\alpha} \parallel n}}\frac{\delta_{E/K,n,k}^{\Prime}(\ell^{\alpha})}{1-1/\ell} \cdot \prod_{\substack{ \ell \nmid nfm_E}} \left(1 - \chi(\ell)\frac{\ell^2-\ell-1}{(\ell-\chi(\ell))(\ell-1)^2} \right).
\end{equation}
\end{Proposition}

\begin{proof}
One can prove  (1)-(3) following the same strategy as in the proof of Proposition \ref{KoblitzAPeulerproduct}. One only needs to replace $m_E$ by  $fm_E$ and $\GL_2(\Z/\ell^{\alpha}\Z)$ by $(\mathcal{O}/\ell^{\alpha}\mathcal{O})^\times$. 
Therefore, from these results, we get
\begin{align*}
    C^{\Prime}_{E/K,n,k} 
    = & \frac{\delta_{E/K,n,k}^{\Prime}(L)}{\prod_{\ell \mid L} (1-1/\ell)} \cdot \prod_{\substack{\ell \nmid fm_E \\ \ell^{\alpha} \parallel n}}\frac{\delta_{E/K,n,k}^{\Prime}(\ell^{\alpha})}{1-1/\ell} \cdot \prod_{\substack{\ell \nmid nfm_E}} \frac{\delta_{E/K,n,k}^{\Prime}(\ell)}{1-1/\ell}.
\end{align*}
Now, we see that (\ref{KoblitzAPeulerCM}) follows from Lemma \ref{CMdeltaprime}.
\end{proof}

\begin{Remark}\label{convCM}
    Given that $\ell \nmid nfm_E$, we observe that
\begin{align*}
     \frac{\delta^{\Prime}_{E/K,n,k}(\ell)}{1-1/\ell} &= 1 - \chi_K(\ell) \frac{\ell^2-\ell-1}{(\ell-\chi_K(\ell))(\ell-1)^2} \\
     &= \left( 1 -  \frac{\chi_K(\ell)}{\ell} + \bigO\left(\frac{1}{\ell^2}\right)\right) \\
     &= \left(1 - \frac{\chi_K(\ell)}{\ell}\right)\left(1+\bigO\left(\frac{1}{\ell^2}\right)\right).
\end{align*}
Thus, we have
\begin{equation*}
\prod_{\substack{\ell \nmid fm_En}} \left( 1 - \chi_K(\ell) \frac{\ell^2-\ell-1}{(\ell-\chi_K(\ell))(\ell-1)^2}\right) = \prod_{\substack{ \ell \nmid fm_En}} \left(1-\frac{\chi_K(\ell)}{\ell}\right)\left(1+\bigO\left(\frac{1}{\ell^2}\right)\right).
\end{equation*}
Note that this is a product of an Euler factorization of $L(s,\chi_K)^{-1}$ at $s = 1$ (with some correction factor) and an absolutely convergent product. Since $L(1,\chi_K)$ converges to a non-zero number for a non-principal character $\chi_K$, the infinite product in \eqref{KoblitzAPeulerCM} is conditionally convergent.
\end{Remark}

By \eqref{expectation}, \eqref{relation-CM-constants},  \lemref{CMdeltaprime}, and \propref{KoblitzAPCM} we can explicitly formulate the conjectural Koblitz constant for CM elliptic curves. Let $n = n_1n_2$ where $n_1 \mid (fm_E)^\infty$ and $(n_2,fm_E) = 1$. (In particular, $n_2$ is the product of $\ell^{\alpha}$ for which $\ell^{\alpha} \parallel n$ with $\ell \nmid L$.) We have
\begin{equation}\label{lengthyproduct}
\begin{split}
    C_{E,n,k}^{\Prime} = \frac{1}{2} \cdot \frac{1}{\phi(n_2)} \cdot \frac{\delta_{E/K, n,k}^{\Prime}(L)}{\prod_{\ell \mid L}(1-1/\ell)} & \cdot \prod_{\substack{\ell^{\alpha} \parallel n \\ \ell \nmid L \\ \ell\mid k-1}} \left( 1 - \chi_K(\ell) \frac{\ell}{(\ell-\chi_K(\ell))(\ell-1)}\right) \\
    &\cdot \prod_{\substack{\ell^{\alpha} \parallel n \\ \ell \nmid L \\ \ell\nmid k-1}} \left(1-\chi_K(\ell) \frac{\ell+1}{(\ell-\chi_K(\ell))(\ell-1)}\right) \\
    &\cdot  \prod_{\ell \nmid nfm_E}\left(1-\chi_K(\ell)\frac{\ell^2-\ell-1}{(\ell-\chi_K(\ell))(\ell-1)^2}\right). 
\end{split}
\end{equation}

\begin{Proposition}\label{CMupperbound}
 For any CM elliptic curve  $E/\Q$, we have
    $$C^{\Prime}_{E,n,k} \ll_n 1.$$
\end{Proposition}
\begin{proof}
    Note that the finite product terms in \eqref{lengthyproduct} are all bounded by $1$. By definition, we have $$\delta_{E/K,n,k}^{\Prime}(L) \leq 1,$$ and hence,
    $$\frac{\delta_{E/K,n,k}^{\Prime}(L)}{\prod_{\ell \mid L}(1-1/\ell)} \ll \max\{1, \log \log \rad(fm_E)\} \ll 1,$$
 by \propref{easylemma} and \propref{uniform-CM}. Finally, the infinite product, up to a correction factor depending on $n$, is universally bounded, since there are only finitely many possibilities for $K$.
\end{proof}

\section{Moments} \label{Moments}

The goal of this section is to complete the proof of Theorem \ref{Main}. We begin by setting forth the general strategy. Let $x > 0$ and $A = A(x)$ and $B = B(x)$ be positive real-valued functions such that $A(x) \to \infty$ and $B(x) \to \infty$ as $x \to \infty$. Let $\Eab$ be an elliptic curve given by the model
$$\Eab \colon Y^2= X^3+aX+b,$$
for some $a,b \in \Z$ and $4a^3+27b^2 \neq 0$. Define
$$\mathcal{F} \coloneqq \mathcal{F}(x) = \left\{ \Eab \colon |a| \leq A, |b| \leq B \right\}.$$
Our objective is to compute, for any positive integer $t$, the $t$-th moment
\begin{equation}\label{tmoment}
    \frac{1}{|\mathcal{F}|} \sum_{E \in \mathcal{F}}\left|C^{\mathcal{X}}_{E,n,k} - C^{\mathcal{X}}_{n,k}\right|^t,
\end{equation}
where $\mathcal{X}$ denotes either ``$\Cyc$" or ``$\Prime$". We know that \eqref{tmoment} can be expressed as
\begin{align*}
    \frac{1}{\mathcal{|F|}} \left( \sum_{\substack{E \in \mathcal{F} \\ E \text{ is Serre}}} \left|C^{\mathcal{X}}_{E,n,k} - C^{\mathcal{X}}_{n,k}\right|^t + \sum_{\substack{E \in \mathcal{F} \\ E \text{ is non-CM} \nonumber \\
    E \text{ is non-Serre}}} \left|C^{\mathcal{X}}_{E,n,k} - C^{\mathcal{X}}_{n,k}\right|^t + \sum_{\substack{E \in \mathcal{F} \\ E \text{ is CM}}} \left|C^{\mathcal{X}}_{E,n,k} - C^{\mathcal{X}}_{n,k}\right|^t \right), 
\end{align*}
where ``$E$ is Serre'' indicates that ``$E$ is a Serre curve'', etc. In order to bound \eqref{tmoment}, we are going to bound each of the three sums separately.

For the first sum, recall that we proved explicit formulas for the constants $C^{\mathcal{X}}_{E,n,k}$ for Serre curves in \sectionref{CycKobConSeCur} and found that these constants closely align with their average counterparts $C^{\mathcal{X}}_{n,k}$. For the second and third sums, we will use the fact due to Jones \cite{MR2563740} that non-Serre curves are rare. For the cyclicity case, we will use the fact that $C^{\Cyc}_{E,n,k}$ is bounded above by $1/\phi(n)$, which follows from \eqref{goodidea} (and is sensible, since under GRH, $C^{\Cyc}_{E,n,k}$ describes the density of some subset of the primes congruent to $k$ modulo $n$). However, for the Koblitz case it is not clear that $C^{\Prime}_{E,n,k}$ should be bounded by $1/\phi(n)$, so we will instead employ the bounds of  \propref{nonSerre} and \propref{CMupperbound}.

We first deal with the moments computation for Serre curves.
Let $\Eab/\Q$ be a Serre curve defined by the model 
$$\Eab \colon Y^2 = X^3 + aX + b,$$
of adelic level $m_{\Eab}$. Let $\Delta_{a,b}'$ denote the squarefree part of the discriminant of $\Eab$. Recall that $m_{\Eab}$ is only supported by $2$ and the prime factors of $\Delta'_{a,b}$ (see Proposition \eqref{mEofSerreCurves}). Set 
$$L_{\Eab} = \frac{|\Delta'_{a,b}|}{\gcd(|\Delta'_{a,b}|,n)}.$$ 
By \thmref{serre-constants}, we have
\begin{align*}
    \left|C^{\Cyc}_{\Eab,n,k} - C^{\Cyc}_{n,k}\right| &\leq \frac{1}{5} C^{\Cyc}_{n,k} \prod_{\substack{\ell \mid m_{\Eab} \\ \ell \nmid 2n}} \frac{1}{\ell^4-\ell^3-\ell^2+\ell-1} \ll \frac{1}{\rad(m_{\Eab})^3} \ll \frac{1}{L_{\Eab}^3}, \\
     \left|C^{\Prime}_{\Eab,n,k} - C^{\Prime}_{n,k}\right| &\leq C^{\Prime}_{n,k} \prod_{\substack{\ell \mid m_{\Eab} \\ \ell \nmid 2n}} \frac{1}{\ell^3-2\ell^2-\ell+3} \ll \frac{1}{\rad(m_{\Eab})^2} \ll \frac{1}{L_{\Eab}^2}.
\end{align*}
Let us set $r_{\Cyc} = 3$ and $r_{\Prime} = 2$. Then, we obtain
$$\left|C^{\mathcal{X}}_{\Eab,n,k} - C^{\mathcal{X}}_{n,k}\right| \ll \frac{1}{L_{\Eab}^{r_\mathcal{X}}} = \left(\frac{\gcd\left(|\Delta'_{a,b}|,n\right)}{|\Delta'_{a,b}|}\right)^{r_\mathcal{X}},$$
given that $\Eab/\Q$ is a Serre curve.

Observing that  $|\mathcal{F}| \sim 4AB$ as $x\to \infty$, we have  for any $A,B,Z \geq 2$ and $t\geq 1$,
\begin{align}\label{set0}
    \frac{1}{|\mathcal{F}|} \sum_{\substack{E \in \mathcal{F} \\ E \text{ is Serre}}} \left|C^{\mathcal{X}}_{E,n,k} - C^{\mathcal{X}}_{n,k}\right|^t \ll \frac{1}{AB} \sum_{\substack{|a|\leq A \\ |b| \leq B\\ \Delta'_{a,b} \neq 0 \\ \frac{\left|\Delta'_{a,b}\right|}{\left(\left|\Delta'_{a,b}\right|,n\right)} < Z}} 1 + \frac{1}{AB} \sum_{\substack{|a|\leq A \\ |b| \leq B\\ \Delta'_{a,b} \neq 0 \\ \frac{\left|\Delta'_{a,b}\right|}{\left(\left|\Delta'_{a,b}\right|,n\right)} \geq Z}} \frac{1}{Z^{r_\mathcal{X}t}}.
\end{align}
\begin{Lemma}\label{Joneslemma}
    With the notation above, we have
    $$\sum_{\substack{|a| \leq A \\ |b| \leq B \\ \Delta'_{a,b} \neq 0 \\ |\Delta'_{a,b}| < Z}} 1 \ll n\log B \cdot A \cdot \log^7 A \cdot Z + B.$$
\end{Lemma}
\begin{proof}
    It follows similarly to the argument given in \cite[Section 4.2]{MR2534114}.
\end{proof}
Let $Z = \left(B/n \log B \log^7A \right)^{1/(r_\mathcal{X}t+1)}$. By \eqref{set0} and \lemref{Joneslemma}, we see that
\begin{equation}\label{Serrecase}
    \begin{split}
            \frac{1}{|\mathcal{F}|} \sum_{\substack{E \in \mathcal{F} \\ E \text{ is Serre}}} \left|C^{\mathcal{X}}_{E,n,k} - C^{\mathcal{X}}_{n,k}\right|^t \ll \left(\frac{1}{A} + \frac{n Z \log B \log^7 A}{B}\right) + \frac{1}{Z^{r_\mathcal{X}t}} \ll \left(\frac{n \log B \log^7A}{B} \right)^{\frac{r_\mathcal{X}t}{r_\mathcal{X}t+1}}.
    \end{split}
\end{equation}
By \cite[Theorem 25]{MR2534114} and \eqref{Serrecase}, there exists $\gamma > 0$ such that for any positive integer $t$,
$$\frac{1}{\mathcal{|F|}} \sum_{E \in \mathcal{F}}\left|C^{\Cyc}_{E,n,k}-C^{\Cyc}_{n,k}\right|^t \ll_t \max \left\{\left(\frac{n\log B \log^7A}{B}\right)^{\frac{3t}{3t+1}}, \frac{\log^\gamma(\min \{A,B\})}{\sqrt{\min\{A,B\}}}  \right\}.$$
This completes the proof for the cyclicity case. 

For primes of Koblitz reduction, by \propref{nonSerre}, \propref{CMupperbound}, and \cite[Theorem 25]{MR2534114}, there exists $\gamma > 0$ such that for any positive integer $t$,
\begin{align*}
    \frac{1}{|\mathcal{F}|}  \sum_{\substack{E \in \mathcal{F} \\ E \text{ is non-CM } \\ E \text{ is non-Serre}}} \left|C_{E,n,k}^{\Prime} - C_{n,k}^{\Prime}\right|^t &\ll_t \log \log (\max\{A^3,B^2\})^t \frac{\log^\gamma(\min\{A,B\})}{\sqrt{\min\{A,B\}}},\\
    \frac{1}{|\mathcal{F}|} \sum_{\substack{E \in \mathcal{F} \\ E \text{ is CM}}} \left| C^{\Prime}_{E,n,k} - C^{\Prime}_{n,k} \right|^t &\ll_{t,n} \frac{\log^\gamma(\min\{A,B\})}{\sqrt{\min\{A,B\}}}.
\end{align*}
Therefore, we obtain the inequality claimed in the statement of \thmref{Main}.

\section{Numerical examples} \label{NumEx}

\subsection{Example 1} Let $E$ be the elliptic curve with LMFDB \cite{lmfdb} label \texttt{1728.w1}, which is given by 
\[ E \colon y^2 = x^3 + 6x - 2. \]
From the curve's LMFDB page, we note that it is a Serre curve with adelic level $m_E = 6$. Zywina \cite[Section 5]{MR2805578} computed the Koblitz constant of $E$,
\[
C_E^{\Prime} \approx 0.561296.
\]
Running either our Magma functions \texttt{KoblitzAP} or  \texttt{SerreCurveKoblitzAP} \cite{LMW-GitHub} on $E$ with modulus $n = 6$, we find that
\[
C_{E,6,1}^{\Prime} = C_E^{\Prime}
\quad \text{and} \quad
C_{E,6,5}^{\Prime} = 0.
\]
This result can be verified ``manually'' by studying the mod $6$ Galois image of $E$, as we now discuss. 

The mod $6$ Galois image $G_E(6)$ is the index $2$ subgroup of $\GL_2(\Z/6\Z)$ generated by
\[
G_E(6) = 
\left\langle
\begin{pmatrix}
    1 & 1 \\ 0 & 5
\end{pmatrix},
\begin{pmatrix}
    1 & 0 \\ 5 & 5
\end{pmatrix},
\begin{pmatrix}
    5 & 0 \\ 5 & 1
\end{pmatrix}
\right\rangle.
\]
From this description, we compute that
\[
 \left\{\tr(M) : M \in G_E(6) \text{ and } \det(M) \equiv 5 \neghs \pmod{6} \right\} = \left\{ 0, 2, 4 \right\}.
\]
Thus, if $p$ is a good prime for $E$ that is congruent to $5$ modulo $6$, then
\[
|\widetilde{E}_p(\F_p)| \equiv p + 1 - \tr(\rho_{E,6}(\Frob_p)) \equiv 1 + 1 - 0 \equiv 0 \pmod{2}.
\]
Hence $|\widetilde{E}_p(\F_p)|$ is even for all good primes $p$ congruent to $5$ modulo $6$. By Hasse's bound and computing a few values of $|\widetilde{E}_p(\F_p)|$, we find that $|\widetilde{E}_p(\F_p)|$ is never $2$ for such a prime $p$. Thus, the only good primes $p$ for which $|\widetilde{E}_p(\F_p)|$  is prime are congruent to $1$ modulo $6$.

\subsection{Example 2} Let $E$ be the elliptic curve with LMFDB label \texttt{200.e1}, which is given by 
\[
E \colon y^2=x^3+5x-10.
\]
From this curve's LMFDB page, we learn that $E$ has adelic index $2$ (i.e., $E$ is a Serre curve) and adelic level $m_E = 8$. Running our Magma function \texttt{SerreCurveKoblitzAP} on $E$ with $n = 8$, we find that
\begin{align*}
    C_{E,8,1}^{\Prime} = C_{E,8,3}^{\Prime} = \tfrac{1}{2} C_E^{\Prime} \quad \text{and} \quad C_{E,8,5}^{\Prime} = C_{E,8,7}^{\Prime} = 0
\end{align*}
where
\[
C_E^{\Prime} \approx 0.505166.
\]
Running our Magma function \texttt{SerreCurveCyclicityAP} on $E$ with $n = 8$, we find that
\begin{align*}
    C_{E,8,1}^{\Cyc} = C_{E,8,3}^{\Cyc} = \tfrac{1}{5} C_E^{\Cyc} \quad \text{and} \quad C_{E,8,5}^{\Cyc} = C_{E,8,7}^{\Cyc} = \tfrac{3}{10} C_E^{\Cyc}.
\end{align*}
where
\[
C_E^{\Cyc} \approx 0.813752.
\]

The values obtained above align well with numerical data for the curve. Among all primes of Koblitz reduction for $E$ up to $10^7$,  $11114$ are congruent to $1$ modulo $8$ and $11259$ are congruent to $3$ modulo $8$; none are congruent to $5$ or $7$ modulo $8$. Among all primes of cyclic reduction for $E$ up to $10^7$, $108096$ are congruent to $1$ modulo $8$, $108251$ are congruent to $3$ modulo $8$, $162234$ are congruent to $5$ modulo $8$, and $162286$ are congruent to $7$ modulo $8$.

\subsection{Example 3}

Let $E$ be the elliptic curve with LMFDB label \texttt{864.a1}, which is given by 
\[ E \colon y^2=x^3-216x-1296. \]
This curve does not have complex multiplication and is not a Serre curve. Its adelic index is $24$ and adelic level is $m_E = 12$. Running our Magma function \texttt{KoblitzAP} on $E$ with $n = 12$, we find that
\begin{align*}
    C_{E,12,1}^{\Prime} = \tfrac{3}{7} C_E^{\Prime}, \quad
    C_{E,12,5}^{\Prime} = 0, \quad
    C_{E,12,7}^{\Prime} = \tfrac{4}{7} C_E^{\Prime}, \quad
    C_{E,12,11}^{\Prime} = 0
\end{align*}
where
\[
C_E^{\Prime} \approx 0.785814.
\]
Running our Magma function \texttt{CyclicityAP} on $E$ with $n = 12$, we find that
\begin{align*}
    C_{E,12,1}^{\Cyc} = \tfrac{3}{19} C_E^{\Cyc}, \quad
    C_{E,12,5}^{\Cyc} = \tfrac{6}{19} C_E^{\Cyc}, \quad
    C_{E,12,7}^{\Cyc} = \tfrac{4}{19} C_E^{\Cyc}, \quad
    C_{E,12,11}^{\Cyc} = \tfrac{6}{19} C_E^{\Cyc}.
\end{align*}
where
\[
C_E^{\Cyc} \approx 0.789512.
\]
As with the previous example, these values agree well with the numerical data for the curve, which is available in our GitHub repository \cite{LMW-GitHub}.

\subsection{Example 4} Let $n=6$ and $E$ be the CM elliptic curve with LMFDB label \texttt{432.d1} defined by 
\begin{equation}\label{CM-model}
y^2=x^3-4.
\end{equation}
We keep the notation from \sectionref{GalRepCM}. From the LMFDB, we know that 
\begin{enumerate}
\item $E$ has CM by the maximal order $\mathcal{O} \coloneqq \Z\left[\frac{1+\sqrt{-3}}{2}\right]$ and the CM field $K = \Q(\sqrt{-3})$. 
\item $E$ has  discriminant $\Delta_E=- 2^{8}  3^{3}$. So  2 and 3 are the only primes of bad reduction for $E$.
\item The map
\[
 \overline{\rho}_{E, \ell} : \Gal(\overline{K}/K) \longrightarrow \left(\mathcal{O}/\ell\mathcal{O}\right)^{\times}
\]
is surjective for all primes $\ell$.
\end{enumerate}
Invoking the proof of \cite[Proposition 2.7]{MR2805578}, we see that $m_E$ is only supported by $2$ and $3$. Further,
\[
 f=1, \; L= n_1=n=6, \; n_2=1.
\]
Therefore, by (\ref{lengthyproduct}), 
\begin{equation*}
\begin{split}
    C_{E,6,k}^{\Prime} =   \frac{3}{2} \frac{|G_E(6)\cap \Psi^{\Prime}_{K, 6, k}(6)|}{|G_E(6)|}
    \cdot  \prod_{\ell \nmid 6}\left(1-\chi_K(\ell)\frac{\ell^2-\ell-1}{(\ell-\chi_K(\ell))(\ell-1)^2}\right). 
\end{split}
\end{equation*}
By adapting Sutherland's \texttt{Galrep} code \cite{Su2016}, we compute $G_E(6)$ in Magma and find that
\[
\left|\Psi^{\Prime}_{K,6,k}(6)\cap G_E(6) \right|
=\begin{cases}
    2 & \text{ if }k \equiv 1 \neghs\pmod{6}, \\
    0 & \text{ if } k \equiv 5 \neghs\pmod{6}.
\end{cases}
\]
Thus, we conclude that
\[ C_{E,6,1}^{\Prime} = C_{E}^{\Prime} \quad \text{and} \quad C_{E,6,5}^{\Prime} =0 \]
where
\[
C_{E}^{\Prime} = \frac{1}{2}
    \cdot  \prod_{\ell \nmid 6}\left(1-\chi_K(\ell)\frac{\ell^2-\ell-1}{(\ell-\chi_K(\ell))(\ell-1)^2}\right) \approx 0.505448.
\]

In fact, we can verify that $C^{\Prime}_{E,6,5} = 0$ using Deuring's criterion. Let $p$ be a rational prime such that  $p \equiv 5 \pmod 6$. Then $p$ is inert in the CM field, $\Q(\sqrt{-3})$. By Deuring's criterion $p$ is supersingular, and hence $|\widetilde{E}_p(\F_p)| = p+1$. Since $p$ is an odd prime, we see that $p$ cannot be a prime of Koblitz reduction for $E$.

\bibliographystyle{amsplain}
\bibliography{References}

\end{document}